\documentclass[11pt]{article}

\usepackage[english]{babel}
\usepackage[utf8]{inputenc}
\usepackage[T1]{fontenc}
\usepackage{amssymb,amsmath,amsthm}
\usepackage[all]{xy}
\usepackage{color}
\usepackage{hyperref}
\usepackage{caption,subcaption}
\usepackage{tikz}
\usetikzlibrary{patterns,arrows,decorations.pathreplacing,shapes,automata,calc}

\usepackage{authblk}

\usepackage{a4wide}
\usepackage{todonotes}
\usepackage{lineno}

\theoremstyle{plain}
\newtheorem{theorem}{Theorem}[section]
\newtheorem{lemma}[theorem]{Lemma}
\newtheorem{proposition}[theorem]{Proposition}
\newtheorem{corollary}[theorem]{Corollary}
\theoremstyle{definition}

\theoremstyle{remark}
\newtheorem{remark}[theorem]{Remark}
\newtheorem{example}[theorem]{Example}

\numberwithin{equation}{section}

\def\N{\mathbb{N}}

\def\id{{\rm id}}

\def\bsigma{{\boldsymbol{\sigma}}}
\def\IET{\mathrm{IET}}

\def\cA{\mathcal{A}}
\def\cB{\mathcal{B}}
\def\cC{\mathcal{C}}
\def\cD{\mathcal{D}}
\def\cE{\mathcal{E}}
\def\cG{\mathcal{G}}
\def\cL{\mathcal{L}}
\def\cP{\mathcal{P}}
\def\cR{\mathcal{R}}
\def\cS{\mathcal{S}}
\def\cT{\mathcal{T}}

\def\cSD{\mathcal{S}_3}
\def\cSL{\mathcal{S}_{\rm LI}}
\def\cSR{\mathcal{S}_{\rm RI}}

\def\NN{{\mathbb{N}}}
\def\ZZ{{\mathbb{Z}}}

\newcommand{\cX}[1]{\mathcal{X}_{[3,#1]}}

\DeclareMathOperator{\Fac}{Fac}

\DeclareMathOperator{\DP}{DP}
\DeclareMathOperator{\DPP}{DPP}

\definecolor{light-gray}{gray}{0.7}
\DeclareMathOperator{\A0}{\mathcal{A}_0}

\tikzset{auto shift/.style={auto=right,->,
to path={ let \p1=(\tikztostart),\p2=(\tikztotarget),
\n1={atan2(\y2-\y1,\x2-\x1)},\n2={\n1+180}
in ($(\tikztostart.{\n1})!1mm!270:(\tikztotarget.{\n2})$) -- 
($(\tikztotarget.{\n2})!1mm!90:(\tikztostart.{\n1})$) \tikztonodes}}}

\makeatletter
\newcommand{\subjclass}[2][2010]{%
  \let\@oldtitle\@title%
  \gdef\@title{\@oldtitle\footnotetext{#1 \emph{Mathematics subject classification.} #2}}%
}
\newcommand{\keywords}[1]{%
  \let\@@oldtitle\@title%
  \gdef\@title{\@@oldtitle\footnotetext{\emph{Key words.} #1.}}%
}
\makeatother

\title{$\cS$-adic characterization of minimal ternary dendric shifts}

\author[1,2]{France Gheeraert}
\author[1,2]{Marie Lejeune}
\author[1]{Julien Leroy}
\affil[1]{Department of Mathematics, University of Liège, Allée de la Découverte 12 (B37), B-4000 Liège, Belgium.
\url{{france.gheeraert,m.lejeune,j.leroy} @uliege.be}}
\affil[2]{Supported by a FNRS fellowship}



\date{}
\begin{document}

\maketitle

\begin{abstract}
Dendric shifts are defined by combinatorial restrictions of the extensions of the words in their languages. 
This family generalizes well-known families of shifts such as Sturmian shifts, Arnoux-Rauzy shifts and codings of interval exchange transformations.
It is known that any minimal dendric shift has a primitive $\mathcal{S}$-adic representation where the morphisms in $\mathcal{S}$ are positive tame automorphisms of the free group generated by the alphabet.
In this paper we investigate those $\mathcal{S}$-adic representations, heading towards an $\mathcal{S}$-adic characterization of this family.
We obtain such a characterization in the ternary case, involving a directed graph with 2 vertices.
\end{abstract}

%########################
%########################
\section{Introduction}
%########################
%########################

Dendric shifts are defined in terms of extension graphs that describe the left and right extensions of their factors. 
Extension graphs are bipartite graphs that can roughly be described as follows: if $u$ is a word in the language $\cL(X)$ of the shift space $X$, one puts an edge between the left and right copies of letters $a$ and $b$ such that $aub$ is in $\cL(X)$. 
A shift space is then said to be dendric if the extension graph of every word of its language is a tree.
These shift spaces were initially defined through their languages under the name of tree sets~\cite{acyclic} and were studied in a series of papers.
They generalize classical families of shift spaces such as Sturmian shifts~\cite{morse_hedlund}, Arnoux-Rauzy shifts~\cite{arnoux_rauzy}, codings of regular interval exchange transformations~\cite{oseledec,arnold} (IET) or else shift spaces arising from the application of the Cassaigne multidimensional continued fraction algorithm~\cite{Cassaigne_Labbe_Leroy_WORDS} (MCF).

Minimal dendric shifts exhibit striking algebraic~\cite{acyclic,finite_index}, combinatorial~\cite{bifix_decoding,rigidity}, 
and ergodic properties~\cite{dimension_group}. 
They for instance have factor complexity $\#(\cL(X) \cap \cA^n) = (\#\cA-1)n+1$~\cite{acyclic} and topological rank $\#\cA$~\cite{dimension_group}, where $\cA$ is the alphabet of the shift space. 
They also fall into the class of shift spaces satisfying the regular bispecial condition~\cite{damron_fickenscher:2020}, which implies that the number of their ergodic measures is at most $\#\cA/2$.
An important property for our work is that the derived shift of a minimal dendric shift is again a minimal dendric shift on the same alphabet, where derivation is here understood as derivation by return words (see Section~\ref{section:s-adicity} for definitions).
This allows to give $\cS$-adic representations of such shift spaces~\cite{Ferenczi:1996}, i.e., to define a set $\cS$ of endomorphisms of the free monoid $\cA^*$ and a sequence $\bsigma = (\sigma_n)_{n \geq 1} \in \cS^\NN$, called an $\cS$-adic representation, such that 
\[
    X = \{x \in \cA^\ZZ \mid u \in \cL(x) \Rightarrow \exists n\in \NN, a\in \cA: u \in \cL(\sigma_1 \sigma_2 \cdots \sigma_n(a))\}.
\]

$\cS$-adic representations are a classical tool that allows to study several properties of shift spaces such as factor complexity~\cite{Durand_Leroy_Richomme,donoso_durand_maass_petite}, the number of ergodic measures~\cite{berthe_delecroix,bedaride_hilion_lustig_1,bedaride_hilion_lustig_2}, the dimension group and topological rank~\cite{dimension_group} or yet the automorphism group~\cite{espinoza_maass}. 
In the case of minimal dendric shifts, the involved endomorphisms are particular tame automorphisms of the free group generated by the alphabet~\cite{bifix_decoding,rigidity}. 
This in particular allows to prove that minimal dendric shifts have topological rank equal to the cardinality of the alphabet and that ergodic measures are completely determined by the measures of the letter cylinders~\cite{dimension_group,bedaride_hilion_lustig_2}.

An important open problem concerning $\cS$-adic representations is the $\cS$-adic conjecture whose goal is to give an $\cS$-adic characterization of shift spaces with at most linear complexity~\cite{leroy_improvements}, i.e., to find a stronger notion of $\cS$-adicity such that a shift space has an at most linear factor complexity if and only if it is ``strongly $\cS$-adic''.
Our work goes one step further towards this conjecture by studying $\cS$-adic representations of minimal dendric shifts.
Our main result is the following that gives an $\Sigma_3\cSD\Sigma_3$-adic characterization of minimal dendric shifts over a ternary alphabet, where $\cSD$ is defined in Section~\ref{subsection:dendric morphisms Ster} and $\Sigma_3$ is the symmetric group.
It involves a labeled directed graph $\cG$ with 2 vertices and which is non-deterministic, i.e., a given morphism may label several edges leaving a given vertex.

\begin{theorem}
\label{thm:main}
A shift space $(X,S)$ is a minimal dendric shift over $\cA_3 = \{1,2,3\}$ if and only if it has a primitive $\Sigma_3\cSD\Sigma_3$-adic representation $\bsigma \in (\Sigma_3\cSD\Sigma_3)^\NN$ that labels a path in the graph $\cG$ represented in Figure~\ref{fig:graph of graphs}.
\end{theorem}

We then characterize, within this graph, the well-known families of Arnoux-Rauzy shifts and of coding of regular 3-IET (Theorem~\ref{thm:representation interval exchange}). 
We also show that shift spaces arising from the Cassaigne MCF are never Arnoux-Rauzy shifts, nor codings of regular 3-IET (Proposition~\ref{prop:cassaigne not iet}).
Observe that minimal ternary dendric shifts have factor complexity $2n+1$ and another $S$-adic characterization could be deduced from~\cite{leroy_2n}.
This other $S$-adic characterization would also involve a labeled graph, but with 9 vertices.

Observe that we do not focus only on the ternary case. 
We investigate the $\cS$-adic representations of minimal dendric shifts over any alphabet obtained when considering derivation by return words to letters.
We for instance show that when taking the image $Y$ of a shift space $X$ under a morphism in $\cS$, the extension graphs of long enough factors of $Y$ are the image of the extension graph of factors of $X$ under some graph homomorphism (Proposition~\ref{prop:morphic image of extension graph}).
This allows us to introduce the notion of dendric preserving morphism for $X$ which is the fundamental notion for the construction of the graph $\cG$.
We also characterize the morphisms $\sigma$ of $\cS$ that are dendric preserving for all $X$ using Arnoux-Rauzy morphisms (Proposition~\ref{prop:characterization dendric preserving}).

The paper is organized as follows.
We start by giving, in Section~\ref{section:preliminaries}, the basic definitions for the study of shift spaces. We introduce the notion of extension graph of a word, of dendric shift and of $\cS$-adic representation of a shift space.

In Section~\ref{section:s-adicity}, we recall the existence of an $\cS$-adic representation using return words for minimal shift spaces (Theorem~\ref{thm:S-adic representation of minimal}) and the link between return words and Rauzy graph.

In Section~\ref{section:bispecial factors}, we then study the relation between words in a shift space and in its image by a strongly left proper morphism (Proposition~\ref{prop:def antecedent and bsp ext image}). We deduce from it a link between the extension graphs (Proposition~\ref{prop:morphic image of extension graph}) using graph morphisms and we prove that the injective and strongly left proper morphisms that preserve dendricity can be characterized using Arnoux-Rauzy morphisms (Proposition~\ref{prop:characterization dendric preserving}).

In Section~\ref{section:ternary case}, we study the notions and results of Section~\ref{section:bispecial factors} in the case of a ternary alphabet.
We then prove the main result of this paper (Theorem~\ref{thm:main}) which gives an $\cS$-adic characterization of ternary minimal dendric shifts using infinite paths in a graph.

Finally, in Section~\ref{section:sub families}, we focus on three sub-families of dendric shifts: Arnoux-Rauzy shifts, interval exchanges and Cassaigne shifts. For interval exchanges, we first recall the associated definitions and basic properties, then provide an $\cS$-adic characterization (Theorem~\ref{thm:representation interval exchange}) in the ternary case using a subgraph of the graph obtained in the dendric case. 
We also prove that the families of Cassaigne shifts, of Arnoux-Rauzy shifts and of regular interval exchanges are disjoint.

%########################
%########################
\section{Preliminaries}
\label{section:preliminaries}
%########################
%########################

%########################
\subsection{Words, languages and shift spaces}
%########################
Let $\cA$ be a finite alphabet of cardinality $d \geq 2$.
Let us denote by $\varepsilon$ the empty word of the free monoid $\cA^*$ (endowed with concatenation), and by $\cA^{\ZZ}$ the set of bi-infinite words over $\cA$.
For a word  $w= w_{1} \cdots w_{\ell} \in \cA^\ell$,
its  {\em length} is denoted $|w|$ and equals $\ell$.
We say that a word $u$ is a {\em factor} of a word $w$ if there exist words $p,s$ such that $w = pus$.
If $p = \varepsilon$ (resp., $s = \varepsilon$) we say that $u$ is a {\em prefix} (resp., {\em suffix}) of $w$.  
For a word  $u \in  \cA^{*}$,  an index $ 1 \le j \le \ell$ such that $w_{j}\cdots w_{j+|u|-1} =u$ is called an {\em occurrence} of $u$ in $w$ and we use the same term for bi-infinite word in $\cA^{\ZZ}$. 
The number of occurrences of  a  word $u \in \cA^*$ in a  finite word $w$ is denoted as $|w|_u$.

The set  $\cA^{\ZZ}$ endowed with the product topology of the discrete topology on each copy of $\cA$ is topologically a Cantor set. 
The {\em shift map}  $S$ defined by $S \left( (x_n)_{n \in \mathbb{Z}} \right) = (x_{n+1})_{n \in \mathbb{Z}}$ is a homeomorphism of  $\cA^{\ZZ}$. 
A {\em shift space} is a pair $(X,S)$ where $X$ is a closed shift-invariant  subset of some $\cA^{\ZZ}$.
It is thus a {\em topological dynamical system}.
It is {\em minimal} if the only closed shift-invariant subset $Y \subset X$ are $\emptyset$ and $X$.
Equivalently, $(X,S)$ is minimal if and only if the orbit of every $x \in X$ is dense in $X$.
Usually we say that the set $X$ is itself a shift space.

The {\em language} of a sequence $x \in \cA^{\ZZ}$ is its set of factors and is denoted $\cL(x)$. 
For a shift space $X$, its {\em language} $\cL(X)$ is $\cup_{x\in X} \cL(x)$ and we set $\cL_n(X) = \cL(X) \cap \cA^n$, $n \in \NN$. 
Its {\em factor complexity} is the function $p_X:\NN \to \NN$ defined by $p_X(n) = \#\cL_n(X)$.
We say that a shift space $X$ is {\em over} $\cA$ if $\cL_1(X) = \cA$.

%########################
\subsection{Extension graphs and dendric shifts}
%########################

Dendric shifts are  defined with respect to combinatorial properties of their language expressed in terms of extension graphs. 
Let $F$ be a set of finite words on the alphabet $\cA$ which is factorial, i.e., if $u \in F$ and $v$ is a factor of $u$, then $v \in F$.
For $w \in F$, we define the sets of left, right and bi-extensions of $w$ by
\begin{align*}
	E_F^-(w) &= \{ a \in \cA \mid aw \in F\};	\\
	E_F^+(w) &= \{ b \in \cA \mid wb \in F\};	\\
	E_F(w) &= \{ (a,b) \in \cA \times \cA \mid awb \in F\}.		
\end{align*}
The elements of $E_F^-(w)$, $E_F^+(w)$ and $E_F^-(w)$ are respectively called the {\em left extensions}, the {\em right extensions} and the {\em bi-extensions} of $w$ in $F$. If $X$ is a shift space over $\cA$, we will use the terminology extensions in $X$ instead of extensions in $\cL(X)$ and the index $\cL(X)$ will be replaced by $X$ or even omitted if the context is clear.
Observe that as $X \subset \cA^\ZZ$, the set $E_X(w)$ completely determines $E_X^-(w)$ and $E_X^+(w)$.  
A word $w$ is said to be {\em right special} (resp., {\em left special}) if $\#(E^+(w))\geq 2$ (resp., $\#(E^-(w)) \geq 2$). 
It is {\em bispecial} if it is both left and right special.
The factor complexity of a shift space is completely governed by the 
extensions of its special factors.
In particular, we have the following result.

\begin{proposition}[Cassaigne and Nicolas~\cite{CANT_cassaigne}]
\label{prop:complexity}
Let $X$ be a shift space. 
For all $n$, we have 
\begin{align*}
	p_X(n+1)-p_X(n) 
	&= \sum_{w \in \cL_n(X)} (\#E^+(w)-1)	\\
	&= \sum_{w \in \cL_n(X)} (\#E^-(w)-1)
\end{align*}
In addition, if for every bispecial factor $w \in \cL(X)$, one has 
\begin{equation}
\label{eq:bilateral multiplicity}
	\#E(w) - \# E^-(w) - \# E^+(w) + 1 = 0,
\end{equation}
then $p_X(n) = (p_X(1)-1)n +1$ for every $n$. 
\end{proposition}

A classical family of bispecial factors satisfying Equation~\eqref{eq:bilateral multiplicity} is made of the {\em ordinary} bispecial factors that are defined by $E(w) \subset (\{a\} \times \cA) \cup (\cA \times \{b\})$ for some $(a,b) \in E(w)$.
A larger family of bispecial factors also satisfying Equation~\eqref{eq:bilateral multiplicity} are the dendric bispecial factors defined below.

For a word $w \in F$, we consider the undirected bipartite graph $\cE_F(w)$ called its \emph{extension graph} with respect to $F$ and defined as follows:
its set of vertices is the disjoint union of $E_F^-(w)$ and $E_F^+(w)$ and its edges are the pairs $(a,b) \in E_F^-(w) \times E_F^+(w)$ such that $awb \in F$.
For an illustration, see Example~\ref{ex:fibo} below.
We say that $w$ is {\em dendric} if $\cE(w)$ is a tree. 
We then say that a shift space $X$ is a \emph{dendric shift} if all its factors are dendric in $\cL(X)$.
Note that every non-bispecial word and every ordinary bispecial word is trivially dendric.
In particular, the Arnoux-Rauzy shift spaces are dendric (recall that Arnoux-Rauzy shift spaces are the minimal shift spaces having exactly one left-special factor $u_n$ and one right-special factor $v_n$ of each length $n$ and such that $E^-(u_n) = \cA = E^+(v_n)$; all bispecial factors of an Arnoux-Rauzy shift are ordinary).
By Proposition~\ref{prop:complexity}, we deduce that any dendric shift has factor complexity $p_X(n) = (p_X(1)-1)n +1$ for every $n$.

\begin{example}\label{ex:fibo}
\rm
Let $\sigma$ be the Fibonacci substitution defined over the alphabet  $\{0,1\}$ by  $\sigma \colon 0 \mapsto 01, 1 \mapsto 0$  and consider the  shift space generated by $\sigma$ (i.e.,  the set of  bi-infinite words   over $\{0,1\}$ whose factors are factors of some $\sigma^n (0)$).
The extension graphs of the empty word and of the two letters $0$ and $1$ are represented in Figure~\ref{fig:fibo-ext}.

\begin{figure}[h]
 \tikzset{node/.style={circle,draw,minimum size=0.5cm,inner sep=0pt}}
 \tikzset{title/.style={minimum size=0.5cm,inner sep=0pt}}

 \begin{center}
  \begin{tikzpicture}
   \node[title](ee) {$\cE(\varepsilon)$};
   \node[node](eal) [below left= 0.5cm and 0.6cm of ee] {$0$};
   \node[node](ebl) [below= 0.7cm of eal] {$1$};
   \node[node](ear) [right= 1.5cm of eal] {$0$};
   \node[node](ebr) [below= 0.7cm of ear] {$1$};
   \path[draw,thick]
    (eal) edge node {} (ear)
    (eal) edge node {} (ebr)
    (ebl) edge node {} (ear);
   \node[title](ea) [right = 3cm of ee] {$\cE(0)$};
   \node[node](aal) [below left= 0.5cm and 0.6cm of ea] {$0$};
   \node[node](abl) [below= 0.7cm of aal] {$1$};
   \node[node](aar) [right= 1.5cm of aal] {$0$};
   \node[node](abr) [below= 0.7cm of aar] {$1$};
   \path[draw,thick]
    (aal) edge node {} (abr)
    (abl) edge node {} (aar)
    (abl) edge node {} (abr);
   \node[title](eb) [right = 3cm of ea] {$\cE(1)$};
   \node[node](bal) [below left= 0.5cm and 0.6cm of eb] {$0$};
   \node[node](bar) [right= 1.5cm of bal] {$0$};
   \path[draw,thick]
    (bal) edge node {} (bar);
  \end{tikzpicture}
 \end{center}

 \caption{The extension graphs of $\varepsilon$ (on the left), $0$ (in the center) and $1$ (on the right) are trees.}
 \label{fig:fibo-ext}
\end{figure}
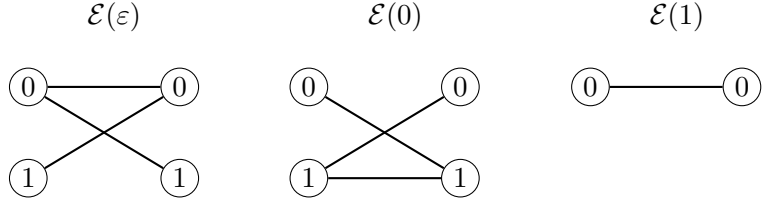
\end{example}

%########################
\subsection{$\cS$-adicity}
\label{subsection:s-adicity}
%########################

Let $\cA, \cB$ be finite alphabets with cardinality at least 2. 
By a {\em morphism} $\sigma:\cA^* \to \cB^*$, we mean a \emph{non-erasing} monoid homomorphism (also called a \emph{substitution} when $\cA = \cB$). 
By non-erasing, we mean that the image of any letter is a non-empty word. 
We stress the fact that all morphisms are assumed to be non-erasing in the following.
Using concatenation, we extend $\sigma$ to~$\cA^\mathbb{N}$ and~$\cA^\mathbb{Z}$. 
In particular, if $X$ is a shift space over $\cA$, the {\em image of $X$ under $\sigma$} is the shift space 
\[
Y = \{S^k \sigma(x) \mid x \in X, 0\leq k < |\sigma(x_0)|\}.
\]
The \emph{incidence matrix} of $\sigma$ is the matrix $M_\sigma \in \N^{\cB \times \cA}$ such that $(M_\sigma)_{b,a} = |\sigma(a)|_b$ for any $a \in \cA$, $b \in \cB$.

The morphism $\sigma$ is said to be {\em left proper} (resp., {\em right proper}) when there exists a letter $\ell \in \cB$ such that for all $a \in \cA$, $\sigma(a)$ starts with $\ell$ (resp., ends with $\ell$).
It is {\em strongly left proper} (resp., {\em strongly right proper}) if it is left proper (resp., right proper) and the starting letter (resp., ending letter) $\ell$ only occurs once in each image $\sigma(a)$, $a \in \cA$. 
It is said to be {\em proper} if it is both left and right proper.
With a left proper morphism $\sigma:\cA^* \to \cB^*$ with first letter $\ell$, we associate a right proper morphism $\bar{\sigma}:\cA^* \to \cB^*$ by
\[
	\sigma(a)\ell = \ell \bar{\sigma}(a), \quad \forall a \in \cA.	
\]

Let $\bsigma = (\sigma_n : \cA_{n+1}^* \to \cA_n^*)_{n \geq 1}$ be a sequence of morphisms such that 
$
\max_{a \in \cA_n} |\sigma_1 \circ \cdots \circ \sigma_{n-1}(a)|
$
goes to infinity when $n$ increases. We assume that all the alphabets $\cA_{n}$ are minimal, in the sense that for all $n \in \mathbb{N}$ and $b \in \cA_{n}$, there exists $a \in \cA_{n+1}$ such that $b$ is a factor of $\sigma_n(a)$.
For $1\leq n<N$, we define the morphism $\sigma_{[n,N)} = \sigma_n \circ \sigma_{n+1} \circ \dots \circ \sigma_{N-1}$.
For $n\geq 1$, the \emph{language $\mathcal{L}^{(n)}({\bsigma})$ of level $n$ associated with $\bsigma$} is defined~by 
\[
\mathcal{L}^{(n)}({\bsigma}) = 
\left\{ w \in \cA_n^* \mid \mbox{$w$ occurs in $\sigma_{[n,N)}(a)$ for some $a \in \cA_N$ and $N>n$} \right\}.
\]

As $\max_{a \in \cA_n} |\sigma_{[1,n)}(a)|$ goes to infinity when $n$ increases, $\mathcal{L}^{(n)}({\bsigma})$ defines a non-empty shift space $X_{\bsigma}^{(n)}$ that we call the {\em shift space generated by $\mathcal{L}^{(n)}({\bsigma})$}.
More precisely, $X_{\bsigma}^{(n)}$ is the set of points $x \in \cA_n^\mathbb{Z}$ such that $\cL (x) \subseteq \cL^{(n)}({\bsigma})$. 
Note that it may happen that $\cL(X_{\bsigma}^{(n)})$ is strictly contained in $\mathcal{L}^{(n)}({\bsigma})$.
Also observe that for all $n$, $X_{\bsigma}^{(n)}$ is the image of $X_{\bsigma}^{(n+1)}$ under $\sigma_n$.

We set $\mathcal{L}({\bsigma}) = \mathcal{L}^{(1)}({\bsigma})$, $X_{\bsigma} = X_{\bsigma}^{(1)}$ and call $X_{\bsigma}$ the {\em $\cS$-adic shift} generated by the {\em directive sequence}~$\bsigma$. 
We also say that the directive sequence $\bsigma$ is an {\em $\cS$-adic representation} of $X_{\bsigma}$.

We say that  $\bsigma$ is {\em primitive} if, for any $n\geq 1$, there exists $N>n$ such that for all $(a,b) \in \cA_n \times \cA_N$, $a$ occurs in $\sigma_{[n,N)}(b)$.
Observe that if $\bsigma$ is primitive, then $\min_{a \in \cA_n} |\sigma_{[1,n)}(a)|$ goes to infinity when $n$ increases, $\mathcal{L}(X_{\bsigma}^{(n)})= \mathcal{L}^{(n)}({\bsigma})$, and $X_{\bsigma}^{(n)}$ is a minimal shift space (see for instance~\cite[Lemma 7]{Durand:2000}).

We say that $\bsigma$ is ({\em (strongly) left}, {\em (strongly) right}) {\em proper} whenever each morphism $\sigma_n$ is ((strongly) left, (strongly) right) proper.
We also say that $\bsigma$ is {\em injective} if each morphism $\sigma_n$ is injective (seen as an application from $\cA_{n+1}^*$ to $\cA_n^*$).
By abuse of language, we say that a shift space is a (strongly left or right proper, primitive, injective) $\cS$-adic shift if there exists a (strongly left or right proper, primitive, injective) sequence of morphisms $\bsigma$ such that $X = X_{\bsigma}$.

%########################
%########################
\section{$\cS$-adicity using return words and shapes of Rauzy graphs}
\label{section:s-adicity}
%########################
%########################

%########################
\subsection{$\cS$-adicity using return words and derived shifts}
\label{subsection:S-adicity through derivation}
%########################

Let $X$ be a minimal shift space over the alphabet $\cA$ and let $w \in \cL(X)$ be a non-empty word.
A {\em return word} to $w$ in $X$ is a non-empty word $r$ such that $w$ is a prefix of $rw$ and, $rw \in \cL(X)$ and $rw$ contains exactly two occurrences of $w$ (one as a prefix and one as a suffix).
We let $\cR_X(w)$ denote the set of return words to $w$ in $X$ and we omit the subscript $X$ whenever it is clear from the context.
The shift space $X$ being minimal, $\cR(w)$ is always finite.

Let $w \in \cL(X)$ be a non-empty word and write $R_X(w) = \{1,\dots,\#(\cR_X(w))\}$.
A morphism $\sigma: {R(w)}^* \to \cA^*$ is a {\em coding morphism} associated with $w$ if $\sigma(R(w)) = \cR(w)$.
It is trivially injective.
Let us consider the set $\cD_w(X) = \{x \in {R(w)}^\ZZ \mid \sigma(x) \in X\}$.
It is a minimal shift space, called the {\em derived shift of $X$ (with respect to $w$)}.
We now show that derivation of minimal shift spaces allows to build left proper and primitive $\cS$-adic representations. 
We inductively define the sequences $(a_n)_{n \geq 1}$, $(R_n)_{n \geq 1}$, $(X_n)_{n \geq 1}$ and $(\sigma_n)_{n \geq 1}$ by
\begin{itemize}
\item
$X_1 = X$, $R_1 = \cA$ and $a_1 \in \cA$;
\item
for all $n$, $R_{n+1} = R_{X_{n}}(a_n)$, $\sigma_n: R_{n+1}^* \to R_n^*$ is a coding morphism associated with $a_n$, $X_{n+1} = \cD_{a_n}(X_n)$ and $a_{n+1} \in R_{n+1}$.
\end{itemize}

Observe that the sequence $(a_n)_{n \geq 1}$ is not uniquely defined as well as the morphism $\sigma_n$ (even if $a_n$ is fixed).
However, to avoid heavy considerations when we deal with sequences of morphisms obtained in this way, we will speak about ``the'' sequence $(\sigma_n)_{n \geq 1}$ and it is understood that we may consider any such sequence. 
Also observe that as we consider derived shifts with respect to letters, each coding morphism $\sigma_n$ is strongly left proper.
It is furthermore a characterization: a morphism $\sigma:\{1,\dots,n\}^* \to \cA^*$ is injective and strongly left proper (with first letter $\ell$) if and only if there is a shift space $X$ over $\cA$ such that $\sigma$ is a coding morphism associated with $\ell$.

\begin{theorem}[Durand~\cite{Durand:1998}]
\label{thm:S-adic representation of minimal}
Let $X$ be a minimal shift space.
Using the notation defined above, the sequence of morphisms $\bsigma = (\sigma_n:R_{n+1}^* \to R_n^*)_{n \geq 1}$ is a strongly left proper, primitive and injective $\cS$-adic representation of $X$.
In particular, for all $n$, we have $X_n = X_{\bsigma}^{(n)}$.
\end{theorem}

In the case of minimal dendric shifts, the $\cS$-adic representation $\bsigma$ can be made stronger.
This is summarized by the following result.
Recall that if $F_\cA$ is the free group generated by $\cA$, an automorphism $\alpha$ of $F_\cA$ is {\em tame} if it belongs to the monoid generated by the permutations of $\cA$ and by the elementary automorphisms 
\[
	\begin{cases}
		a \mapsto ab, \\
		c \mapsto c, & \text{for } c \neq a,
	\end{cases}
\qquad \text{and} \qquad
	\begin{cases}
		a \mapsto ba, \\
		c \mapsto c, & \text{for } c \neq a.
	\end{cases}
\]

\begin{theorem}[Berthé et al.~\cite{bifix_decoding}]
\label{thm:decoding dendric}
Let $X$ be a minimal dendric shift over the alphabet $\cA = \{1,\dots,d\}$.
For any $w \in \cL(X)$, $\cD_w(X)$ is a minimal dendric shift over $\cA$ and the coding morphism associated with $w$ is a tame automorphism of $F_\cA$.
As a consequence, if $\bsigma = (\sigma_n)_{n \geq 1}$ is the primitive directive sequence of Theorem~\ref{thm:S-adic representation of minimal}, then all morphisms $\sigma_n$ are strongly left proper tame automorphisms of $F_\cA$. 
\end{theorem}

%########################
\subsection{Rauzy graphs}
%########################

Let $X$ be a shift space over an alphabet $\cA$. 
The \textit{Rauzy graph of order $n$} of $X$, is the directed graph $G_n(X)$ whose set of vertices is $\cL_n(X)$ and there is an edge from $u$ to $v$ if there are letters $a,b$ such that $ub = av \in \cL(X)$; this edge is sometimes labeled by $a$.

Assuming that $X$ is minimal, any Rauzy graph $G_n(X)$ is strongly connected. 
It is also easily seen that any return word to a non-empty word $w \in \cL(X)$ labels a path from $w$ to $w$ in $G_{|w|}(X)$ in which no internal vertex is $w$. 
As a consequence, the shape of the Rauzy graph $G_1(X)$ provides restrictions on the possible return words to a letter $a$ in $X$.
Furthermore, the extension graph of the empty word having for edges the pairs $(a,b) \in E_X(\varepsilon)$, it completely determines $\cL_2(X)$, hence the Rauzy graph $G_1(X)$.
Therefore, the extension graph $\cE_X(\varepsilon)$ provides restrictions on the possible return words to letters in $X$.

\begin{example} 
\label{ex:rauzy graph fibo}
The Rauzy graph of order 2 of the shift space $X$ generated by the Fibonacci substitution is given in Figure~\ref{fig:rauzy graph fibo order 2}. 
The word $00100$ belongs to $\cL(X)$, so that $001$ is a return word to $00$ and it labels the circuit $(00,01,10,00)$.
The converse however does not hold, i.e., there might exist  paths from $w$ to $w$ whose label is not a return word.
For any $n \geq 1$, the word $0(01)^n$ labels a path from $00$ to $00$ but does not belong to $\cL(X)$ when $n \geq 3$.

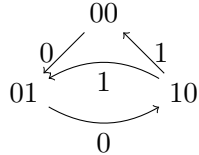
\begin{figure}[h]
\centering
\begin{tikzpicture}
\node (aa) at (0,0) {$00$};
\node (ab) [below left of = aa, node distance = 1.5cm] {$01$};
\node (ba) [below right of = aa, node distance = 1.5cm] {$10$};

\draw [->] (aa) edge node[left,pos=.5] {$0$} (ab);
\draw [->,bend right] (ab) edge node[below,pos=.5] {$0$} (ba);
\draw [->,bend right] (ba) edge node[below,pos=.5] {$1$} (ab);
\draw [->] (ba) edge node[right,pos=.5] {$1$}(aa);
\end{tikzpicture}
\caption{Rauzy graph of order 2 for the Fibonacci shift}
\label{fig:rauzy graph fibo order 2}
\end{figure}

\end{example}

\section{Bispecial factors in $\cS$-adic shifts}
\label{section:bispecial factors}

%########################
\subsection{Description of bispecial factors in injective strongly left proper $\cS$-adic shifts}
%########################

Our aim is to describe bispecial factors and their bi-extensions in an $\cS$-adic shift $X_\bsigma$. 
A classical way to do this is to ``desubstitute'' a bispecial factor $u$, i.e., to find the set of ``minimal'' factors $v_i$ in $\cL(X_\bsigma^{(k)})$ such that $u$ is a factor of the words $\sigma_{[1,k)}(v_i)$ and then to deduce the extensions of $u$ from those of the $v_i$'s.
The set of such $v_i$'s can be easily described when $\bsigma$ is an injective and strongly left proper sequence of morphisms.

\begin{proposition}
\label{prop:def antecedent and bsp ext image}
Let $X$ be a shift space over $\cA$, $\sigma:\cA^* \to \cB^*$ be an injective and strongly left proper morphism (with first letter $\ell$), $Y$ the image of $X$ under $\sigma$ and $u$ a non-empty word in $\cL(Y)$.
If $\ell$ does not occur in $u$, then there exists $b \in \cA$ such that $u$ is a non-prefix factor of $\sigma (b)$.  
Otherwise, there is a unique triplet $(s,v,p) \in \cB^* \times \cL(X) \times \cB^*$ for which there exists a pair $(a,b) \in E_{X}(v)$ 
such that $u = s \sigma(v)p$ with
\begin{enumerate}
\item
$s$ a proper suffix of $\sigma(a)$;
\item
$p$ is a non-empty prefix of $\sigma(b)$.
\end{enumerate}

In the case where $\ell$ occurs in $u$, the left, right and bi-extensions of $u$ are governed by those of $v$ through the equation
\begin{equation}
\label{eq:extensions of u from v}
	E_Y(u) = 
		\{(a',b') \in \cB \times \cB 
		\mid 
		\exists (a,b) \in E_X(v): 
	 	\sigma(a) \in \cB^* a' s 
	 	\wedge 
	 	\sigma(b)\ell \in p b' \cB^* \}.
\end{equation}
\end{proposition}

\begin{proof}
Since $u$ is a word in $\cL(Y)$, it is a factor of $\sigma(w)$ for some $w \in \cL(X)$.
The word $u$ being non-empty, any such $w$ is non-empty as well.
We say that $w$ is {\em covering} $u$ if $u$ is a factor of $\sigma(w)$ and for any proper factor $w'$ of $w$, $u$ is not a factor of $\sigma(w')$.
Recall that $\ell$ is the first letter of $\sigma(a)$ for all $a \in \cA$.
Thus, if $|u|_\ell = 0$, then any word $w$ covering $u$ is a letter and $u$ is a non-prefix factor of $\sigma(w)$.

Now assume that $|u|_\ell \geq 1$ and let $u = s u' p$ with $|s|_\ell = 0$, $p \in \ell (\cA\setminus\{\ell\})^*$ and $u' \in \{\varepsilon\} \cup \ell \cA^*$.
As the letter $\ell$ occurs only as a prefix in any image $\sigma(a)$, $a \in \cA$, any word $w$ covering $u$ is of the form $x v' y$ with $x \in \cA \cup \{\varepsilon\}$ and $y \in \cA$, where one has $\sigma(v') = u'$, $s$ is a suffix of $\sigma(x)$ (which is proper if $x \neq \varepsilon$) and $p$ is a prefix of $\sigma(y)$.
In particular, the triplet $(s,v,p)$ satisfies the requirements of the result (take $b = y$ and $a = x$ if $x \neq \varepsilon$ or any $a \in E^-(vy)$ if $x = \varepsilon$).

Let us show the uniqueness of $(s,v,p)$.
Assume that $(s',v',p')$ is a triplet satisfying the requirements with the extension $(a',b') \in E(v')$. 
As $u = s'\sigma(v')p'$, where $s'$ is a proper suffix of $\sigma(a')$ and $p'$ is a non-empty proper prefix of $\sigma(b')\ell$, we have $|s'|_\ell = 0$, $p' \in \ell (\cA\setminus\{\ell\})^*$ and $\sigma(v') \in \{\varepsilon\} \cup \ell \cA^*$.
This implies that $(s',\sigma(v'), p') = (s,\sigma(v),p)$ and, as $\sigma$ is injective, that $(s',v',p') = (s,v,p)$.

Let us now prove Equation~\eqref{eq:extensions of u from v}.
The inclusion 
\[
	E_Y(u) \supset 
		\{(a', b') \in \cB \times \cB 
		\mid 
		\exists (a,b) \in E_X(v): 
	 	\sigma(a) \in \cB^* a' s 
	 	\wedge 
	 	\sigma(b)\ell \in p b' \cB^* \}
\]
is trivial.
For the other one, assume that $(a',b')$ is in $E_Y(u)$.
Thus we have $a'ub' \in \cL(Y)$.
Let $w \in \cL(X)$ be a covering word for $a'ub'$.
By definition of $v$, $v$ is a factor of $w$, thus one has $w = xvy$ for some words $x,y$.
We then have $a'ub' = a's\sigma(v)pb'$, with $a's$ a suffix of $\sigma(x)$ and $pb'$ a prefix of $\sigma(y)\ell$.
In particular, $x$ and $y$ are non-empty.
Let $a$ be the last letter of $x$ and $b$ be the first letter of $y$.
We have $(a,b) \in E_X(v)$ and,
as $|s|_{\ell} = 0$, $s$ is a proper suffix of $\sigma(a)$, from which we have $\sigma(a) \in \cB^* a's$. 
We also have that $p$ is a prefix of $\sigma(b)$.
If it is proper, then $pb'$ is a prefix of $\sigma(b)$ so that $\sigma(b)\ell \in pb' \cB^*$.
Otherwise, $p = \sigma(b)$, $y$ has length at least 2 and $b'$ is the first letter of $\sigma(c)$, where $c$ is such that $bc$ is prefix of $y$.
Otherwise stated, $b' = \ell$ and we indeed have $\sigma(b)\ell \in pb'\cB^*$. 
\end{proof}

Motivated by the previous result, if $X$ is a shift space over $\cA$ and $\sigma:\cA^* \to \cB^*$ is a strongly left proper morphism (with first letter $\ell$), then for any words $v \in \cL(X)$ and $x,y \in \cB^*$, we define the sets
\begin{align*}
E_{X,x}^-(v) &= \{ a \in E_X^-(v) \mid \sigma(a) \in \cB^* x \};	\\
E_{X,y}^+(v) &= \{ b \in E_X^+(v) \mid \sigma(b)\ell \in y \cB^* \}; 	\\
E_{X,x,y}(v) &= E_X(v) \cap (E_{X,x}^-(v) \times E_{X,y}^+(v)).
\end{align*}
Thus Equation~\eqref{eq:extensions of u from v} can be written
\[
	E_Y(u) = 
		\{(a',b') \in \cB \times \cB 
		\mid 
		\exists (a,b) \in E_{X,a's,pb'}(v) \}.
\]

\begin{remark}
Observe that, as we have seen in the previous proof (and using the same notation), as $s$ is a proper suffix of $\sigma(a)$, the letter $\ell$ does not occur in it.
As a consequence, for any $a' \in E^-_{X,s}(v)$, $s$ is a proper suffix of $\sigma(a')$.
\end{remark}

Whenever $u$ and $v$ are as in the previous proposition with $|u|_\ell \geq 1$, the word $v$ is called the {\em antecedent} of $u$ under $\sigma$ and $u$ is said to be an {\em extended image} of $v$.
Thus, the antecedent is defined only for words containing an occurrence of the letter $\ell$ and an extended image always contains an occurrence of $\ell$.
Whenever $u$ is a bispecial factor, the next result gives additional information about $s$ and $p$.
We first need to define the following notation.
If $\sigma : \cA^* \rightarrow \cB^*$ is a morphism and $a_1, a_2 \in \cA$, let us denote by $s(a_1,a_2)$ (resp., $p(a_1,a_2)$) the longest common suffix (resp., prefix) between $\sigma(a_1)$ and $\sigma(a_2)$.

\begin{corollary}
\label{cor:characterization of extended images}
Let $X$ be a shift space over $\cA$, $\sigma:\cA^* \to \cB^*$ be an injective and strongly left proper morphism (with first letter $\ell$), $Y$ the image of $X$ under $\sigma$ and $v$ a word in $\cL(X)$.
A word $u$ is a bispecial extended image of $v$ if and only if there exist $(a_1,b_1),(a_2,b_2) \in E_X(v)$ with $a_1 \neq a_2$, $b_1 \neq b_2$ and $u = s(a_1,a_2) \sigma(v) p(b_1,b_2)$.
In particular, the antecedent $v$ of a bispecial word $u$ is bispecial.
\end{corollary}
\begin{proof}
First assume that there exist $(a_1,b_1),(a_2,b_2) \in E_X(v)$ with $a_1 \neq a_2$, $b_1 \neq b_2$ and $u = s(a_1,a_2) \sigma(v) p(b_1,b_2)$.
Let us fix $s = s(a_1,a_2)$ and $p = p(b_1,b_2)$.
Since $\sigma$ is injective, $\sigma(a_1) \neq \sigma(a_2)$, hence $s$ is a proper suffix of one of them.
Thus, as $\sigma$ is strongly left proper, $s$ does not contain any occurrence of the letter $\ell$. 
As a consequence, $s$ is a proper suffix of both $\sigma(a_1)$ and $\sigma(a_2)$.
In particular, $u$ is left special.
The same reasoning shows that $p$ is a proper prefix of $\sigma(b_i)$ for some $i \in \{1,2\}$ and that $u$ is right special, hence bispecial.
Furthermore, $p$ is non-empty since it admits $\ell$ as a prefix.
The pair $(a_i,b_i)$ thus satisfies Proposition~\ref{prop:def antecedent and bsp ext image}.

Now assume that $u$ is a bispecial extended image of $v$ with $u = s \sigma(v)p$.
Since $u$ is bispecial, there exist $(a_1',b_1'),(a_2',b_2') \in E_Y(u)$ with $a_1' \neq a_2'$ and $b_1' \neq b_2'$.
From Equation~\ref{eq:extensions of u from v}, there exist $(a_1,b_1),(a_2,b_2) \in E_{X,s,p}(v)$ such that
\[
	\sigma(a_i) \in \cB^* a_i's
	\quad \text{and} \quad
	\sigma(b_i)\ell \in pb_i'\cB^*
\]	
for all $i \in \{1,2\}$. 
We deduce that $s = s(a_1,a_2)$ and $a_1 \neq a_2$.
As $b_1'$ and $b_2'$ cannot be simultaneously equal to $\ell$, we deduce that $b_1 \neq b_2$ and that $p = p(b_1,b_2)$.
\end{proof}

\begin{example}
\label{ex:bispecial images}
Consider the morphism 
\[
	\delta:
	\begin{cases}
		1 \mapsto 1 	\\ 
		2 \mapsto 123	\\
		3 \mapsto 1233	\\
	\end{cases}
\]
that will appear again in Section~\ref{section:ternary case}. Assume that $v$ is a bispecial factor of $X \subset \{1,2,3\}^{\mathbb{Z}}$ whose extension graph is
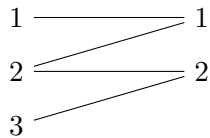
\begin{center}
\begin{tikzpicture}
   \node(1l) {$1$};
   \node(2l) [below= .2cm of 1l] {$2$};
   \node(3l) [below= .2cm of 2l] {$3$};

   \node(1r) [right=2cm of 1l] 	 {$1$};
   \node(2r) [below= .2cm of 1r] {$2$};

	\draw[-] (1l)--(1r);
	\draw[-] (2l)--(1r);
	\draw[-] (2l)--(2r);
	\draw[-] (3l)--(2r);
\end{tikzpicture}
\end{center}

By Corollary~\ref{cor:characterization of extended images}, the word $v$ admits $\delta(v) 1$ and $3 \delta(v) 1$ as bispecial extended images. Using Proposition~\ref{prop:def antecedent and bsp ext image}, their extension graphs are given in Figure~\ref{fig:graph of bispecial images}.

\begin{figure}[h]
\centering
\begin{tabular}{cc}
\begin{tikzpicture}
\node (1) at (1, 1) {$\cE(\delta(v)1)$};
\node (1L) at (0,0) {$1$};
\node (23L) [below of = 1L, node distance = .5cm] {$3$};

\node (1R) [right of = 1L, node distance = 2cm] {$1$};
\node (23R) [below of = 1R, node distance = .5cm] {$2$};

\draw (1L)--(1R);
\draw (23L)--(1R);
\draw (23L)--(23R);
\end{tikzpicture}
&
\begin{tikzpicture}
\node (31) at (1, 1) {$\cE(3 \delta(v) 1)$};
\node (2L) at (0,0) {$2$};
\node (3L) [below of = 2L, node distance = .5cm] {$3$};

\node (1R) [right of = 2L, node distance = 2cm] {$1$};
\node (23R) [below of = 1R, node distance = .5cm] {$2$};

\draw (2L) -- (1R);
\draw (2L) -- (23R);
\draw (3L) -- (23R);
\end{tikzpicture}
\end{tabular}
\caption{Extension graphs of the bispecial extended images of $v$}
\label{fig:graph of bispecial images}
\end{figure}
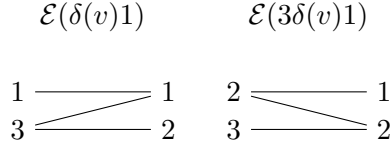

\end{example}

Assume that $X_\bsigma$ is an $\cS$-adic shift where the directive sequence $\bsigma = (\sigma_n:\cA_{n+1}^* \to \cA_n^*)_{n \geq 1}$ is primitive and contains only strongly left proper injective morphisms.
The directive sequence $\bsigma$ being primitive, the sequence $(\min_{a \in \cA_n} |\sigma_{[1,n)}(a)|)_{n \geq 1}$ goes to infinity.
Hence, iterating Proposition~\ref{prop:def antecedent and bsp ext image}, with any word $u \in \cL(X_\bsigma)$ one can associate a unique finite sequence $(u_1,u_2,\dots,u_k)$ such that $u_1 = u$, $u_k \in \cL(X_\bsigma^{(k)})$ does not have any antecedent under $\sigma_k$ and, for $i <k$, $u_{i+1} \in \cL(X_\bsigma^{(i+1)})$ is the antecedent of $u_i$ under $\sigma_i$. 
We say that $u$ is a {\em descendant} of each $u_i$, $1 \leq i \leq k$, and, reciprocally, that each $u_i$, $1 \leq i \leq k$, is an {\em ancestor} of $u$.
The word $u_k \in \cL(X_\bsigma^{(k)})$ is its {\em oldest ancestor} and it is either empty or a non-prefix factor of $\sigma_k (b)$ for some letter $b \in \cA_{k+1}$.
Observe that with our definition, $u$ is an ancestor and a descendant of itself.

Let $X$ be an $\cS$-adic shift with a strongly left proper and injective $\cS$-adic representation $\bsigma$.
Let $u$ be a bispecial factor of $X$.
From Corollary~\ref{cor:characterization of extended images}, all ancestors of $u$ are bispecial factors of some $X_\bsigma^{(k)}$.
From Proposition~\ref{prop:def antecedent and bsp ext image}, the extensions of $u$ are completely governed by those of its oldest ancestor.
More precisely, we have the following direct corollary. 

\begin{corollary}
\label{cor:back to epsilon}
For all $k \geq 1$, there is a finite number of bispecial factors of $X_\bsigma^{(k)}$ that do not have an antecedent under $\sigma_k$. They are called {\em initial bispecial factors of order $k$}.
Furthermore, for any bispecial factor $u$ of $X$, there is a unique $k \geq 1$ and a unique initial bispecial factor $v\in \cL(X_\bsigma^{(k)})$ such that $u$ is a descendant of $v$.
Finally, $\cE_X(u)$ depends only on $\cE_{X_\bsigma^{(k)}}(v)$, i.e., if $Y$ is a shift space such that $\cE_{Y}(v) = \cE_{X_\bsigma^{(k)}}(v)$ and if $Z$ is the image of $Y$ under $\sigma_{[1,k)}$, then $\cE_Z(u) = \cE_X(u)$.
\end{corollary}

%########################
\subsection{Action of morphisms on extension graphs}
\label{subsection:stability of tree}
%########################

Proposition~\ref{prop:def antecedent and bsp ext image} shows that whenever $v$ is the antecedent of $u$ under $\sigma$, the extension graph $\cE_Y(u)$ is the image under a graph morphism of a subgraph of $\cE_X(v)$ (where by subgraph we mean the subgraph generated by a subset of edges).
In particular, if $\#(E^-_Y(u)) = \#(E^-_X(v))$ and $\#(E^+_Y(u)) = \#(E^+_X(v))$, then $\cE_Y(u)$ and $\cE_X(v)$ are isomorphic.
In this section, we formalize this observation and study the behavior of a tree structure when we consider the extension graphs of bispecial extended images.

In this section, $X$ is a shift space over $\cA$, $\sigma:\cA^* \to \cB^*$ an injective and strongly left proper morphism (with first letter $\ell$), $Y$ the image of $X$ under $\sigma$ and $v$ a bispecial word in $\cL(X)$.
By Corollary~\ref{cor:characterization of extended images}, the bispecial extended images of $v$ under $\sigma$ are the words $u$ of the form $s \sigma(v) p$ where $s=s(a_1,a_2)$ and $p=p(b_1,b_2)$ for some $(a_1,b_1),(a_2,b_2) \in E_X(v)$ such that $a_1 \neq a_2$ and $b_1 \neq b_2$.
%Furthermore, using Proposition~\ref{prop:def antecedent and bsp ext image}, $E_Y(u)$ is completely determined by $E_X(v)$.
For any such $\sigma$ and $v$, we introduce the following notations:
\begin{align*}
\cT^-(\sigma) &= \{s(a_1,a_2) \mid a_1, a_2 \in \cA, a_1 \neq a_2\}; \\
\cT^+(\sigma) &= \{p(b_1,b_2) \mid b_1, b_2 \in \cA, b_1 \neq b_2\};	\\ 
\cT^-_v(\sigma) &= \{s(a_1,a_2) \mid a_1,a_2 \in E^-_{X}(v), a_1 \neq a_2\};	\\
\cT^+_v(\sigma) &= \{p(b_1,b_2) \mid b_1,b_2 \in E^+_{X}(v), b_1 \neq b_2\};	\\ 
\overline{\cT^-_v(\sigma)} &= \cT^-_v(\sigma) \cup \sigma(E_X^-(v));\\
\overline{\cT^+_v(\sigma)} &= \cT^+_v(\sigma) \cup \sigma(E_X^+(v)) \ell. 
\end{align*}

The prefix strict order (resp., suffix strict order) defines a tree structure called \emph{radix tree} on $\overline{\cT^+_v(\sigma)}$ (resp., on $\overline{\cT^-_v(\sigma)}$), where the root $p_0$ (resp., $s_0$) is the shortest word of the set.
In particular, $E^-_{X,s_0}(v) = E^-_X(v)$ and $E^+_{X,p_0}(v) = E^+_X(v)$.
Furthermore, the leafs of $\overline{\cT^+_v(\sigma)}$ (resp., $\overline{\cT^-_v(\sigma)}$) are exactly the elements of $\sigma(E_X^+(v)) \ell$ (resp., $\sigma(E_X^-(v))$) and every internal node (i.e., every element of $\cT^+_v(\sigma)$ or $\cT^-_v(\sigma)$) has at least two children.

For $(s,p) \in \cT^-_v(\sigma) \times \cT^+_v(\sigma)$, we define the subgraph $\cE_{X,s,p}(v)$ of $\cE_X(v)$ whose vertices are those involved by the edges  in $E_{X,s,p}(v)$.
Observe that the sets of left and right vertices of $\cE_{X,s,p}(v)$ are respectively included in $E^-_{X,s}(v)$ and $E^+_{X,p}(v)$. 
Furthermore, if $s'\in \overline{\cT^-_v(\sigma)}$ (resp., $p'\in \overline{\cT^+_v(\sigma)}$) is a child of $s$ (resp., of $p$), then $\cE_{X,s',p}(v)$ (resp., $\cE_{X,s,p'}(v)$) is a subgraph of $\cE_{X,s,p}(v)$.

\begin{example}
Using the notations from Example~\ref{ex:bispecial images}, we have the following radix trees:
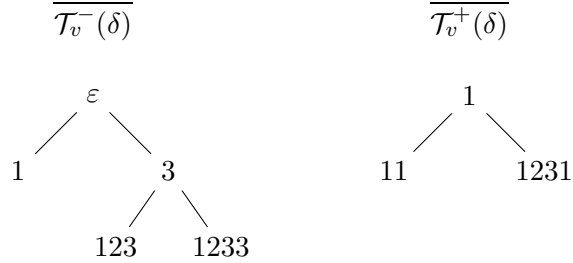
\begin{center}
\begin{tikzpicture}
	\node(S) at (-5,0) {$\overline{\cT^-_v(\delta)}$};
	\node(s0)at (-5,-1) {$\varepsilon$}; 
	\node(1s)at (-6,-2){$1$};
	\node(23s)at (-4,-2) {$3$};
	\node(2s)at (-4.7,-3)  {$123$};
	\node(3s)at (-3.3,-3) {$1233$};

	\draw (s0) -- (1s);
	\draw (s0) -- (23s);
	\draw (23s) -- (2s);
	\draw (23s) -- (3s);

	\node(P) at (0,0) {$\overline{\cT^+_v(\delta)}$};
	\node(p0)at (0,-1) {$1$}; 
	\node(1p)at (-1,-2){$11$};
	\node(2p)at (1,-2) {$1231$};

	\draw (p0) -- (1p);
	\draw (p0) -- (2p);
\end{tikzpicture}
\end{center}
These structures help us understand the construction of the extension graphs of Figure~\ref{fig:graph of bispecial images}. Let $s \in \cT^-_v(\sigma)$ and $p \in \cT^+_v(\sigma)$ be such that $u = s\sigma(v)p$.
The extension graph of $u$ can be obtained from the extension graph of $v$ as follows:
\begin{enumerate}
	\item
	Start by selecting the elements of $E^-_{X, s}(v)$ and $E^+_{X, p}(v)$ (see Table~\ref{table:graph of bispecial images step 1}).
	These elements are the letters such that the corresponding leaf in $\overline{\cT^-_v(\sigma)}$ (resp., $\overline{\cT^+_v(\sigma)}$) is in the subtree with root $s$ (resp., $p$). 
	\item
	Take the subgraph of $\cE_X(v)$ with only the vertices which are in these two sets and remove the isolated vertices that were created. This gives the graph $\cE_{X, s, p}(v)$ (see Figure~\ref{fig:graph of bispecial images step 2}).
	\item
	For any letter $a \in \cA$, merge the vertices $b$ on the left side such that $\sigma(b) \in \cA^*as$ into a new left vertex labeled by $a$.
	In other words, for any left vertex $b$, map it to the left vertex labeled by the letter $a$ such that the leaf corresponding to $b$ is in the subtree whose root is the only child of $s$ ending by $as$.
	Do the same on the right side with the vertices $b$ such that $\sigma(b) \ell \in pa\cA^*$ (see Table~\ref{table:graph of bispecial images step 3}).
\end{enumerate}

\begin{table}[h]
\centering
$\begin{array}{c|c}
\delta(v)1
&3\delta(v)1
\\
\hline
&\\
E^-_{X, \varepsilon}(v) = \{1, 2, 3\}
&
E^-_{X, 3}(v) = \{2,3\}
\\
E^+_{X, 1}(v) = \{1,2\}  
&
E^+_{X, 1}(v) = \{1,2\}
\end{array}$
\caption{Step 1}
\label{table:graph of bispecial images step 1}
\end{table}

\begin{figure}[h]
\centering
\begin{tabular}{cc}
$\cE_{X, \varepsilon, 1}(v)$
&
$\cE_{X, 3,1}(v)$
\\
\\
\begin{tikzpicture}
    \node(1L) at (0,0) {$1$};
    \node(2l) [below= .2cm of 1l] {$2$};
    \node(3l) [below= .2cm of 2l] {$3$};

   \node(1r) [right=2cm of 1l] 	 {$1$};
   \node(2r) [below= .2cm of 1r] {$2$};

	\draw[-] (1l)--(1r);
	\draw[-] (2l)--(1r);
	\draw[-] (2l)--(2r);
	\draw[-] (3l)--(2r);
\end{tikzpicture}
&
\begin{tikzpicture}
    \node(2l) at (0,0) {$2$};
    \node(3l) [below= .2cm of 2l] {$3$};
    \node(1l) [below= .35cm of 3l] {};

   \node(1r) [right=2cm of 2l] 	 {$1$};
   \node(2r) [below= .2cm of 1r] {$2$};

	\draw[-] (2l)--(1r);
	\draw[-] (2l)--(2r);
	\draw[-] (3l)--(2r);
\end{tikzpicture}
\end{tabular}
\caption{Step 2}
\label{fig:graph of bispecial images step 2}
\end{figure}

\begin{table}[h]
\centering
$\begin{array}{|c|c|c|}
\hline
&
\text{Left side}
&
\text{Right side}
\\
\hline
\delta(v) 1
&
\begin{array}{c}
1 \mapsto 1\\
2 \mapsto 3\\
3 \mapsto 3\\
\end{array}
&
\begin{array}{c}
1 \mapsto 1\\
2 \mapsto 2\\
\end{array}
\\
\hline
3 \delta(v) 1
&
\begin{array}{c}
2 \mapsto 2\\
3 \mapsto 3\\
\end{array}
&
\begin{array}{c}
1 \mapsto 1\\
2 \mapsto 2\\
\end{array}
\\
\hline
\end{array}$

\caption{Step 3}
\label{table:graph of bispecial images step 3}
\end{table}

\end{example}

The next result gives a more formal description of this construction and directly follows from Equation~\eqref{eq:extensions of u from v}.

\begin{proposition}
\label{prop:morphic image of extension graph}
If $(s,p) \in \cT^-_v(\sigma) \times \cT^+_v(\sigma)$ is such that $u = s \sigma(v)p$ is an extended image of $v$, the extension graph of $u$ is the image of $\cE_{X,s,p}(v)$ under the graph morphism $\varphi_{v,s,p}:\cE_{X,s,p}(v) \to \cE_{Y}(u)$ that
\begin{itemize}
\item
    for every child $s'$ of $s$ in $\overline{\cT^-_v(\sigma)}$, maps all vertices of $E^-_{X,s'}(v)$ to the left vertex with label $a \in \cB$ such that $s' \in \cB^* as$;
\item
    for every child $p'$ of $p$ in $\overline{\cT^+_v(\sigma)}$, maps all vertices of $E^+_{X,p'}(v)$ to the right vertex with label $b \in \cB$ such that $p' \in pb \cB^*$.
\end{itemize}
In particular, if $\cT^-_v(\sigma) = \{s_0\}$ and $\cT^+_v(\sigma) = \{p_0\}$, then $v$ has a unique bispecial extended image $u$ and the associated morphism $\varphi_{v,s,p}$ is an isomorphism. 
\end{proposition}

Observe that the morphism $\varphi_{v,s,p}$ of the previous result acts independently on the left and right vertices of $\cE_{X,s,p}(v)$, i.e. it can be seen as the composition of two commuting graph morphisms $\varphi^-_{v,s}$ and $\varphi^+_{v,p}$, where $\varphi^-_{v,s}$ acts only on the left vertices of $\cE_{X,s,p}(v)$ and $\varphi^+_{v,p}$ acts only on the right vertices of $\cE_{X,s,p}(v)$.
Furthermore, if a letter $a$ belongs to $E^-_{X,s}(v) \cap E^-_{X,s}(v')$ (resp., to $E^+_{X,p}(v) \cap E^+_{X,p}(v')$) then $\varphi^-_{v,s}(a) = \varphi^-_{v',s}(a)$ (resp., $\varphi^+_{v,p}(a) = \varphi^+_{v',p}(a)$).
Thus we can define partial maps $\varphi^-_{s},\varphi^+_{p}:\cA \to \cB$ by $\varphi^-_{s}(a) = \varphi^-_{\varepsilon,s}(a)$ whenever $a \in E^-_{X,s}(\varepsilon)$ and by $\varphi^+_{p}(a) = \varphi^+_{\varepsilon,p}(a)$ whenever $a \in E^+_{X,p}(\varepsilon)$.

\subsection{Stability of dendricity}

In this section, we use the results and notations of the previous section to understand under which conditions a dendric bispecial factor only has dendric bispecial extended images under some morphism.
We then characterize the morphisms for which every dendric bispecial factor only has dendric bispecial extended images.

If $X$ is a shift space over $\cA$ and $v$ is a dendric bispecial factor of $X$, we say that an injective and strongly left proper morphism $\sigma:\cA^* \to \cB^*$ is {\em dendric preserving} for $v \in \cL(X)$ if all bispecial extended images of $v$ under $\sigma$ are dendric.
We extend this definition by saying that a morphism is \emph{dendric preserving} for a shift space $X$ if it is dendric preserving for all $v \in \cL(X)$.

\begin{proposition}
\label{prop:bispecial has only tree bispecial extensions}
Let $X$ be a shift space over $\cA$ and $v \in \cL(X)$ be a dendric bispecial factor.
An injective and strongly left proper morphism $\sigma:\cA^* \to \cB^*$ is dendric preserving for $v$ if and only if the following conditions are satisfied
\begin{enumerate}
\item \label{item 1 tree extension}
for every $s \in \cT^-_v(\sigma) \setminus \{s_0\}$, $\cE_{X,s,p_0}(v)$ is a tree;
\item \label{item 2 tree extension}
for every $p \in \cT^+_v(\sigma) \setminus \{p_0\}$, $\cE_{X,s_0,p}(v)$ is a tree.
\end{enumerate}
\end{proposition}

\begin{proof}
Let us first assume that every bispecial extended image of $v$ is dendric.
We show item~\ref{item 2 tree extension}, the other one being symmetric.
Consider $p \in \cT^+_v(\sigma)$.
The graph $\cE_{X,s_0,p}(v)$ is a subgraph of $\cE_{X}(v)$, which is a tree.
Thus, $\cE_{X,s_0,p}(v)$ is acyclic.
Let us show that it is connected.
As there is no isolated vertex, it suffices to show that for all distinct right vertices $b_1,b_2$ of $\cE_{X,s_0,p}(v)$, there is a path in $\cE_{X,s_0,p}(v)$ from $b_1$ to $b_2$.
Assume by contrary that there exist two right vertices $b_1,b_2$ of $\cE_{X,s_0,p}(v)$ that are not connected.

For $i \in \{1,2\}$, let $A_i$ denote the set of left vertices of $\cE_{X,s_0,p}(v)$ that are connected to $b_i$.
Consider a maximal word $s'$ (for the suffix order) in $\{s(a_1,a_2) \mid a_1 \in A_1,a_2 \in A_2\}$.
Let also $B_i$, $i \in \{1,2\}$, denote the set of right vertices of $\cE_{X,s',p}(v)$ that are connected to vertices of $A_i$ and consider a maximal word $p'$ (for the prefix order) in $\{p(b_1',b_2') \mid b_1' \in B_1,b_2' \in B_2\}$.
We claim that the extension graph of the bispecial extended image $u = s' \sigma(v) p'$ is not connected.

%Let $b_1' \in B_1$ and $b_2' \in B_2$ such that $p(b_1',b_2')=p'$ and
Let $Y$ be the image of $X$ under $\sigma$ and let $\varphi$ be the morphism from $\cE_{X,s',p'}(v)$ to $\cE_Y(u)$ given by Proposition~\ref{prop:morphic image of extension graph}.
By maximality of $s'$ and $p'$, the morphism $\varphi$ identifies two left vertices $a,a'$ (resp., right vertices $b,b'$) only if they belong to the same $A_i$ (resp., $B_i$).
This implies that $\cE_Y(u)$ is not connected, which is a contradiction.

Let us now show that, under the hypothesis \ref{item 1 tree extension} and \ref{item 2 tree extension}, any bispecial extended image of $v$ is dendric.
By Corollary~\ref{cor:characterization of extended images}, the bispecial extended images of $v$ are of the form $s \sigma(v) p \in \cL(Y)$ where $s$ is in $\cT^-_v(\sigma)$ and $p$ is in $\cT^+_v(\sigma)$.

For any such pair $(s, p)$, we first show that the graph $\cE_{X,s,p}(v)$ is a tree.
It is trivially acyclic as it is a subgraph of $\cE_{X,s,p_0}(v)$, which is assumed to be a tree.
Let us show that it is connected.
Let $b_1,b_2$ be right vertices of $\cE_{X,s,p}(v)$.
As $\cE_{X,s,p}(v)$ is a subgraph of both $\cE_{X,s_0,p}(v)$ and $\cE_{X,s,p_0}(v)$ which are trees, there exist a path $q$ in $\cE_{X,s_0,p}(v)$ and a path $q'$ in $\cE_{X,s,p_0}(v)$ connecting $b_1$ and $b_2$.
In particular, the right vertices occurring in $q$ belong to $E^+_{X,p}(v)$ and the left vertices occurring in $q'$ belong to $E^-_{X,s}(v)$. 
As $\cE_{X,s_0,p}(v)$ and $\cE_{X,s,p_0}(v)$ are both subgraphs of $\cE_{X,s_0,p_0}(v)$, which is a tree, the paths $q$ and $q'$ coincide.
It means that this path only goes through left vertices belonging to $E^-_{X,s}(v)$ and through right vertices belonging to $E^+_{X,p}(v)$. 
This implies that it is a path of $\cE_{X,s,p}(v)$, hence that $b_1$ and $b_2$ are connected in $\cE_{X,s,p}(v)$.

We now show that if $u = s \sigma(v)p$ is an extended image of $v$, then it is dendric.
By Proposition~\ref{prop:morphic image of extension graph}, $\cE_Y(u)$ is the image of $\cE_{X,s,p}(v)$ under some graph morphism $\varphi$, hence it is connected.
We proceed by contradiction to show that it is acyclic.
Assume that $c = (a_1,b_1,a_2,b_2, \dots, a_{n},b_{n},a_1)$, $n \geq 2$, is a non-trivial cycle in $\cE_Y(u)$, where $a_1,\dots,a_{n}$ are left vertices and $b_1,\dots,b_{n}$ are right vertices.
Again by Proposition~\ref{prop:morphic image of extension graph}, for every $i\leq n$, there exist
\begin{itemize}
\item
a child $s_i$ of $s$ in $\overline{\cT^-_v(\sigma)}$;
\item
a child $p_i$ of $p$ in $\overline{\cT^+_v(\sigma)}$;
\item
left vertices $a_i', a_i''$ of $\cE_{X,s,p}(v)$, belonging to $E^-_{X,s_i}(v)$;
\item
right vertices $b_i', b_i''$ of $\cE_{X,s,p}(v)$, belonging to $E^+_{X,p_i}(v)$;
\end{itemize}
such that $\varphi(a_i') = \varphi(a_i'') = a_i$, $\varphi(b_i') = \varphi(b_i'') = b_i$ and such that the pairs 
\begin{align*}	
	& (a_j',b_j'), (a_{j+1}'',b_j''), \quad j < n, \\
	& (a_n',b_n'), (a_1'',b_n''),
\end{align*}
are edges of $\cE_{X,s,p}(v)$.

Observe that as $\cE_{X,s,p}(v)$ is a tree, for each $i\leq n$ there is a unique simple path $q_{i}$ from $a_{i}''$ to $a_{i}'$ and a unique simple path $q_i'$ from $b_i'$ to $b_i''$. In $\cE_{X,s,p}(v)$, we thus have the circuit
\[
	c' = (\underbrace{a_1'',\dots,a_1'}_{q_1},  
	\underbrace{b_1',\dots,b_1''}_{q_1'},
	\underbrace{a_2'',\dots,a_2'}_{q_2},
	\underbrace{b_2',\dots,b_2''}_{q_2'}, 
	\dots,	
	\underbrace{a_n'',\dots,a_n'}_{q_n},
	\underbrace{b_n',\dots,b_n''}_{q_n'},
	a_1'').
\]

We will now prove that the image of $c'$ by $\varphi$ reduces to the non trivial cycle $c$ of $\cE_Y(u)$. By reducing, we mean that we remove consecutive redundant edges, i.e. every occurrence of $a,b,a$ in the path is replaced by $a$. This will imply that $c'$ is not trivial and contradict the fact that $\cE_{X,s,p}(v)$ is a tree, which will end the proof.

As $\cE_{X,s_i,p}(v)$ is a sub-tree of $\cE_{X,s,p}(v)$, the path $q_i$ is a path of $\cE_{X,s_i,p}(v)$, otherwise this would contradict the fact that $\cE_{X,s,p}(v)$ is acyclic.
In particular,
the only left vertex of $\varphi(q_i)$ is $a_i$ thus $\varphi(q_i)$ reduces to the length-$0$ path $a_i$ in $\cE_Y(u)$. Similarly, $q_i'$ is a path of $\cE_{X,s,p_i}(v)$ thus $\varphi(q_i')$ reduces to the length-$0$ path $b_i$. This concludes the proof that $\varphi(c')$ reduces to $c$.
\end{proof}

\begin{corollary}\label{cor:image of Arnoux-Rauzy}
If $v$ is an ordinary bispecial factor of a shift space over $\cA$, then any injective and strongly left proper morphism $\sigma : \cA^* \to \cB^*$ is dendric preserving for $v$.
In particular, the image under $\sigma$ of an Arnoux-Rauzy shift space over $\cA$ is a minimal dendric shift. 
\end{corollary}

The previous result is illustrated in the preceding examples.
Indeed, for $\delta$ and $v$ as in Example~\ref{ex:bispecial images}, we have $\cT^-_v(\delta) = \{3, \varepsilon\}$ and $\cT^+_v(\delta) = \{1\} = \{p_0\}$. 
The extension graph $\cE_{X,3,1}(v)$ in Figure~\ref{fig:graph of bispecial images step 2} being a tree, $v$ has only dendric bispecial extended images, as already observed in Figure~\ref{fig:graph of bispecial images}.

Another consequence of Proposition~\ref{prop:bispecial has only tree bispecial extensions} is given by the following corollary.

\begin{corollary}
\label{cor:dendric preserving for any word}
Let $\sigma:\cA^* \to \cB^*$ be an injective and strongly left proper morphism. The following properties are equivalent.
\begin{enumerate}
\item
	The sets $\cT^-(\sigma)$ and $\cT^+(\sigma)$ both only contain one element.
\item
	For any shift space $X$ on the alphabet $\cA$ and any dendric bispecial factor $v \in \cL(X)$, $\sigma$ is dendric preserving for $v$.
\end{enumerate}
\end{corollary}
\begin{proof}
If $\cT^-(\sigma) = \{s_0\}$ and $\cT^+(\sigma) = \{p_0\}$, the fact that $\sigma$ is dendric preserving for any dendric bispecial word $v$ is a direct consequence of Proposition~\ref{prop:bispecial has only tree bispecial extensions} since $\cT^-_v(\sigma) \subset \cT^-(\sigma)$ and $\cT^+_v(\sigma) \subset \cT^+(\sigma)$.

To prove the other implication, let us assume that $s_0,s_1 \in \cT^-(\sigma)$, where $s_0$ is the root of $\cT^-(\sigma)$ and $s_1 \neq s_0$.
Let $a, b$ be such that $s_1 = s(a, b)$. As $s_0 \in \cT^-(\sigma)$, there exists $c$ such that $s_1$ is not a suffix of $\sigma(c)$.
In particular, we have $\#\cA \geq 3$.

Let $a_1, \dots, a_n$ be the elements of $\cA \setminus \{a, b, c\}$. 
Let us denote by $X$ the shift space coding the interval exchange transformation represented below (for precise definitions and more details about interval exchanges, see Subsection~\ref{subsection:interval exchange}).

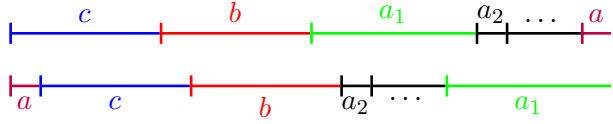
\begin{center}
\begin{tikzpicture}
\draw[|-,line width=1pt,color=blue] 
	(0,0)
	-- 
	node[above] {$c$} 
	(2,0)
	;
\draw[|-,line width=1pt,color=red] 
	(2,0) 
	-- 
	node[above] {$b$} 
	(4,0)
	;
\draw[|-,line width=1pt,color=green] 
	(4,0) 
	-- 
	node[above] {$a_1$} 
	(6.2,0)
	;
\draw[|-,line width=1pt,color=black] 
	(6.2,0) 
	-- 
	node[above] {$a_2$} 
	(6.6,0)
	;
\draw[|-,line width=1pt,color=black] 
	(6.6,0) 
	-- 
	node[above] {$\dots$} 
	(7.6,0)
	;
\draw[|-,line width=1pt,color=purple] 
	(7.6,0) 
	-- 
	node[above] {$a$} 
	(8,0) 
	;
\draw[|-,line width=1pt,color=purple] 
	(0,-.7)
	-- 
	node[below] {$a$} 
	(0.4,-.7);
\draw[|-,line width=1pt,color=blue] 
	(0.4,-.7) 
	-- 
	node[below] {$c$} 
	(2.4,-.7)
	;
\draw[|-,line width=1pt,color=red] 
	(2.4,-.7) 
	-- 
	node[below] {$b$} 
	(4.4,-.7)
	;
\draw[|-,line width=1pt,color=black] 
	(4.4,-.7) 
	-- 
	node[below] {$a_2$} 
	(4.8,-.7)
	;
\draw[|-,line width=1pt,color=black] 
	(4.8,-.7) 
	-- 
	node[below] {$\dots$} 
	(5.8,-.7)
	;
\draw[|-,line width=1pt,color=green] 
	(5.8,-.7) 
	-- 
	node[below] {$a_1$} 
	(8,-.7)
	;
\end{tikzpicture}
\end{center}

The extension graph of the word $\varepsilon$ is given by

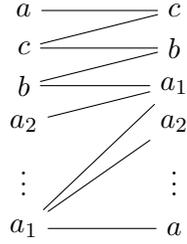
\begin{center}
\begin{tikzpicture}
	\node (aL) at (0,0) {$a$};
	\node (cL) [below of = aL, node distance = .5cm] {$c$};
	\node (bL) [below of = cL, node distance = .5cm] {$b$};
	\node (a2L) [below of = bL, node distance = .5cm] {$a_2$};
	\node (doL) [below of = a2L, node distance = .7cm] {$\vdots$};
	\node (a1L) [below of = doL, node distance = .7cm] {$a_1$};
	
	\node (cR) [right of = aL, node distance = 2cm] {$c$};
	\node (bR) [below of = cR, node distance = .5cm] {$b$};
	\node (a1R) [below of = bR, node distance = .5cm] {$a_1$};
	\node (a2R) [below of = a1R, node distance = .5cm] {$a_2$};
	\node (doR) [below of = a2R, node distance = .7cm] {$\vdots$};
	\node (aR) [below of = doR, node distance = .7cm] {$a$};
	
	\draw[-] (aL)--(cR);
	\draw[-] (cL)--(cR);
	\draw[-] (cL)--(bR);
	\draw[-] (bL)--(bR);
	\draw[-] (bL)--(a1R);
	\draw[-] (a2L)--(a1R);
	\draw[-] (a1L)--(a1R);
	\draw[-] (a1L)--(a2R);
	\draw[-] (a1L)--(aR);
\end{tikzpicture}
\end{center}

It is a tree thus $\varepsilon$ is dendric bispecial. However, the graph $\cE_{X, s_1, p_0}(\varepsilon)$ is not connected as it does not contain the vertex $c$ on the left but both $a$ and $b$ are left vertices. By Proposition~\ref{prop:bispecial has only tree bispecial extensions}, $\sigma$ is not dendric preserving for $\varepsilon$. This proves that $\cT^-(\sigma)$ must contain exactly one element. Similarly, $\cT^+(\sigma)$ also contains exactly one element.
\end{proof}

The previous result characterizes injective and strongly left proper morphisms for which every dendric bispecial factor has only dendric bispecial extended images.
However, the condition does not imply that the image of a dendric shift by such a morphism  $\sigma$ is again dendric.
Indeed, the result gives information only on the bispecial factors that are extended images under $\sigma$, i.e., that have an antecedent.
The next result characterizes those morphisms $\sigma$ for which even the new initial bispecial factors are dendric.
For any letter $a \in \cA$, let $\alpha_a$ and $\bar{\alpha}_a$ denote the so-called {\em Arnoux-Rauzy morphisms}
\[
	\alpha_a(b) =
	\begin{cases}
		a & \text{ if } b = a,\\
		ab & \text{ otherwise,}
	\end{cases}
	\qquad
	\bar{\alpha}_a(b) = 
	\begin{cases}
		a & \text{ if } b = a,\\
		ba & \text{ otherwise.}
	\end{cases}
\]

\begin{proposition}
\label{prop:characterization dendric preserving}
The injective and strongly left proper morphisms preserving dendricity, i.e. such that the image of any dendric shift is a dendric shift, are exactly the morphisms
\[
	\bar{\alpha}_{a_1} \circ \dots \circ \bar{\alpha}_{a_n} \circ \alpha_\ell \circ \pi
\]
for any $n \geq 0$, any $a_1, \dots, a_n \in \cA \setminus \{\ell\}$ and any permutation $\pi$ of $\cA$.
\end{proposition}

\begin{proof}
Using Corollary~\ref{cor:dendric preserving for any word}, an injective and strongly left proper morphism $\sigma$ for the letter $\ell$ preserves dendricity if and only if the following conditions are satisfied:
\begin{enumerate}
\item
\label{item:S and P contain one element}
	the sets $\cT^-(\sigma)$ and $\cT^+(\sigma)$ both only contain one element;
\item
\label{item:initial factors are dendric}
	if $F_\sigma$ is the set
	\[
		F_\sigma = \Fac\left(\{\sigma(a)\ell : a \in \cA\}\right),
	\]
	then any word $w \in F_\sigma$ such that $|w|_\ell = 0$ is dendric in $F_\sigma$.
\end{enumerate}
Let $\sigma = \bar{\alpha}_{a_1} \circ \dots \circ \bar{\alpha}_{a_n} \circ \alpha_\ell \circ \pi$ for $a_1, \dots, a_n \in \cA \setminus \{\ell\}$. It is easily verified that $\sigma$ is injective and strongly left proper for the letter $\ell$. As conditions~\ref{item:S and P contain one element} and \ref{item:initial factors are dendric} only depend on the set $\sigma(\cA)$, one can assume that $\pi = id$.
For any letter $c$, $\ell c$ is a prefix of $\alpha_\ell(c) \ell$ thus
\[
	\left(\bar{\alpha}_{a_1} \circ \dots \circ \bar{\alpha}_{a_n}(\ell)\right) c
\]
is a prefix of $\sigma(c) \ell$. This shows that
\[
	\cT^+(\sigma) = \{\bar{\alpha}_{a_1} \circ \dots \circ \bar{\alpha}_{a_n}(\ell)\}.
\]
Similarly, $c \ell$ is a suffix of $\ell \bar{\alpha}_\ell(c)$ thus
\[
	c \left(\alpha_{a_1} \circ \dots \circ \alpha_{a_n} (\ell)\right)
\]
is a suffix of
\[
	\ell \left(\alpha_{a_1} \circ \dots \circ \alpha_{a_n} \circ \bar{\alpha}_\ell(c)\right)
	= \left(\bar{\alpha}_{a_1} \circ \dots \circ \bar{\alpha}_{a_n} \circ \alpha_\ell(c)\right) \ell
	= \sigma(c) \ell
\]
thus $\cT^-(\sigma)$ only contains one element and $\sigma$ satisfies the condition~\ref{item:S and P contain one element}.
To prove condition~\ref{item:initial factors are dendric}, let us proceed by induction on $n$. If $n = 0$, then the only bispecial word is $\varepsilon$ and it is easy to verify that it is dendric. If $\sigma' = \bar{\alpha}_{a_2} \circ \dots \circ \bar{\alpha}_{a_n} \circ \alpha_\ell$ satisfies condition~\ref{item:initial factors are dendric}, then simple adaptions of Proposition~\ref{prop:def antecedent and bsp ext image} and of Proposition~\ref{prop:bispecial has only tree bispecial extensions} tell us that, as $\bar{\alpha}_{a_1}$ is strongly right proper for the letter $a_1$ and the images of letters by $\bar{\alpha}_{a_1}$ have a unique longest common suffix and a unique longest common prefix, it suffices to prove that the words $w \in F_\sigma$ such that $|w|_{a_1} = 0$ are dendric. These words are the elements of $\{\varepsilon\} \cup \cA \setminus \{a_1\}$, of which only $\varepsilon$ is bispecial. The conclusion follows.

Let us now assume that $\sigma'$ is a strongly left proper morphism for the letter $\ell$ which satisfies conditions~\ref{item:S and P contain one element} and \ref{item:initial factors are dendric}. As $\cS(\sigma') = \{s_0\}$, for any letter $a \in \cA$, $as_0$ is suffix of some $\sigma'(b)$. In particular, for $a = \ell$, this implies that there exists $b \in \cA$ such that $\ell s_0 = \sigma'(b)$ and that, for any letter $c \ne b$, $\sigma'(c)$ is strictly longer than $\sigma'(b)$. Similarly, $p_0 \ell$ is prefix of some $\sigma(b') \ell$ thus $\sigma'(b') = p_0$ and, as $\sigma'(b')$ must be strictly shorter than any other $\sigma'(a)$, we obtain $b = b'$ and $p_0 = \ell s_0$.
We can thus assume that $\sigma' = \sigma \circ \pi$ where $\pi$ is a permutation of $\cA$ such that $\ell s_0 a$ is a prefix of $\sigma(a) \ell$ for all $a \in \cA$. In particular, $\sigma(\ell) = \ell s_0$. We have
\[
	F_\sigma = F_{\sigma'}.
\]
By construction, for any prefix $u$ of $s_0 \ell$ and any suffix $v$ of $\ell s_0$,
\[
	E_{F_\sigma}^-(u) = \cA \quad \text{and} \quad E_{F_\sigma}^+ = \cA.
\]
In particular, if $s_0 \ell \in a \cA^*$ and $\ell s_0 \in \cA^* b$, then
\[
	E_{F_\sigma}(\varepsilon) = (\cA \times \{a\}) \cup (\{b\} \times \cA)
\]
because $\varepsilon$ is dendric in $E_{F_\sigma}$. Thus, any occurrence of $c \ne b$ in $F_\sigma$ can only be followed by an occurrence of $a$. Let us use this observation to prove that $\sigma = \bar{\alpha}_{a_1} \circ \dots \circ \bar{\alpha}_{a_n} \circ \alpha_\ell$ for some letters $a_1, \dots, a_n \in \cA \setminus \{\ell\}$.
If $s_0 = \varepsilon$, then $a = b = \ell$ thus $\sigma(c) = \ell c$ for all $c \in \cA \setminus \{\ell\}$ and $\sigma = \alpha_\ell$.
If $s_0$ is not empty, then $a$ is the first letter of $s_0$ and $b$ the last one. In particular, $a$ and $b$ cannot be the letter $\ell$ and, as $s_0\ell$ is an element of $F_\sigma$, $a$ cannot only be followed by occurrences of $a$ thus $a$ must be equal to $b$. For any letter $c \in \cA$, we know that $\sigma(c)$ begins with $\ell$ and that any letter $d \ne a$ in $\sigma(c)$ is followed by an $a$ thus
\[
	\sigma(c) \in \ell a \left(\{a\} \cup (\cA \setminus \{\ell\}) a\right)^*.
\]
Let us define the morphism $\tau$ such that
\[
	\sigma(c) = \bar{\alpha}_a \circ \tau(c).
\]
This morphism is unique and $\tau(c)$ is obtained by removing an occurrence of $a$ after each letter $d \ne a$ in $\sigma(c)$.
By construction, $\tau$ is injective and strongly left proper for the letter $\ell$. In addition, $s_0$ begins and ends with the letter $a$ thus there exists $s_0' \in \cA^*$ such that
\[
	a \bar{\alpha}_a (s_0') = s_0.
\]
It is easy to check that, as $\ell s_0 c$ is a prefix of $\sigma(c) \ell$, $\ell s_0' c$ is a prefix of $\tau(c) \ell$ and that, as $c s_0$ is a suffix of $\sigma(c)$, $c s_0'$ is a suffix of $\tau(c)$ for all $c \in \cA$. Thus, $\cS(\tau)$ and $\cP(\tau)$ both contain only one element and $\tau$ satisfies the condition~\ref{item:S and P contain one element}.
In addition, for all $w \in F_\tau$ and all $c, d \in \cA$
\[
	cwd \in F_\tau \Leftrightarrow ca \bar{\alpha}_a(w) d \in F_\sigma.
\]
Indeed, this equivalence is direct if $c \ne a$ and, for $c = a$, it derives from the fact that $a \ne \ell$ thus, if $awd \in F_\tau$, then there exists $c'$ such that $c'awd \in F_\tau$.
As a consequence, the extension graph of $w$ in $F_\tau$ is the same as the extension graph of $a \bar{\alpha}_a(w)$ in $F_\sigma$ and $\tau$ satisfies the condition~\ref{item:initial factors are dendric}.
By construction, we have $|s_0'| < |s_0|$ thus we can conclude by iterating the proof on $\tau$.
\end{proof}

%########################
%########################
\section{The case of ternary minimal dendric shifts}
\label{section:ternary case}
%########################
%########################

In Section~\ref{section:s-adicity}, we showed that any minimal dendric shift $X$ over the alphabet $\cA$ is $\cS$-adic with $\cS$ a set of tame automorphisms of $F_\cA$, a directive sequence of $X$ being given by Theorem~\ref{thm:S-adic representation of minimal}.
In this section, we give an $\cS$-adic characterization of minimal dendric shifts over the alphabet $\cA_3 = \{1,2,3\}$.
More precisely, we strengthen Theorem~\ref{thm:S-adic representation of minimal} by exhibiting a set $\cS$ (several choices are possible) and a subset $D \subset \cS^\NN$ such that a ternary subshift is minimal dendric if and only if it has an $\cS$-adic representation in $D$.

%########################
\subsection{Return morphisms in the ternary case}
\label{subsection:dendric morphisms Ster}
%########################

Let us start with an example. 
Assume that $X$ is a minimal dendric shift over $\cA_3$ and that the extension graph of $\varepsilon$ in $X$ is 
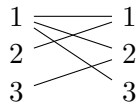
\begin{center}
\begin{tikzpicture}
\node (1L) at (0,0) {$1$};
\node (2L) [below of = 1L, node distance = .5cm] {$2$};
\node (3L) [below of = 2L, node distance = .5cm] {$3$};

\node (1R) [right of = 1L, node distance = 1.5cm] {$1$};
\node (2R) [below of = 1R, node distance = .5cm] {$2$};
\node (3R) [below of = 2R, node distance = .5cm] {$3$};

\draw (1L) -- (1R);
\draw (1L) -- (2R);
\draw (1L) -- (3R);

\draw (2L) -- (1R);
\draw (3L) -- (2R);
\end{tikzpicture}
\end{center}
The associated Rauzy graph $G_1(X)$ is 
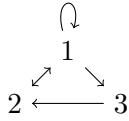
\begin{center}
\begin{tikzpicture}
\node (1) at (0,0) {$1$};
\node (2) [below left of = 1, node distance = 1cm] {$2$};
\node (3) [below right of = 1, node distance = 1cm] {$3$};

\draw [->] (1) edge[loop above] (1);
\draw [<->] (1) edge (2);
\draw [->] (1) edge (3);
\draw [->] (3) edge (2);
\end{tikzpicture}
\end{center}
From it, we deduce that 
\begin{itemize}
\item 
the return words to $1$ are $1$, $12$ and $132$;
\item
the return words to $2$ are of the form $21^k$ or $21^k3$ with $k \geq 1$;
\item
the return words to $3$ belong to $3(21^+)^+$.
\end{itemize}
An additional restriction concerning the powers of $1$ occurring in the return words to $2$ can be deduced from the fact that $X$ is dendric.
We claim that if $21^k$ is a return word to $2$, then the other return words cannot be of the form $21^\ell$ for some $\ell \geq k+2$.
Indeed, if both $21^k$ and $21^{\ell}$, $k \geq 1$, $\ell \geq k+2$, are return words, then by definition of return words, the words $21^k2, 21^\ell 2$ belong to $\cL(X)$. 
This implies that there is a cycle in the extension graph of $1^{k}$, contradicting the fact that $X$ is dendric.
In addition, the third return word cannot be of the form $21^n3$ with $n \geq \min\{k, \ell\} + 2$ for the same reason.
Similarly, if $21^k3$ and $21^\ell 3$ are return word to $2$ then $|k - \ell| \leq 1$ and the third return word is $21^n$ with $n \leq \min\{k, \ell\} + 1$.
Therefore, the set of return words to $2$ is one of the following for some $k \geq 1$ and some $1 \leq \ell \leq k+1$:
\[
	\{21^k,21^{k+1},21^\ell 3\},
	\
	\{21^\ell, 21^k 3,21^{k+1}3\}.
\]

Return words to $3$ are less easily described. 
Since $\cR(3) \subset 3(21^+)^+$ and $\#(\cR(3)) = 3$, the set $\cR(3)$ is determined by three sequences $(k^{(j)}_i)_{1 \leq i \leq n_j}$, $j \in \{1,2,3\}$ such that 
\[
	\cR(3) = \{321^{k^{(j)}_1}21^{k^{(j)}_2} \cdots 21^{k^{(j)}_{n_j}} \mid j \in \{1,2,3\}\}.
\]
Similar arguments show that there exist inequality constraints between the $k^{(j)}_i$,
but precisely describing the three sequences $(k^{(j)}_i)_{1 \leq i \leq n_j}$, $j \in \{1,2,3\}$, is much more tricky.
The main reason for this difference is that the letter $3$ is not left special in $X$.
If $u$ is the smallest left special factor having $3$ as a suffix, then writing $u = v3$, we have $v \cR(3) = \cR(u)v$. 
To better understand the possible sequences $(k^{(j)}_i)_{1 \leq i \leq n_j}$, we thus need the Rauzy graph of order $|u|$ of $X$ and not just $G_1(X)$.

With the notation of Theorem~\ref{thm:S-adic representation of minimal}, any choice of sequence of letters $(a_n)_{n \geq 1}$ leads to a directive sequence of $X$.
Consequently, in the sequel we will only consider return words to left special letters with the ``simplest'' return words.
In other words, if the extension graph of the empty word in $X$ is as in the previous example, we will only consider the coding morphisms associated with the left special letter $1$.

Up to a permutation on $\cA_3$, the possible extension graphs of the empty word for minimal dendric shifts on $\cA_3$ are given in Figures~\ref{fig:extension graph of epsilon} and \ref{fig:extension graph of epsilon2}.
They must satisfy two conditions: $\cE(\varepsilon)$ must be a tree and the associated Rauzy graph $G_1(X)$ must be strongly connected (by minimality of $X$). 
We always assume that $1$ is a left special letter and we present these associated Rauzy graph of order 1 as well as coding morphisms associated with $\cR(1)$.
Whenever some power appear in an image, we always have $k \geq 1$.
The reason why we only have $k$ and $k+1$ as exponent is the same as in the previous example: a bigger difference would contradict dendricity by inducing a cycle in some extension graph. 
We denote the set
\[
	\cSD = \{\alpha,\beta,\gamma,\eta\} \cup \{\delta^{(k)},\zeta^{(k)} \mid k \geq 1\}
\]
of morphisms as defined in Figures~\ref{fig:extension graph of epsilon} and \ref{fig:extension graph of epsilon2}.

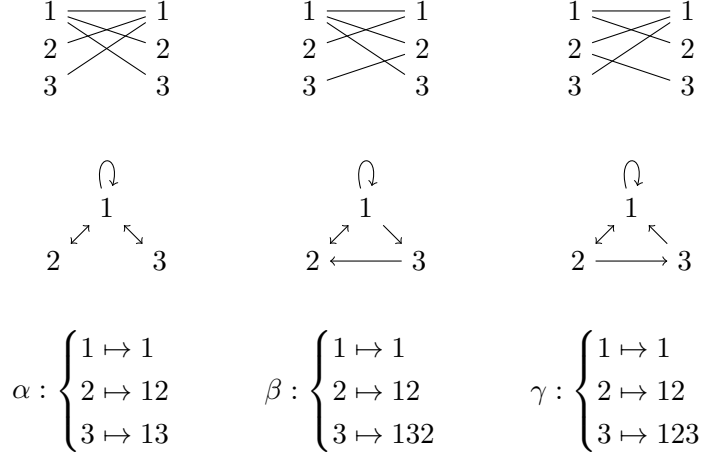
\begin{figure}[h]
\centering
\begin{tabular}{ccc}
\begin{tikzpicture}
\node (1L) at (0,0) {$1$};
\node (2L) [below of = 1L, node distance = .5cm] {$2$};
\node (3L) [below of = 2L, node distance = .5cm] {$3$};

\node (1R) [right of = 1L, node distance = 1.5cm] {$1$};
\node (2R) [below of = 1R, node distance = .5cm] {$2$};
\node (3R) [below of = 2R, node distance = .5cm] {$3$};

\draw (1L) -- (1R);
\draw (1L) -- (2R);
\draw (1L) -- (3R);
\draw (2L) -- (1R);
\draw (3L) -- (1R);
\end{tikzpicture}
&
\begin{tikzpicture}
\node (1L) at (0,0) {$1$};
\node (2L) [below of = 1L, node distance = .5cm] {$2$};
\node (3L) [below of = 2L, node distance = .5cm] {$3$};

\node (1R) [right of = 1L, node distance = 1.5cm] {$1$};
\node (2R) [below of = 1R, node distance = .5cm] {$2$};
\node (3R) [below of = 2R, node distance = .5cm] {$3$};

\draw (1L) -- (1R);
\draw (1L) -- (2R);
\draw (1L) -- (3R);
\draw (2L) -- (1R);
\draw (3L) -- (2R);
\end{tikzpicture}
&
\begin{tikzpicture}
\node (1L) at (0,0) {$1$};
\node (2L) [below of = 1L, node distance = .5cm] {$2$};
\node (3L) [below of = 2L, node distance = .5cm] {$3$};

\node (1R) [right of = 1L, node distance = 1.5cm] {$1$};
\node (2R) [below of = 1R, node distance = .5cm] {$2$};
\node (3R) [below of = 2R, node distance = .5cm] {$3$};

\draw (1L) -- (1R);
\draw (1L) -- (2R);
\draw (2L) -- (3R);
\draw (2L) -- (1R);
\draw (3L) -- (1R);
\end{tikzpicture}
\\ \\
\begin{tikzpicture}
\node (1) at (0,0) {$1$};
\node (2) [below left of = 1, node distance = 1cm] {$2$};
\node (3) [below right of = 1, node distance = 1cm] {$3$};

\draw [->] (1) edge[loop above] (1);
\draw [<->] (1) edge (2);
\draw [<->] (1) edge (3);
\end{tikzpicture}
&
\begin{tikzpicture}
\node (1) at (0,0) {$1$};
\node (2) [below left of = 1, node distance = 1cm] {$2$};
\node (3) [below right of = 1, node distance = 1cm] {$3$};

\draw [->] (1) edge[loop above] (1);
\draw [<->] (1) edge (2);
\draw [->] (1) edge (3);
\draw [->] (3) edge (2);
\end{tikzpicture}
&
\begin{tikzpicture}
\node (1) at (0,0) {$1$};
\node (2) [below left of = 1, node distance = 1cm] {$2$};
\node (3) [below right of = 1, node distance = 1cm] {$3$};

\draw [->] (1) edge[loop above] (1);
\draw [<->] (1) edge (2);
\draw [->] (2) edge (3);
\draw [->] (3) edge (1);
\end{tikzpicture}
\\ \\
\begin{tabular}[t]{l}
$\alpha:
\begin{cases}
	1 \mapsto 1 	\\ 
	2 \mapsto 12	\\
	3 \mapsto 13
\end{cases}
$
\end{tabular}
&
\begin{tabular}[t]{l}
$\beta:
\begin{cases}
	1 \mapsto 1 	\\ 
	2 \mapsto 12	\\
	3 \mapsto 132
\end{cases}
$
%\\
%$
%\begin{cases}
%	a \mapsto ba^{k+1} 	\\ 
%	b \mapsto ba^{k}	\\
%	c \mapsto ba^{k}c
%\end{cases}
%$
%\\
%$
%\begin{cases}
%	a \mapsto ba^{k+1} 	\\ 
%	b \mapsto ba^{k}	\\
%	c \mapsto ba^{k+1}c
%\end{cases}
%$
%\\
%$
%\begin{cases}
%	a \mapsto ba^{k+1}c 	\\ 
%	b \mapsto ba^{k}	\\
%	c \mapsto ba^{k}c
%\end{cases}
%$
%\\
%$
%\begin{cases}
%	a \mapsto ba^{k+1}c 	\\ 
%	b \mapsto ba^{k+1}	\\
%	c \mapsto ba^{k}c
%\end{cases}
%$
\end{tabular}
&
\begin{tabular}[t]{l}
$\gamma:
\begin{cases}
	1 \mapsto 1 	\\ 
	2 \mapsto 12	\\
	3 \mapsto 123
\end{cases}
$
%\\
%$
%\begin{cases}
%	a \mapsto ba^{k+1} 	\\ 
%	b \mapsto ba^{k}	\\
%	c \mapsto bca^{k}
%\end{cases}
%$
%\\
%$
%\begin{cases}
%	a \mapsto ba^{k+1} 	\\ 
%	b \mapsto ba^{k}	\\
%	c \mapsto bca^{k+1}
%\end{cases}
%$
%\\
%$
%\begin{cases}
%	a \mapsto bca^{k+1} 	\\ 
%	b \mapsto ba^{k}	\\
%	c \mapsto bca^{k}
%\end{cases}
%$
%\\
%$
%\begin{cases}
%	a \mapsto bca^{k+1} 	\\ 
%	b \mapsto ba^{k+1}	\\
%	c \mapsto bca^{k}
%\end{cases}
%$
\end{tabular}
\end{tabular}
\caption{The cases with a unique left special letter and/or a unique right special letter}
\label{fig:extension graph of epsilon}
\end{figure}

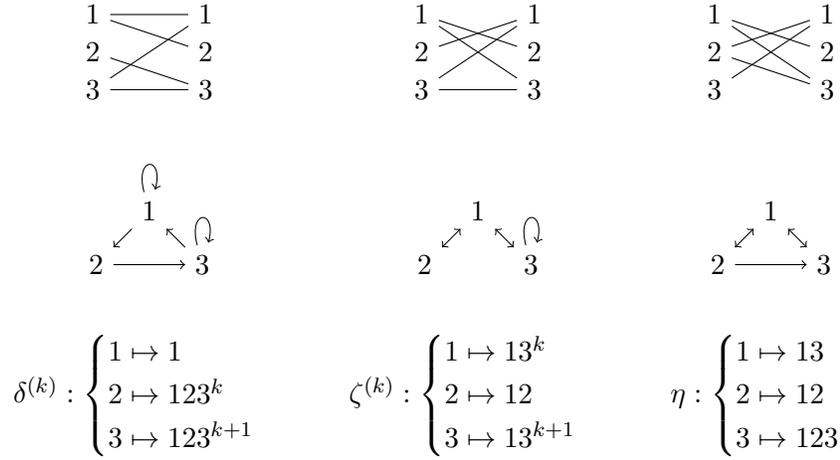
\begin{figure}[h]
\centering
\begin{tabular}{ccc}
\begin{tikzpicture}
\node (1L) at (0,0) {$1$};
\node (2L) [below of = 1L, node distance = .5cm] {$2$};
\node (3L) [below of = 2L, node distance = .5cm] {$3$};

\node (1R) [right of = 1L, node distance = 1.5cm] {$1$};
\node (2R) [below of = 1R, node distance = .5cm] {$2$};
\node (3R) [below of = 2R, node distance = .5cm] {$3$};

\draw (1L) -- (1R);
\draw (1L) -- (2R);
\draw (2L) -- (3R);
\draw (3L) -- (1R);
\draw (3L) -- (3R);
\end{tikzpicture}
&
\begin{tikzpicture}
\node (1L) at (0,0) {$1$};
\node (2L) [below of = 1L, node distance = .5cm] {$2$};
\node (3L) [below of = 2L, node distance = .5cm] {$3$};

\node (1R) [right of = 1L, node distance = 1.5cm] {$1$};
\node (2R) [below of = 1R, node distance = .5cm] {$2$};
\node (3R) [below of = 2R, node distance = .5cm] {$3$};

\draw (1L) -- (2R);
\draw (1L) -- (3R);
\draw (2L) -- (1R);
\draw (3L) -- (1R);
\draw (3L) -- (3R);
\end{tikzpicture}
&
\begin{tikzpicture}
\node (1L) at (0,0) {$1$};
\node (2L) [below of = 1L, node distance = .5cm] {$2$};
\node (3L) [below of = 2L, node distance = .5cm] {$3$};

\node (1R) [right of = 1L, node distance = 1.5cm] {$1$};
\node (2R) [below of = 1R, node distance = .5cm] {$2$};
\node (3R) [below of = 2R, node distance = .5cm] {$3$};

\draw (1L) -- (2R);
\draw (1L) -- (3R);
\draw (2L) -- (1R);
\draw (2L) -- (3R);
\draw (3L) -- (1R);
\end{tikzpicture}
\\
\\
\begin{tikzpicture}
\node (1) at (0,0) {$1$};
\node (2) [below left of = 1, node distance = 1cm] {$2$};
\node (3) [below right of = 1, node distance = 1cm] {$3$};

\draw [->] (1) edge[loop above] (1);
\draw [->] (1) edge (2);
\draw [->] (2) edge (3);
\draw [->] (3) edge (1);
\draw [->] (3) edge[loop above] (3);
\end{tikzpicture}
&
\begin{tikzpicture}
\node (1) at (0,0) {$1$};
\node (2) [below left of = 1, node distance = 1cm] {$2$};
\node (3) [below right of = 1, node distance = 1cm] {$3$};

\draw [->] (3) edge[loop above] (3);
\draw [<->] (3) edge (1);
\draw [<->] (1) edge (2);
\end{tikzpicture}
&
\begin{tikzpicture}
\node (1) at (0,0) {$1$};
\node (2) [below left of = 1, node distance = 1cm] {$2$};
\node (3) [below right of = 1, node distance = 1cm] {$3$};

\draw [<->] (1) edge (2);
\draw [<->] (1) edge (3);
\draw [->] (2) edge (3);
\end{tikzpicture}
\\
\\
\begin{tabular}[t]{l}
$\delta^{(k)}:
\begin{cases}
	1 \mapsto 1 	\\ 
	2 \mapsto 123^{k}	\\
	3 \mapsto 123^{k+1}
\end{cases}
$
%\\
%$
%\begin{cases}
%	a \mapsto ca^{k+1}b 	\\ 
%	b \mapsto ca^{k}b	\\
%	c \mapsto c
%\end{cases}
%$
\end{tabular}
&
\begin{tabular}[t]{l}
%$
%\begin{cases}
%	a \mapsto a 	\\ 
%	b \mapsto a(bc)^{k}	\\
%	c \mapsto a(bc)^{k+1}
%\end{cases}
%$
%\\
$\zeta^{(k)}:
\begin{cases}
	1 \mapsto 13^k 	\\ 
	2 \mapsto 12	\\
	3 \mapsto 13^{k+1}
\end{cases}
$
\end{tabular}
&
\begin{tabular}[t]{l}
$\eta:
\begin{cases}
	1 \mapsto 13	\\ 
	2 \mapsto 12	\\
	3 \mapsto 123
\end{cases}
$
%\\
%$
%\begin{cases}
%	a \mapsto b(ca)^{k+1} 	\\ 
%	b \mapsto b(ca)^{k}	 \\
%	c \mapsto b(ca)^{k}c
%\end{cases}
%$
%\\
%$
%\begin{cases}
%	a \mapsto b(ca)^{k+1} 	\\ 
%	b \mapsto b(ca)^{k}	 \\
%	c \mapsto b(ca)^{k+1}c
%\end{cases}
%$
%\\
%$
%\begin{cases}
%	a \mapsto b(ca)^{k+1}c 	\\ 
%	b \mapsto b(ca)^{k}c	 \\
%	c \mapsto b(ca)^{k}
%\end{cases}
%$
%\\
%$
%\begin{cases}
%	a \mapsto b(ca)^{k}c 	\\ 
%	b \mapsto b(ca)^{k-1}c	 \\
%	c \mapsto b(ca)^{k}
%\end{cases}
%$
\end{tabular}
\end{tabular}
\caption{The cases with two left special letters and two right special letters}
\label{fig:extension graph of epsilon2}
\end{figure}

Let $\Sigma_3$ be the symmetric group on $\cA_3 = \{1, 2, 3\}$. 
If $\cA_3 = \{a, b, c\}$, we let $\pi_{abc} \in \Sigma_3$ denote the permutation $1 \mapsto a$, $2 \mapsto b$, $3 \mapsto c$.

\begin{proposition}
\label{prop:Ster-adic representation}
Any ternary minimal dendric shift $X$ has a primitive $\Sigma_3\cSD\Sigma_3$-adic representation $\boldsymbol{\sigma}$ and, for each such representation, $X_{\bsigma}^{(n)}$ is a ternary minimal dendric shift for each $n$.
\end{proposition}
\begin{proof}
For the existence, we consider the construction of a directive sequence following Theorem~\ref{thm:S-adic representation of minimal} where at each step, we choose the left special letter $a_n$ for which $\sigma_n$ is in $\Sigma_3\cSD\Sigma_3$. 
Using Theorem~\ref{thm:S-adic representation of minimal} and Theorem~\ref{thm:decoding dendric}, we obtain that for each $\Sigma_3\cSD\Sigma_3$-adic representation $\bsigma$ and each $n$, $X_{\bsigma}^{(n)}$ is a ternary minimal dendric shift.
\end{proof}

Note that the previous result is also true when considering $\Sigma_3\cSD$-adic representations but, for our results, $\Sigma_3\cSD\Sigma-3$-adic representations are more convenient.

While Proposition~\ref{prop:Ster-adic representation} deals with $\Sigma_3\cSD\Sigma_3$-adic representations, it is obvious that only the morphisms in $\cSD$ really matter.
In the next sections, we essentially focus on them and we involve permutations only when it is needed.

%########################
\subsection{Conditions for having only dendric bispecial extended images}
\label{subsection:action of morphisms ternary}
%########################

Assume that $X$ is a minimal shift space over $\cA$ and that $v \in \cL(X)$ is a bispecial factor and let $Y$ be the image of $X$ under some injective and strongly left proper morphism $\sigma: \cA^* \to \cB^*$. 
Recall from Section~\ref{subsection:stability of tree} (and, in particular, Proposition~\ref{prop:morphic image of extension graph}) that the extension graph of any bispecial extended image of $v$ under $\sigma$ is the image under two consecutive graph morphisms $\varphi^-_s$ and $\varphi^+_p$ of a subgraph of $\cE_X(v)$.
In this section, we determine those graph morphisms when $\sigma$ is a morphism in $\cSD$ and we give necessary and sufficient conditions on the extension graph of $v \in \cL(X)$ so that $v$ only has dendric bispecial extended images.
In this particular case, since the alphabet has cardinality 3, $\cT^+(\sigma)$ and $\cT^-(\sigma)$ have cardinality at most 2.
Tables~\ref{table:def des varphi1} and~\ref{table:def des varphi2} define the (possibly partial) maps $\varphi^-_s, \varphi^+_p: \cA_3 \to \cA_3$ associated with each morphism $\sigma \in \cSD$.

\begin{table}[h]
\centering
$\begin{array}{|c|c|c|c|c|}
\hline
\sigma & \cT^-(\sigma) & \varphi^-_s & \cT^+(\sigma) & \varphi^+_p	\\
\hline
\alpha	&	\{\varepsilon\}	& \id	& \{1\} & \id 	\\
\hline
\beta	&	\{\varepsilon,2\}	&	
\varphi^-_\varepsilon :
\begin{cases}
	1   \mapsto 1	\\
	2,3 \mapsto 2	
\end{cases}
 & \{1\} & 	\id \\
& & 
\varphi^-_2:
\begin{cases}
	2 \mapsto 1	\\
	3 \mapsto 3 
\end{cases}
& & \\
\hline
\gamma	&	\{\varepsilon\}	& \id & \{1,12\} & 
\varphi^+_{1}:
\begin{cases}
	1 	 \mapsto 1	\\
	2,3  \mapsto 2	
\end{cases}
	\\
& & & & 
\varphi^+_{12}:
\begin{cases}
	2 \mapsto 1	\\
	3 \mapsto 3 
\end{cases}
\\
\hline
\end{array}
$
\caption{Definition of the graph morphisms $\varphi^-_s$ and $\varphi^+_p$ associated with the morphisms $\alpha$, $\beta$ and $\gamma$}
\label{table:def des varphi1}
\end{table}

\begin{table}[h]
\centering
$\begin{array}{|c|c|c|c|c|}
\hline
\sigma & \cT^-(\sigma) & \varphi_s & \cT^+(\sigma) & \varphi_p	\\
\hline
\delta^{(k)}	&	\{\varepsilon,3^k\}	&  
\varphi^-_{\varepsilon}:
\begin{cases}
	1   \mapsto 1	\\
	2,3 \mapsto 3	
\end{cases}
& \{1,123^k\} & 
\varphi^+_{1}:
\begin{cases}
	1   \mapsto 1	\\
	2,3 \mapsto 2	
\end{cases}
	\\
& & 
\varphi^-_{3^k}:
\begin{cases}
	2 \mapsto 2	\\
	3 \mapsto 3 
\end{cases}
& &
\varphi^+_{123^k}:
\begin{cases}
	2 \mapsto 1	\\
	3 \mapsto 3 
\end{cases}
\\
\hline
\zeta^{(k)}	&	\{\varepsilon,3^k\}	&  
\varphi^-_{\varepsilon}:
\begin{cases}
	1,3 \mapsto 3   \\
	2   \mapsto 2
\end{cases}
& \{1,13^k\} & 
\varphi^+_{1}:
\begin{cases}
	1,3 \mapsto 3   \\
	2   \mapsto 2
\end{cases}
	\\
& & 
\varphi^-_{3^k}:
\begin{cases}
	1 \mapsto 1	\\
	3 \mapsto 3 
\end{cases}
 & &
\varphi^+_{13^k}:
\begin{cases}
	1 \mapsto 1	\\
	3 \mapsto 3 
\end{cases}
\\
\hline
\eta	&	\{\varepsilon,3\}	&  
\varphi^-_{\varepsilon}:
\begin{cases}
	1,3 \mapsto 3   \\
	2   \mapsto 2
\end{cases}
& \{1,12\} & 
\varphi^+_{1}:
\begin{cases}
	1   \mapsto 3	\\
	2,3 \mapsto 2 
\end{cases}
	\\
& & 
\varphi^-_{3}:
\begin{cases}
	1 \mapsto 1	\\
	3 \mapsto 2 
\end{cases}
& & 
\varphi^+_{12}:
\begin{cases}
	2 \mapsto 1	\\
	3 \mapsto 3	\\
\end{cases}
\\
\hline
\end{array}
$
\caption{Definition of the graph morphisms $\varphi^-_s$ and $\varphi^+_p$ associated with the morphisms $\delta^{(k)}$, $\zeta^{(k)}$ and $\eta$}
\label{table:def des varphi2}
\end{table}

A direct application of Proposition~\ref{prop:bispecial has only tree bispecial extensions} shows that whenever $v$ is dendric, then $v$ has only dendric bispecial extended images if and only if the following conditions are satisfied:
\begin{enumerate}
\item
either $\cT^-_v(\sigma) = \{s_0\}$, or both $\cT^-_v(\sigma) = \{s_0,s\}$ and $\cE_{X,s,p_0}(v)$ is a tree;
\item
either $\cT^+_v(\sigma) = \{p_0\}$, or both $\cT^+_v(\sigma) = \{p_0,p\}$ and $\cE_{X,s_0,p}(v)$ is a tree.
\end{enumerate}

We first give a handier interpretation of these conditions.
Observe that for convenience, we actually characterize the dendric bispecial factors $v \in \cL(X)$ that have a non-dendric bispecial extended image. 
When considering a letter $a \in \cA_3$ as a vertex of $\cE(v)$, we respectively write $a^-$ or $a^+$ to emphasize that $a$ is considered as a left or right vertex.

For $v \in \cL(X)$, we define $\cC_X^-(v)$ (resp., $\cC_X^+(v)$) as the set of letters $a \in \cA$ such that the subgraph of $\cE_X(v)$ obtained by removing the vertex $a^-$ (resp., $a^+$) and all the induced isolated vertices (if any) is not connected. When the context is clear, the subscript $X$ will be omitted.

\begin{remark}
\label{remark:interpretations of C^- and C^+}
If $v \in \cL(X)$ is a dendric factor, then $a \in \cC^-(v)$ if and only if $a^-$ has at least two neighbors that are not leaves, i.e., that have degree at least 2. In particular, $v$ is bispecial.
Observe also that as $\cE(v)$ is a tree, this implies that the left side of $\cE(v)$ contains three vertices.
Hence, another equivalent condition when $\cE(v)$ is a tree is that, writing $\cA_3 = \{a,b,c\}$, the path from $b^-$ to $c^-$ has length 4.
\end{remark}

\begin{proposition}
\label{prop:condition for non-tree-successor}
Let $X$ be a shift space over $\cA_3$ and $\sigma$ be a morphism in $\cSD$.

If $v \in \cL(X)$ is a dendric bispecial factor, then $v$ has a non-dendric bispecial extended image under $\sigma$ if and only if one of the following conditions is satisfied:
\begin{enumerate}
\item
$1 \in \cC_X^-(v)$ and $\sigma \in \{\beta,\delta^{(k)} \mid k \geq 1\}$;
\item
$2 \in \cC_X^-(v)$ and $\sigma \in \{\zeta^{(k)},\eta \mid k \geq 1\}$;
\item
$1 \in \cC_X^+(v)$ and $\sigma \in \{\gamma,\delta^{(k)},\eta \mid k \geq 1\}$;
\item
$2 \in \cC_X^+(v)$ and $\sigma \in \{\zeta^{(k)} \mid k \geq 1\}$.
\end{enumerate}
\end{proposition}

\begin{proof}
The negation of item~\ref{item 1 tree extension} of Proposition~\ref{prop:bispecial has only tree bispecial extensions} is equivalent to ``there exists $a \in E^-_X(v)$ such that $\cE_{X, s(b, c), p_0}(v)$ is not a tree''. 
As $\cE_{X, s(b, c), p_0}(v)$ is a subgraph of $\cE(v)$, it is acyclic and, if $s(b, c)$ is a suffix of $\sigma(a)$ it is also connected (indeed, we then have $\cE_{X, s(b, c), p_0}(v) = \cE_{X, s_0, p_0}(v) = \cE(v)$). 
Thus, the first condition of Proposition~\ref{prop:bispecial has only tree bispecial extensions} is not satisfied if and only if there exists a permutation $\{a, b, c\}$ of $\cA_3$ such that $s(b, c)$ is not a prefix of $\sigma(a)$ and $a \in \cC^-(v)$. 
Using Figures~\ref{fig:extension graph of epsilon} and~\ref{fig:extension graph of epsilon2}, we see that it is equivalent to condition 1 or 2.
We proceed in a similar way to show that the second condition of Proposition~\ref{prop:bispecial has only tree bispecial extensions} is not satisfied if and only if one of the conditions 3 and 4 above is.
\end{proof}

\begin{example}
Assume that $v$ is a dendric bispecial factor in some minimal ternary shift space $X$ with extension graph
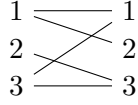
\begin{center}
\begin{tikzpicture}
\node (1L) at (0,0) {$1$};
\node (2L) [below of = 1L, node distance = .5cm] {$2$};
\node (3L) [below of = 2L, node distance = .5cm] {$3$};

\node (1R) [right of = 1L, node distance = 1.5cm] {$1$};
\node (2R) [below of = 1R, node distance = .5cm] {$2$};
\node (3R) [below of = 2R, node distance = .5cm] {$3$};

\draw (1L) -- (1R);
\draw (1L) -- (2R);
\draw (2L) -- (3R);

\draw (3L) -- (1R);
\draw (3L) -- (3R);
\end{tikzpicture}
\end{center}
Thus we have $3 \in \cC^-(v)$ and $1 \in \cC^+(v)$. 
If $Y_1$ is the image of $X$ under $\beta$, the bispecial extended images of $v$ in $Y_1$ are $u_1 = \beta(v)1$ and $u_2 = 2 \beta(v)1$ and they have the following extension graphs:
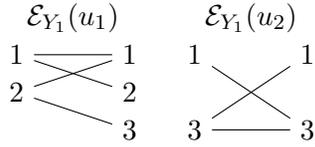
\begin{center}
\begin{tabular}[t]{cc}
$\cE_{Y_1}(u_1)$ & $\cE_{Y_1}(u_2)$ \\ 
\begin{tikzpicture}
\node (1L) at (0,0) {$1$};
\node (2L) [below of = 1L, node distance = .5cm] {$2$};

\node (1R) [right of = 1L, node distance = 1.5cm] {$1$};
\node (2R) [below of = 1R, node distance = .5cm] {$2$};
\node (3R) [below of = 2R, node distance = .5cm] {$3$};

\draw (1L) -- (1R);
\draw (1L) -- (2R);
\draw (2L) -- (3R);

\draw (2L) -- (1R);
\end{tikzpicture}
&
\begin{tikzpicture}
\node (1L) at (0,0) {1};
\node (3L) [below of = 1L, node distance = 1cm] {$3$};

\node (1R) [right of = 1L, node distance = 1.5cm] {$1$};
\node (3R) [below of = 1R, node distance = 1cm] {$3$};

\draw (1L) -- (3R);
\draw (3L) -- (1R);
\draw (3L) -- (3R);
\end{tikzpicture}
\end{tabular}
\end{center}
Similarly, if $Y_2$ is the image of $X$ under $\gamma$, the bispecial extended images of $v$ in $Y_2$ are $w_1 = \gamma(v)1$ and $w_2 = \gamma(v)12$ and they have the following extension graphs: 
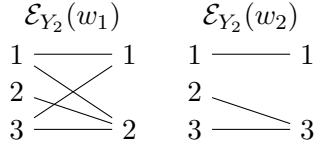
\begin{center}
\begin{tabular}[t]{cc}
$\cE_{Y_2}(w_1)$ & $\cE_{Y_2}(w_2)$ \\ 
\begin{tikzpicture}
\node (1L) at (0,0) {$1$};
\node (2L) [below of = 1L, node distance = .5cm] {$2$};
\node (3L) [below of = 2L, node distance = .5cm] {$3$};

\node (1R) [right of = 1L, node distance = 1.5cm] {$1$};
\node (2R) [below of = 1R, node distance = 1cm] {$2$};

\draw (1L) -- (1R);
\draw (1L) -- (2R);
\draw (2L) -- (2R);
\draw (3L) -- (1R);
\draw (3L) -- (2R);
\end{tikzpicture}
&
\begin{tikzpicture}
\node (1L) at (0,0) {$1$};
\node (2L) [below of = 1L, node distance = .5cm] {$2$};
\node (3L) [below of = 2L, node distance = .5cm] {$3$};

\node (1R) [right of = 1L, node distance = 1.5cm] {$1$};
\node (3R) [below of = 1R, node distance = 1cm] {$3$};

\draw (1L) -- (1R);
\draw (2L) -- (3R);
\draw (3L) -- (3R);
\end{tikzpicture}
\end{tabular}
\end{center}

\end{example}

\subsection{Ternary dendric preserving morphisms}

Assuming that $v \in \cL(X)$ is a dendric bispecial factor, Proposition~\ref{prop:condition for non-tree-successor} characterizes under which conditions $v$ has a non-dendric bispecial extended image under $\sigma \in \cSD$ or, in other words, under which conditions $\sigma \in \cSD$ is not dendric preserving for $v$.
As we consider the ternary case, we denote by $\DP(v)$ the set of dendric preserving morphisms for $v$ in $\cSD$.

When $X$ is a ternary dendric shift, we extend the notations $\cC^-$ and $\cC^+$ and set
\begin{align*}
	\cC^-(X) & = \bigcup_{v \in \cL(X)} \cC^-_X(v);	\\
	\cC^+(X) & = \bigcup_{v \in \cL(X)} \cC^+_X(v);	\\
	\DP(X) &= \bigcap_{v \in \cL(X)} \DP(v).
\end{align*}

Using Proposition~\ref{prop:condition for non-tree-successor}, the sets $\cC^-(X)$ and $\cC^+(X)$ completely determine the set $\DP(X)$ of all morphisms in $\cSD$ that are dendric preserving for $X$.
In this section, we in particular show that $\cC^-(X)$ and $\cC^+(X)$ contain at most one letter and we show that, when $Y$ is the image of $X$ under $\sigma \in \DP(X)$, $\cC^-(Y)$ (resp., $\cC^+(Y)$) is completely determined by $\cC^-(X)$ (resp., $\cC^+(X)$) and $\sigma$.
The next lemma is a trivial consequence of Remark~\ref{remark:interpretations of C^- and C^+}.

\begin{lemma}
\label{lemma:AL Ar at most one letter}
Let $X$ be a shift space over $\cA_3$. For every dendric bispecial factor $v \in \cL(X)$, $\cC^-_X(v)$ (resp., $\cC^+_X(v)$) contains at most one letter. 
\end{lemma}

\begin{lemma}
\label{lemma::AL et AR du mot vide}
Let $X$ be a shift space over $\cA_3$ which is the image under $\sigma \in \cSD$ of another shift space $Z$ over $\cA_3$.
The sets $\cC_X^-(\varepsilon)$ and $\cC_X^+(\varepsilon)$ are given in Table~\ref{table:AL et AR du mot vide}.
\begin{table}[h]
\centering
$\begin{array}{|c|c|c|c|c|c|c|}
\hline
\sigma & \alpha & \beta & \gamma & \delta^{(k)} & \zeta^{(k)} & \eta	\\
\hline
\cC^-(\varepsilon) & \emptyset & \{1\} & \emptyset & \{3\} & \{3\} & \{2\} 
\\
\hline
\cC^+(\varepsilon)& \emptyset & \emptyset & \{1\} & \{1\} & \{3\} & \{3\}
\\
\hline
\end{array}
$
\caption{Sets $\cC_X^-(\varepsilon)$ and $\cC_X^+(\varepsilon)$ whenever $X$ is the image under $\sigma$ of a shift space over $\cA_3$.}
\label{table:AL et AR du mot vide}
\end{table}
\end{lemma}
\begin{proof}
Indeed, the morphism $\sigma$ completely determines the extension graph $\cE_X(\varepsilon)$.
The result thus directly follows from the definition of $\cC_X^-(\varepsilon)$ and $\cC_X^+(\varepsilon)$.
\end{proof}

\begin{lemma}
\label{lemma:DP implies dendric}
If $X$ is a ternary dendric shift and if $Y$ is the image of $X$ under some morphism $\sigma \in \cSD$, then $Y$ is dendric if and only if $\sigma \in \DP(X)$.
\end{lemma}
\begin{proof}
Indeed, if $\sigma \in \DP(X)$, every bispecial extended image of a bispecial factor of $X$ is dendric by definition of $\DP(X)$.
Any other bispecial factor of $Y$ is the empty word or a non-prefix factor of an image $\sigma(a)$, $a \in \cA_3$ (by Proposition~\ref{prop:def antecedent and bsp ext image}).
It suffices to check that any such bispecial factor is dendric when $\sigma$ belongs to $\cSD$ to prove that $Y$ is dendric.

Assume now that $Y$ is dendric. If $\sigma$ is not in $\DP(X)$ then there exists $v \in \cL(X)$ such that $\sigma \notin \DP(v)$ thus $v$ has an extended image in $Y$ which is not dendric.
\end{proof}

We say that a morphism $\sigma \in \cSD$ is {\em left-invariant} (resp., {\em right-invariant}) if $\cT^-(\sigma)$ (resp., $\cT^+(\sigma)$) is a singleton, i.e. $\cT^-(\sigma) = \{s_0\}$ (resp., $\cT^+(\sigma) = \{p_0\}$).
The next lemma directly follows from the definition of the morphisms in $\cSD$.

\begin{lemma}
\label{lemma:which ones are left,right invariant}
\begin{enumerate}
\item
$\alpha$ is both left-invariant and right-invariant;
\item
$\beta$ is right-invariant, but not left-invariant;
\item
$\gamma$ is left-invariant, but not right-invariant;
\item
$\delta^{(k)}$, $\zeta^{(k)}$ and $\eta$ neither are left-invariant, nor right-invariant.
\end{enumerate}
\end{lemma}

We let $\cSL$ and $\cSR$ respectively denote the left-invariant and right-invariant morphisms, i.e., 
\[
\cSL = \{\alpha,\gamma\}
\quad \text{and} \quad
\cSR = \{\alpha,\beta\}.
\]
Observe that if $\sigma \in \cSL$ (resp., $\sigma \in \cSR$), then the associated graph morphism $\varphi^-_{s_0}$ (resp., $\varphi^+_{p_0}$) is the identity.
Moreover, from Table~\ref{table:AL et AR du mot vide}, a morphism $\sigma$ belongs to $\cSL$ (resp., to $\cSR$) if and only if $\cC_X^-(\varepsilon)$ (resp., $\cC_X^+(\varepsilon)$) is empty, where $X$ is the image under $\sigma$ of a shift over $\cA_3$.

\begin{lemma}
\label{lemma:non-left-invariant}
Let $X$ be a shift space over $\cA_3$, $\sigma \in \cSD$ a non left-invariant (resp., non right-invariant) morphism and $Y$ the image of $X$ under $\sigma$.
If $v \in \cL(X)$ is a dendric bispecial factor, then any dendric extended image $u$ of $v$ is such that $\cC^-_Y(u) = \emptyset$ (resp., $\cC^+_Y(u) = \emptyset$).

In particular, if $X$ is dendric and $\sigma$ is in $\DP(X)$, then $\cC^-(Y) = \cC^-_Y(\varepsilon) \ne \emptyset$ (resp., $\cC^+(Y) = \cC^+_Y(\varepsilon) \ne \emptyset$).
\end{lemma}

\begin{proof}
Let us show the result for a non-left-invariant morphism, the other case being symmetric. 
By definition of left-invariance, $\cT^-(\sigma)$ contains two elements $s_0$ and $s_1$.
It suffices to check in Tables~\ref{table:def des varphi1} and~\ref{table:def des varphi2} that for each of them, the range of the associated graph morphism $\varphi^-_{s_i}$ has cardinality 2. 
As the left vertices of $\cE_Y(u)$ are images of the left vertices of $\cE_X(v)$, $\cE_Y(u)$ contains at most two left vertices.
By Remark~\ref{remark:interpretations of C^- and C^+}, $C^-_Y(u)$ is empty.

Now assume that $X$ is dendric and that $\sigma$ belongs to $\DP(X)$.
By Lemma~\ref{lemma:DP implies dendric}, $Y$ is a dendric shift.
Let $u$ be a non-empty factor of $Y$. 
By Proposition~\ref{prop:def antecedent and bsp ext image}, either $u$ is a non-prefix factor of $\sigma(a)$ for some letter $a \in \cA_3$, or $u$ is an extended image of a factor $v \in \cL(X)$.
In the first case, it suffices to check that $\#(E_Y^-(u)) \leq 2$, which implies that $\cC^-_Y(u) = \emptyset$.
In the second case, as $X$ is dendric, $v$ is also dendric so by the first part of the lemma, $\cC^-_Y(u) = \emptyset$. Thus $\cC^-(Y) = \cC^-_Y(\varepsilon)$ and, by Lemma~\ref{lemma::AL et AR du mot vide}, it is non-empty.
\end{proof}

\begin{lemma}
\label{lemma:left-invariant}
Let $X$ be a shift space over $\cA_3$, $\sigma \in \cSD$ a left-invariant (resp., right-invariant) morphism and $Y$ the image of $X$ under $\sigma$.
If $v \in \cL(X)$ is a dendric bispecial factor, then \begin{align}
\label{eq:al(v)}
	\cC^-_X(v) &= \bigcup_{u \text{ bispecial extended image of } v} \cC^-_Y(u) \\
	\text{(resp., } \cC^+_X(v) &= \bigcup_{u \text{ bispecial extended image of } v} \cC^+_Y(u)\text{)}. 
	\nonumber
\end{align}

In particular, if $X$ is dendric and $\sigma$ is in $DP(X)$, then $\cC^-(Y) = \cC^-(X)$ (resp., $\cC^+(Y) = \cC^+(X)$).
\end{lemma}

\begin{proof}
Let us assume that $\sigma$ is left-invariant, the other case is symmetric.
We first show that if $u$ is a bispecial extended image of $v$, then $\cC^-_Y(u) \subset \cC^-_X(v)$. 
As $\cT^-(\sigma) = \{s_0\}$, there exists $p \in \cT^+(\sigma)$ such that $u = s_0 \sigma(v) p$.
Assume that $a \in \cC^-_Y(u)$.
Writing $\cA_3 = \{a,b,c\}$, Remark~\ref{remark:interpretations of C^- and C^+} states that the path $q$ from $b^-$ to $c^-$ in $\cE_Y(u)$ has length 4.
This path $q$ is the image under $\varphi_{s_0,p}$ of a path $q'$ of length at least 4 in $\cE_X(v)$.
As $\cE_X(v)$ is a tree with at most 6 vertices and the extremities of $q'$ are left vertices, the path has length exactly 4. 
As $\varphi^-_{s_0}$ is the identity, we conclude that $q'$ is a path of length 4 from $b^-$ to $c^-$ in $\cE_X(v)$, hence that $a \in \cC^-_X(v)$.

If $\cC^-_X(v) = \emptyset$, then Equality~\eqref{eq:al(v)} is direct.
Thus we only need to prove it when $\cC^-_X(v) \neq \emptyset$.
As $\sigma$ is left-invariant, we have by Lemma~\ref{lemma:which ones are left,right invariant} that $\sigma = \alpha$ or $\sigma = \gamma$.

If $\sigma = \alpha$, then as $\alpha$ is also right-invariant (see Lemma~\ref{lemma:which ones are left,right invariant}), Proposition~\ref{prop:morphic image of extension graph} implies that $u$ is the unique bispecial extended image of $v$ and that $\cE_Y(u) = \cE_X(v)$ (the graph morphism is the identity), hence that $\cC^-_Y(u) = \cC^-_X(v)$.

If $\sigma = \gamma$, then as $\cT^+(\sigma) = \{1,12\}$, (see Table~\ref{table:def des varphi1}), Corollary~\ref{cor:characterization of extended images} implies that $v$ has at most two bispecial extended images $u_1 = \sigma(v)1$ and $u_2 = \sigma(v)12$.
Let $\varphi = \varphi_{\varepsilon,1}$ and $\varphi' = \varphi_{\varepsilon,12}$. The morphism $\varphi^-_{\varepsilon}$ is the identity thus we have $\varphi = \varphi^+_{1}$ and $\varphi' = \varphi^+_{12}$.

As $\cC^-_X(v) \neq \emptyset$, by Lemma~\ref{lemma:AL Ar at most one letter} there is a letter $a \in \cA_3$ such that $\cC^-_X(v) = \{a\}$.
Using Remark~\ref{remark:interpretations of C^- and C^+}, the path $q$ from $b^-$ to $c^-$ has length 4 in $\cE_X(v)$.
Let us write $q = (b^-,x^+,a^-,y^+,c^-)$, with $x,y \in \cA_3$.

If $1 \in \{x,y\}$, we assume without loss of generality that $1 = x$. 
Then we have $\varphi(q) = (b^-,1^+,a^-,2^+,c^-)$ which is a path of length 4 from $b^-$ to $c^-$ in $\cE_Y(u_1)$.
By Remark~\ref{remark:interpretations of C^- and C^+} and Lemma~\ref{lemma:AL Ar at most one letter}, one has $\cC^-_Y(u_1) = \{a\}$.
As $\cC^-_Y(u_2) \subset \cC^-_X(v)$ by the first part of the proof, we get $\cC^-_Y(u_1) \cup \cC^-_Y(u_2) = \cC^-_X(v)$.  

If $\{x,y\} = \{2,3\}$, we assume without loss of generality that $(x,y) = (2,3)$.
Then we have $\varphi'(q) = (b^-,1^+,a^-,3^+,c^-)$ which is a path of length 4 from $b^-$ to $c^-$ in $\cE_Y(u_2)$.
By Remark~\ref{remark:interpretations of C^- and C^+} and Lemma~\ref{lemma:AL Ar at most one letter}, one has $\cC^-_Y(u_2) = \{a\}$.
As $\cC^-_Y(u_1) \subset \cC^-_X(v)$ by the first part of the proof, we also get $\cC^-_Y(u_1) \cup \cC^-_Y(u_2) = \cC^-_X(v)$.

Let us finally show that $\cC^-(Y) = \cC^-(X)$.
With $\sigma \in \{\alpha, \gamma\}$, the non-prefix factor of  $\sigma(a)$,  $a \in \cA_3$, are not bispecial. 
Hence, using Proposition~\ref{prop:def antecedent and bsp ext image}, a bispecial factor $u \in \cL(Y)$ is either empty, or a bispecial extended image of some bispecial factor $v \in \cL(X)$. 
As $\sigma$ is left-invariant, we have $\cC^-_Y(\varepsilon) = \emptyset$ by Lemma~\ref{lemma::AL et AR du mot vide}.
Using Equation~\eqref{eq:al(v)}, we get
\begin{align*}
	\cC^-(Y) 
	&= 
	\bigcup_{v \in \cL(X), \text{ bispecial}} 
	\bigcup_{u \text{ bispecial extended image of } v} \cC^-_Y(u)
	\\
	&=\bigcup_{v \in \cL(X), \text{ bispecial}} \cC^-_X(v) \\
	&= \cC^-(X),
\end{align*}
which ends the proof.
\end{proof}

The following corollary is a direct consequence of Lemmas~\ref{lemma:non-left-invariant} and~\ref{lemma:left-invariant}.

\begin{corollary}
\label{cor:empty is preserved}
Let $X$ be a dendric shift over $\cA_3$, $\sigma \in \DP(X)$ and $Y$ the image of $X$ under $\sigma$.
If $\cC^-(Y) = \emptyset$, then $\cC^-(X)=\emptyset$ and $\sigma$ is left-invariant.
Respectively, if $\cC^+(Y) = \emptyset$, then $\cC^+(X)=\emptyset$ and $\sigma$ is right-invariant. 
\end{corollary}

\begin{proposition}
\label{prop:at most one letter}
Let $X$ be a ternary minimal dendric shift.
Then $\cC^-(X)$ and $\cC^+(X)$ contain at most one letter.
Moreover, if $\bsigma=(\sigma_n)_{n \geq 1}$ is a $\Sigma_3\cSD\Sigma_3$-adic representation of $X$, then
\begin{enumerate}
\item\label{item:ALX=0}
$\cC^-(X) = \emptyset$ if and only if $\bsigma$ belongs to $(\Sigma_3\cSL\Sigma_3)^\NN$, if and only if $X$ has a unique left special factor of each length;
\item\label{item:ARX=0}
$\cC^+(X) = \emptyset$ if and only if $\bsigma$ belongs to $(\Sigma_3\cSR\Sigma_3)^\NN$, if and only if $X$ has a unique right special factor of each length.
\end{enumerate}
\end{proposition}

\begin{proof}
Using the notation of Section~\ref{subsection:s-adicity}, we have $X = X_{\bsigma}$ and for each $n \geq 1$, $X_{\bsigma}^{(n)}$ is dendric by Proposition~\ref{prop:Ster-adic representation}. 
Let $\sigma_n = \pi_n \sigma'_n \pi'_n$ with $\pi_n, \pi'_n \in \Sigma_3$ and $\sigma'_n \in \cSD$. 
The shift spaces $\pi'_n(X^{(n+1)}_\bsigma)$ and $\pi^{-1}_n(X^{(n)}_\bsigma)$ are dendric and $\pi^{-1}_n(X^{(n)}_\bsigma)$ is the image of $\pi'_n(X^{(n+1)}_\bsigma)$ under $\sigma_n'$.
Thus $\sigma'_n \in \DP(\pi'_n(X_{\bsigma}^{(n+1)}))$ by Lemma~\ref{lemma:DP implies dendric}. 
We first show item~\ref{item:ALX=0}.

Assume that $\cC^-(X) = \emptyset$. We have, by induction using Corollary~\ref{cor:empty is preserved}, that $\cC^-(\pi'_n(X^{(n)}_\bsigma)) = \emptyset$ and $\sigma'_n \in \cSL$ for all $n$ thus $\bsigma \in (\Sigma_3\cSL\Sigma_3)^\NN$.

Now assuming that $\bsigma$ belongs to $(\Sigma_3\cSL\Sigma_3)^\NN$, we deduce that any bispecial factor of $X$ is a descendant of the empty word in some $X_\bsigma^{(n)}$.
Using Figure~\ref{fig:extension graph of epsilon}, the bispecial factor $v = \varepsilon \in \cL(X_\bsigma^{(n)})$ has a unique right extension $va$, $a \in \cA_3$, which is left special and it satisfies $E_{X_\bsigma^{(n)}}^-(va) = \cA_3$.
It then suffices to observe, using Proposition~\ref{prop:def antecedent and bsp ext image}, that this property is preserved by taking bispecial extended images under some morphism $\sigma \in \Sigma_3\cSL\Sigma_3$.
This shows that any bispecial factor $u$, and hence any left special factor, of $X$ satisfies $E_X^-(u) = \cA_3$.
Proposition~\ref{prop:complexity} then implies that $X$ has a unique left special factor of each length.

Finally assume that $X$ has a unique left special factor $u_n$ of each length $n$.
By Proposition~\ref{prop:complexity}, we have $E_X^-(u_n) = \cA_3$.
Using Remark~\ref{remark:interpretations of C^- and C^+}, 
the set $\cC^-_X(u_n)$ is non-empty if and only if there are two letters $x,y \in \cA_3$ such that both $u_nx$ and $u_ny$ are left special factors of $X$.
This implies that $\cC^-_X(u_n) = \emptyset$. Any non-left-special factor $u$ is such that $\cC^-_X(u) = \emptyset$ by Remark~\ref{remark:interpretations of C^- and C^+}, hence $\cC^-(X) = \emptyset$.

The proof of item~\ref{item:ARX=0} is symmetric.

We finish the proof by showing that $\cC^-(X)$ and $\cC^+(X)$ contain at most one letter.
Assume by contrary that $a,b \in \cC^-(X)$ for some different letters $a$ and $b$.
By Lemmas~\ref{lemma:AL Ar at most one letter}, \ref{lemma:non-left-invariant} and~\ref{lemma:left-invariant}, all morphisms $\sigma'_n$ are left-invariant.
But then item~\ref{item:ALX=0} implies that $\cC^-(X) = \emptyset$.
\end{proof}

%########################
\subsection{$\cSD$-adic characterization of minimal ternary dendric shifts}
%########################

By Proposition~\ref{prop:at most one letter}, any ternary minimal dendric shift $X$ satisfies $\#\cC^-(X),\#\cC^+(X) \leq 1$.
To alleviate notations in what follows, we consider the alphabet $\A0 = \cA_3 \cup \{0\}$ and we write $\cC^-(X) = a$ instead of $\cC^-(X) = \{a\}$ and $\cC^-(X) = 0$ instead of $\cC^-(X) = \emptyset$ (and similarly for $\cC^+(X)$).
We then define the equivalence relation $\sim$ on the set of minimal ternary dendric shifts by
\[
	X \sim Y	
	\Leftrightarrow 
	(\cC^-(X),\cC^+(X)) = (\cC^-(Y),\cC^+(Y)).
\]
For all $l,r \in \A0$, we let $[l,r]$ denote the equivalence class of all minimal ternary dendric shifts satisfying $(\cC^-(X),\cC^+(X)) = (l,r)$.

\begin{lemma}
\label{lemma:equivalence relation}
Let $X$ and $Y$ be minimal ternary dendric shifts.
 We have $X \sim Y$ if and only if $\DP(X) = \DP(Y)$. 
Furthermore, if $X \sim Y$, if $\sigma \in \DP(X) \cup \Sigma_3$ and if $X'$ and $Y'$ are the respective images of $X$ and $Y$ under $\sigma$, then $X' \sim Y'$.
\end{lemma}
\begin{proof}
The equivalence between $X \sim Y$ and $\DP(X) = \DP(Y)$ follows from  Proposition~\ref{prop:condition for non-tree-successor}.
The second part of the statement follows from Lemma~\ref{lemma:non-left-invariant} and Lemma~\ref{lemma:left-invariant}.
\end{proof}

\begin{lemma}
\label{lemma:classes not empty}
For each $l, r \in \A0$, the equivalence class $[l, r]$ is non empty.
\end{lemma}
\begin{proof}
Let $X$ be an Arnoux-Rauzy shift space. By Corollary~\ref{cor:image of Arnoux-Rauzy}, the image of $X$ under any morphism $\sigma \in \Sigma_3\cSD$ is dendric. More precisely $\cC^-(X)$ and $\cC^+(X)$ are both empty. Using Lemmas~\ref{lemma::AL et AR du mot vide},~\ref{lemma:non-left-invariant} and~\ref{lemma:left-invariant}, we can then choose such a morphism $\sigma$ to $X$ to obtain an element of any equivalence class.
\end{proof}

Using the previous lemmas, we can define, for each equivalence class $C = [l,r]$, the set
\[
    \DPP(C) = \{\pi \sigma \pi' \mid \pi, \pi' \in \Sigma_3, \sigma \in \DP(\pi'(X))\},
\]
where $X \in C$. Thus it corresponds to the set of morphisms in $\Sigma_3\cSD\Sigma_3$ that are dendric preserving for $X$ ($\DPP$ stands for {\em dendric preserving with permutations}).
Furthermore, for any $\sigma \in \DPP(C)$, there is a unique equivalence class $C'$ such that when $X \in C$ and $Y$ is the image of $X$ under $\sigma$, then $Y \in C'$.
We call $C'$ the {\em image of $C$ under $\sigma$}.

For each morphism $\sigma \in \cSD$, the classes $C$ such that $\sigma \in \DPP(C)$ and their images are summarized in Table~\ref{table:edges of cG}.
They are computed using Proposition~\ref{prop:condition for non-tree-successor} (to determine the allowed classes $C$) and Table~\ref{table:AL et AR du mot vide}, Lemmas~\ref{lemma:non-left-invariant} and~\ref{lemma:left-invariant} (to determine the corresponding image).

\begin{table}[h]
\renewcommand*{\arraystretch}{1.2}
\centering
\begin{tabular}{|c|c|c|c|}
\hline
Morphism & Class $C$	& Image of $C$	 & Conditions \\
\hline
$\alpha$	&	$[l,r]$ &	$[l,r]$ & none	\\
\hline 
$\beta$	&	$[l,r]$	&	$[1,r]$	&	$l \ne 1$  	\\
\hline 
$\gamma$	&	$[l,r]$	&	$[l,1]$ &	$r \ne 1$ \\
\hline 
$\delta^{(k)}$	&	$[l,r]$	&	$[3,1]$ & $l,r \ne 1$	\\
\hline 
$\zeta^{(k)}$	&	$[l,r]$	&	$[3,3]$ & $l,r \ne 2$	\\
\hline 
$\eta$	&	$[l,r]$	&	$[2,3]$ & $l \ne 2, r \ne 1$	\\
\hline 
\end{tabular}
\caption{Images of classes under the morphisms of $\cSD$}
\label{table:edges of cG}
\end{table}

Up to permutations, we can distinguish five types of set $\DPP([l,r])$, depending
on whether $l$ or $r$ is $0$ and on whether $l = r$ or not.
We thus build the following directed graph $\cG'$, whose set of vertices is $V = \{[0,0], [0,3], [3, 0], [3, 2], [3, 3]\}$.
For each vertex, there is an incoming edge labeled by each permutation and, for every vertices $C, C' \in V$ and every morphism $\sigma \in \Sigma_3\cSD\Sigma_3$, there is an edge from $C$ to $C'$ with label $\sigma$ if $\sigma \in \DPP(C')$ and $C$ is the image of $C'$ under $\sigma$.
In other words, we have an edge labeled by $\pi \sigma' \pi'$, $\pi, \pi' \in \Sigma_3$, $\sigma' \in \cSD$, from $[l, r]$ to $[l', r']$ if the class $[\pi'(l'), \pi'(r')]$\footnote{We take as convention that, if $\pi$ is a permutation on $\cA_3$, then $\pi(0) = 0$.} is in Table~\ref{table:edges of cG} for $\sigma'$ and its image is the class $[\pi^{-1}(l), \pi^{-1}(r)]$.

This graph is a co-deterministic automaton, i.e., for every vertex $C$ and every morphism $\sigma$, there is at most one edge with label $\sigma$ reaching $C$. It gives a first $\Sigma_3\cSD\Sigma_3$-adic characterization of minimal dendric shifts.

\begin{theorem}
A shift space $X$ is a minimal dendric shift over $\cA_3$ if and only if it has a primitive $\cS$-adic representation labeling an infinite path in $\cG'$.
\end{theorem}

\begin{proof}
Assume that $X$ is a minimal ternary dendric shift.
By Proposition~\ref{prop:Ster-adic representation}, $X$ has a primitive $\Sigma_3\cSD\Sigma_3$-adic representation $\bsigma$ where for each $n$, $X_{\bsigma}^{(n)}$ is a ternary minimal dendric shift. In addition, up to changing the permutations and considering a permutation of $X$ instead of $X$ itself, we can assume that, for all $n$, the equivalence class of $X_{\bsigma}^{(n)}$ is an element of $V$.
By construction, $X_\bsigma^{(n)}$ and $X_\bsigma^{(n+1)}$ are dendric thus $\sigma_n$ is dendric preserving for $X_\bsigma^{(n+1)}$ and
$\mathcal{P} = ([\cC^-(X_{\bsigma}^{(n)}),\cC^+(X_{\bsigma}^{(n)})])_{n \geq 1}$ is a path in $\cG'$ with label $\bsigma$.

Now consider a primitive sequence $\bsigma$ labeling a path $([\cC^-_n,\cC^+_n])_{n \geq 1}$ in $\cG'$ and let us show that the shift space $X_{\bsigma}$ is minimal and dendric.
It is minimal by primitiveness of $\bsigma$.
If it is not dendric, there exists a bispecial factor $u \in \cL(X_{\bsigma})$ which is not dendric.
Using Corollary~\ref{cor:back to epsilon}, there is a unique $k \geq 1$ and a unique initial bispecial factor $v$ in $\cL(X_{\bsigma}^{(k)})$ such that $u$ is a descendant of $v$ and $\cE_X(u)$ depends only on $\cE_{X_\bsigma^{(k)}}(v)$. 
By definition of initial bispecial factors, $v$ is either the empty word or a non-prefix factor of $\sigma_k(a)$ for some letter $a$.
Moreover, by Lemma~\ref{lemma:classes not empty} there is a dendric shift space $Y$ such that $Y \in [\cC^-_{k+1},\cC^+_{k+1}]$. Thus, if $Z$ is the image of $Y$ by $\sigma_k$, $Z \in [\cC^-_k,\cC^+_k]$ and, as $\cE_{X_\bsigma^{(k)}}(v)$ is completely determined by $\sigma_k$, we have $\cE_{X_\bsigma^{(k)}}(v) = \cE_Z(v)$.
By definition of the edges of $\cG'$, the morphism $\sigma_{[1,k)}$ is dendric preserving for $v$.
Thus $u$ is a dendric bispecial factor of $X$, which is a contradiction.
\end{proof}

We now improve the previous result by considering a smaller graph. We will need the following lemma.

\begin{lemma}
\label{lemma:properties of edges}
Let $\sigma = \pi\sigma'\pi'$ with $\pi, \pi' \in \Sigma_3$, $\sigma' \in \cSD$.
Let $l,r,l',r' \in \A0$ be such that $\sigma \in \DPP([l,r])$ and the image of $[l,r]$ under $\sigma$ is $[l',r']$.
\begin{enumerate}
\item If $\sigma'$ is left-invariant, then for all $\lambda \in \A0$, we have $\sigma \in \DPP([\lambda,r])$ and the image of $[\lambda,r]$ under $\sigma$ is $[\pi\pi'(\lambda),r']$.

\item If $\sigma'$ is not left-invariant, then $l' \neq 0$, there is a unique $\lambda \in \cA_3$ such that $\sigma \notin \DPP([\lambda,r])$ and for all other $\lambda' \in \A0$, the image of $[\lambda', r]$ under $\sigma$ is $[l', r']$.

\item If $\sigma'$ is right-invariant, then for all $\rho \in \A0$, we have $\sigma \in \DPP([l,\rho])$ and the image of $[l, \rho]$ under $\sigma$ is $[l', \pi\pi'(\rho)]$.

\item If $\sigma'$ is not right-invariant, then $r' \neq 0$, there is a unique $\rho \in \cA_3$ such that $\sigma \notin \DPP([l,\rho])$ and for all other $\rho' \in \A0$, the image of $[l, \rho']$ under $\sigma$ is $[l',r']$.
\end{enumerate}
\end{lemma}

\begin{proof}
It follows from Table~\ref{table:edges of cG}
\end{proof}

We are now ready to prove Theorem~\ref{thm:main}. We say that two sequences $(\sigma_n)_{n \in \N}$ and $(\tau_n)_{n \in \N}$ are \emph{equivalent} if, for all $n$, there exist a permutation $\pi_n$ on the alphabet such that
\[
    \sigma_1 \dots \sigma_n \pi_n = \tau_1 \dots \tau_n.
\]
It is then clear that $(\sigma_n)_{n \in \N}$ is a primitive $\cS$-adic representation of a shift space $X$ if and only if $(\tau_n)_{n \in \N}$ also is.

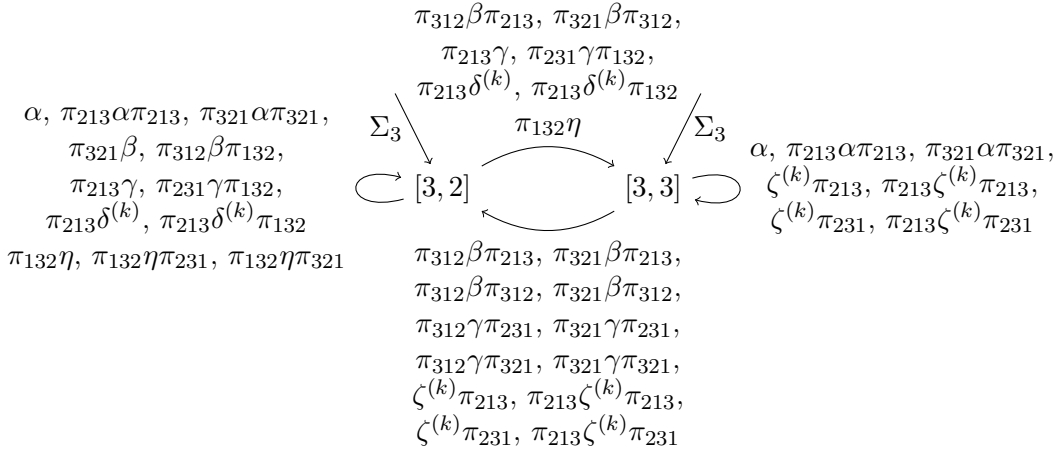
\begin{figure}[h]
\begin{center}
\begin{tikzpicture}[scale=0.7]
\node (32) at (0,0) {$[3,2]$};
\node (33) at (4,0) {$[3,3]$};
\node (fake32) at (-1,2) {};
\node (fake33) at (5,2) {};

\path (fake32) edge [->] node [pos=0.5,left] {$\Sigma_3$} (32);
\path (fake33) edge [->] node [pos=0.5,right] {$\Sigma_3$} (33);

\path (32) edge [loop left, ->] node [pos=0.5,left,align=center] {
$\alpha$, $\pi_{213}\alpha\pi_{213}$, $\pi_{321}\alpha\pi_{321}$,\\
$\pi_{321}\beta$, $\pi_{312}\beta\pi_{132}$,\\
$\pi_{213}\gamma$, $\pi_{231}\gamma\pi_{132}$,\\
$\pi_{213}\delta^{(k)}$, $\pi_{213}\delta^{(k)}\pi_{132}$\\
$\pi_{132}\eta$, $\pi_{132}\eta\pi_{231}$, $\pi_{132}\eta\pi_{321}$} (32);

\path (33) edge [loop right, ->] node [pos=0.5,right,align=center] {
$\alpha$, $\pi_{213}\alpha\pi_{213}$, $\pi_{321}\alpha\pi_{321}$,\\
$\zeta^{(k)}\pi_{213}$, $\pi_{213}\zeta^{(k)}\pi_{213}$,\\
$\zeta^{(k)}\pi_{231}$, $\pi_{213}\zeta^{(k)}\pi_{231}$} (33);

\path (32) edge [bend left, ->] node [pos=0.5,above,align=center] {
$\pi_{312}\beta\pi_{213}$, $\pi_{321}\beta\pi_{312}$,\\
$\pi_{213}\gamma$, $\pi_{231}\gamma\pi_{132}$,\\
$\pi_{213}\delta^{(k)}$, $\pi_{213}\delta^{(k)}\pi_{132}$\\
$\pi_{132}\eta$} (33);

\path (33) edge [bend left, ->] node [pos=0.5,below,align=center] {
$\pi_{312}\beta\pi_{213}$, $\pi_{321}\beta\pi_{213}$,\\
$\pi_{312}\beta\pi_{312}$, $\pi_{321}\beta\pi_{312}$,\\
$\pi_{312}\gamma\pi_{231}$, $\pi_{321}\gamma\pi_{231}$,\\
$\pi_{312}\gamma\pi_{321}$, $\pi_{321}\gamma\pi_{321}$,\\
$\zeta^{(k)}\pi_{213}$, $\pi_{213}\zeta^{(k)}\pi_{213}$,\\
$\zeta^{(k)}\pi_{231}$, $\pi_{213}\zeta^{(k)}\pi_{231}$} (32);
\end{tikzpicture}
\end{center}
\caption{A shift space on $\cA_3$ is minimal and dendric if and only if it has a primitive $\Sigma_3\cSD\Sigma_3$-adic representation labeling an infinite path in this graph denoted $\cG$.}
\label{fig:graph of graphs}
\end{figure}

\begin{proof}[Proof of Theorem~\ref{thm:main}]
Let $\cG$ be the subgraph of $\cG'$ obtained by deleting the vertices $[0,0]$, $[0, 3]$ and $[3, 0]$. By definition of the edges of $\cG'$, for each edge incoming in $[3,3]$ and labeled by $\sigma$, there is also an edge incoming in $[3,3]$ labeled by $\sigma\pi_{213}$ (for instance, $[3,2]$ is the image of $[3,3]$ under both $\pi_{132}\eta$ and $\pi_{132}\eta \pi_{213}$). 
In $\cG$, for each such pair, we only keep one of the two edges. The choice is made so as to reduce the total number of different morphisms labeling an edge in the graph.
This subgraph contains 2 vertices and is represented in Figure~\ref{fig:graph of graphs}.

We show that for each infinite path $\mathcal{P}$ in $\cG'$ labeled by $\bsigma = (\sigma_n)_{n \in \N}$, there is an equivalent path in $\cG$, i.e. a path labeled by a sequence $\boldsymbol{\tau} = (\tau_n)_{n \in \N}$ equivalent to $\bsigma$.
 
We thus consider a path $\mathcal{P} = ([l_n,r_n])_{n \geq 1}$ in $\cG'$ with label $\bsigma = (\sigma_n)_{n \geq 1}$ (hence $\sigma_n$ labels the edge from $[l_n,r_n]$ to $[l_{n+1},r_{n+1}]$).
For all $n \geq 1$, let $\sigma_n =  \pi_n \sigma'_n \pi'_n$ with $\pi_n, \pi'_n \in \Sigma_3$, $\sigma'_n \in \cSD$.

We first prove that we can delete the vertices $[0,0]$, $[0,3]$ and $[3,0]$. 
If $\mathcal{P}$ goes through one of the deleted vertices, then there exists $N$ such that $l_N$ or $r_N$ is $0$.
Assume that $N$ is the smallest integer such that $l_N = 0$.
By Corollary~\ref{cor:empty is preserved}, we deduce that for all $n \geq N$,
$\sigma'_n$ is left-invariant and $l_n = 0$.

We prove by induction that there exist a sequence $(\psi_n)_{n \geq 1}$ of permutations and a sequence $(l'_n)_{n \geq 1} \in \cA_3^\NN$ such that $\psi_1 = \id$ and for all $n \geq 1$, the morphism $\tau_n = \psi_n \sigma_n \psi_{n+1}^{-1}$ labels an edge from $[l'_n, \psi_n(r_n)]$ to $[l'_{n+1}, \psi_{n+1}(r_{n+1})]$ in $\cG'$. The sequence $(\tau_n)_{n \in \N}$ is trivially equivalent to $\bsigma$.

If $N = 1$, we take $l'_1 = 3$.
If $N > 1$, we take $\psi_n = id$ and $l'_n  = l_n$ for all $n \leq N - 1$.
Assume that we have found such $\psi_n$ and $l'_n$ for $n \geq \min\{N-1,1\}$ and let us find $\psi_{n+1}$ and $l'_{n+1}$.
We first show that we can find $l' \in \cA_3$ such that $\psi_n\sigma_n \in \DPP([l', r_{n+1}])$ and the image of $[l', r_{n+1}]$ under $\psi_n\sigma_n$ is $[l'_n, \psi_n(r_n)]$. Recall that $\psi_n\sigma_n$ is in $\DPP([0, r_{n+1}])$ and that the image of $[0, r_{n+1}]$ under $\psi_n\sigma_n$ is $[\psi_n(l_n), \psi_n(r_n)]$.

For $n = N-1$ (with $N > 1$), by Lemma~\ref{lemma:left-invariant}, $\sigma'_{N-1}$ is not left-invariant. Thus, since $l'_{N-1} = l_{N-1} = \psi_{N-1}(l_{N-1})$, we can choose such an $l' \in \cA_3$ by Lemma~\ref{lemma:properties of edges}.
If $n \geq N$, $\sigma'_n$ is left-invariant thus, by Lemma~\ref{lemma:properties of edges}, for any $l' \in \cA_3$, $\psi_n\sigma_n$ is in $\DPP([l', r_{n+1}])$ and, in particular, if $l' = (\psi_n\pi_n\pi'_n)^{-1}(l'_n)$, then the image of $[l', r_{n+1}]$ under $\psi_n\sigma_n$ is $[l'_n, \psi_n(r_n)]$.

Let $\psi_{n+1}$ be a permutation such that $[\psi_{n+1}(l'), \psi_{n+1}(r_{n+1})]$ is a vertex of $\cG'$ and let $l'_{n+1} = \psi_{n+1}(l')$. The morphism $\psi_n\sigma_n\psi_{n+1}^{-1}$ then labels an edge from $[l'_n, \psi_n(r_n)]$ to $[l'_{n+1}, \psi_{n+1}(r_{n+1})]$.

We have thus found a path $\mathcal{P}' = ([l_n',\psi_n(r_n)])_{n \geq 1}$ equivalent to $\mathcal{P}$ and that does not go through vertices of the form $[0,r]$.
Note that for all $n$, $\psi_n(r_n) = 0$ if and only if $r_n = 0$.

Starting from $\mathcal{P}$ or $\mathcal{P}'$, we similarly find another path $\mathcal{P}''$ equivalent to $\mathcal{P}$. This path $\mathcal{P}''$ does not go through the vertices $[0,0]$, $[0,3]$ or $[3,0]$.

We now show that we can delete half of the edges incoming in $[3,3]$, as explained in the beginning of the proof. It follows from the fact that if $[l_{n+1}, r_{n+1}] = [3, 3]$, then $\tau_n = \sigma_n \pi_{213}$ labels an edge from $[l_n, r_n]$ to $[3,3]$ and $\tau_{n+1} = \pi_{213} \sigma_{n+1}$ labels an edge from $[3,3]$ to $[l_{n+2}, r_{n+2}]$. Replacing $\sigma_n$ by $\tau_n$ and $\sigma_{n+1}$ by $\tau_{n+1}$ then gives an equivalent path.
This concludes the proof.
\end{proof}

We can make several observations regarding equivalent sequences in $(\Sigma_3\cSD\Sigma_3)^\N$. The first one is that, if $\bsigma = (\pi_n\sigma'_n\pi'_n)_{n \in \N}$ and $\boldsymbol{\tau} = (\psi_n\tau'_n\psi'_n)_{n \in \N}$, with $\pi_n, \pi'_n, \psi_n, \psi'_n \in \Sigma_3$ and $\sigma'_n, \tau'_n \in \cSD$, are equivalent, then $\sigma'_n = \tau'_n$ for all $n$. Indeed, let $\xi_n$ be the permutation such that $\sigma_1 \dots \sigma_{n-1} \xi_n = \tau_1 \dots \tau_{n-1}$. The morphism $\tau_1\dots\tau_{n-1}$ is injective thus $\tau_n = \xi_n^{-1}\sigma_n\xi_{n+1}$. The conclusion follows from the fact that composing a morphism of $\cSD$ with some permutations on the left and on the right will not give a different morphism of $\cSD$.

Another observation is the following result.

\begin{proposition}\label{prop:equivalent S-adic representations}
Any two $\Sigma_3\cSD\Sigma_3$-adic representations of a shift space are equivalent.
\end{proposition}
\begin{proof}
Let $\bsigma = (\pi_n\sigma'_n\pi'_n)_{n \in \N}$ and $\boldsymbol{\tau} = (\psi_n\tau'_n\psi'_n)_{n \in \N}$, with $\pi_n, \pi'_n, \psi_n, \psi'_n \in \Sigma_3$ and $\sigma'_n, \tau'_n \in \cSD$, be two $\Sigma_3\cSD\Sigma_3$-adic representations of a shift space $X$.

We prove by induction on $n \geq 0$ that there exists a permutation $\xi_n$ such that $\sigma_1\dots\sigma_n = \tau_1\dots\tau_n\xi_n$ and $\xi_n(X_\bsigma^{(n+1)}) = X_{\boldsymbol{\tau}}^{(n+1)}$.
For $n = 0$, it suffices to take $\xi_0 = \text{id}$. Assume now that we have found such $\xi_{n-1}$, $n \geq 1$.
The shift $X_{\boldsymbol{\tau}}^{(n)}$ can then be seen as the image of $\pi'_n(X_\bsigma^{(n+1)})$ under $\xi_{n-1}\pi_n\sigma'_n$ or as the image of $\psi'_n(X_{\boldsymbol{\tau}}^{(n+1)})$ under $\psi_n\tau'_n$.

The shape of the extension graph $\cE_{X_{\boldsymbol{\tau}}^{(n)}}(\varepsilon)$ and the lengths of the longest power of $1$ (resp., $2$, $3$) in $\cL(X_{\boldsymbol{\tau}}^{(n)})$ uniquely determine the morphisms $\sigma'_n = \tau'_n$.
Moreover, if $\sigma'_n \ne \alpha$, we even have $\xi_{n-1}\pi_n\sigma'_n = \psi_n\tau'_n$ thus $\sigma_1\dots\sigma_{n-1}\pi_n\sigma'_n = \tau_1\dots\tau_{n-1}\psi_n\tau'_n$. We then take $\xi_n = (\psi'_n)^{-1}\pi'_n$.
If $\sigma'_n = \alpha$, then either $\xi_{n-1}\pi_n = \psi_n$ and we proceed as above, or $\xi_{n-1}\pi_n = \psi_n\pi_{132}$. In that case, $\xi_{n-1}\pi_n\sigma'_n = \psi_n\tau'_n\pi_{132}$. Let $\psi''_n = \pi_{132}\psi'_n$. We then take $\xi_n = (\psi''_n)^{-1}\pi'_n$ and the conclusion follows as in the previous case.
\end{proof}

%########################
%########################
\section{Descriptions in $\cG$ of well-known families of minimal ternary dendric shifts}
\label{section:sub families}
%########################
%########################

The class of minimal ternary dendric shifts contains several classes of well-known families of shift spaces, namely Arnoux-Rauzy shifts, codings of regular 3-interval exchange transformations and Cassaigne shifts. 
In this section, we study the $\Sigma_3\cSD\Sigma_3$-adic representations of these particular families in the light of $\cG$.

%########################
\subsection{Arnoux-Rauzy shifts}
\label{subsection:AR}
%########################

A shift space $X \subset \cA_3^\ZZ$ is an {\em Arnoux-Rauzy shift} if it is a  minimal shift space with factor complexity $p(n)=2n+1$ that has exactly one left and one right special factor of each length.
Another equivalent definition is that $X$ admits a primitive representation $\bsigma \in \{\alpha, \pi_{213}\alpha\pi_{213}, \pi_{321}\alpha\pi_{321}\}^\N$~\cite{arnoux_rauzy}.
This equivalence is also a consequence of Proposition~\ref{prop:at most one letter} and Theorem~\ref{thm:main}.

%########################
\subsection{Interval exchange shifts}
\label{subsection:interval exchange}
%########################

Let us recall the definition of an interval exchange transformation.
We use the terminology of~\cite{Ferenczi_Zamboni:2008} with two permutations.
Let $I = [l,r)$ be a semi-interval, $\lambda = (\lambda_1,\dots,\lambda_d)$ be a positive $d$-dimensional vector such that $\|\lambda\|_1 = r-l$ and $(\pi_0,\pi_1)$ be two permutations of $\{1,2,\dots,d\}$.
The two permutations induce some partitions of $I$ by semi-intervals whose lengths are given by the vector $\lambda$ and that are ordered according to $\pi_0$ and $\pi_1$ respectively.
More precisely, we consider the partitions $\mathcal{I} = \{I_1,\dots,I_d\}$ and $\mathcal{J} = \{J_1,\dots,J_d\}$, where for each $i$, 
\begin{align*}
    I_i &= 
    \left[
    l + \sum_{\pi_0^{-1}(j)<\pi_0^{-1}(i)} \lambda_j,
    l + \sum_{\pi_0^{-1}(j)\leq\pi_0^{-1}(i)} \lambda_j \right) \quad \text{and} \\  
    J_i &= 
    \left[
    l+ \sum_{\pi_1^{-1}(j)<\pi_1^{-1}(i)} \lambda_j,
    l+ \sum_{\pi_1^{-1}(j)\leq\pi_1^{-1}(i)} \lambda_j
    \right)
\end{align*}
are semi-intervals of length  $\lambda_i$.
Setting $\mu_0 = \nu_0 = l$ and, for every $i \in \{1,\dots,d\}$, $\mu_i = l+\sum_{\pi_0^{-1}(j) \leq i} \lambda_j$ and
$\nu_i = l+\sum_{\pi_1^{-1}(j) \leq i} \lambda_j$, we have $I_{\pi_0(i)} = [\mu_{i-1},\mu_i)$ and $J_{\pi_1(i)} = [\nu_{i-1},\nu_i)$ for all $i$. 

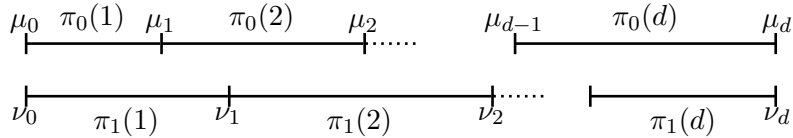
\begin{center}
    \begin{tikzpicture}
\draw[|-,line width=1pt] 
	(0,0) 
	node[above,color=black] {$\mu_0$} 
	-- 
	node[above] {$\pi_0(1)$} 
	(1.8,0);
\draw[|-,line width=1pt] 
	(1.8,0) 
	node[above,color=black] {$\mu_1$} 
	-- 
	node[above] {$\pi_0(2)$} 
	(4.5,0)
	;
\draw[|-,line width=1pt,dotted] 
	(4.5,0) 
	node[above,color=black] {$\mu_2$} 
    --
    (5.2,0)
	;
\draw[|-|,line width=1pt] 
	(6.5,0) 
	node[above,color=black] {$\mu_{d-1}$} 
	-- 
	node[above] {$\pi_0(d)$} 
    (10,0)
	node[above,color=black] {$\mu_d$} 
	;
\draw[|-,line width=1pt] 
	(0,-.7) 
	node[below,color=black] {$\nu_0$} 
	-- 
	node[below] {$\pi_1(1)$} 
	(2.7,-.7)
	node[below,color=black] {$\nu_1$} 
	;
\draw[|-,line width=1pt] 
	(2.7,-.7) 
	-- 
	node[below] {$\pi_1(2)$} 
	(6.2,-.7)
	;
\draw[|-,line width=1pt,dotted] 
	(6.2,-.7) 
	node[below,color=black] {$\nu_2$} 
    --
    (6.9,-.7)
	;
\draw[|-|,line width=1pt] 
	(7.5,-.7) 
	-- 
	node[below] {$\pi_1(d)$} 
	(10,-.7)
	node[below,color=black] {$\nu_d$} 
	;
\end{tikzpicture}
\end{center}

The {\em interval exchange transformation} (IET) on $I$ associated with $(\lambda,\pi_0,\pi_1)$ is the piecewise translation $T_{\lambda,\pi_0,\pi_1}:I \to I$ such that $T_{\lambda,\pi_0,\pi_1}(I_i) = J_i$ for every $i$.
Thus it is the bijection defined by
\[
    T_{\lambda,\pi_0,\pi_1} (x) = x - \sum_{\pi_0^{-1}(j)<\pi_0^{-1}(i)} \lambda_j + \sum_{\pi_1^{-1}(j)<\pi_1^{-1}(i)} \lambda_j, 
    \quad \text{if } x \in I_i
\]
When we want to emphasize the number of intervals, we talk about $d$-interval exchange transformations and when the context is clear, we write $T$ instead of $T_{\lambda,\pi_0,\pi_1}$.
Obviously, up to translation and rescaling, we can always assume that $I = [0,1)$.
More precisely, we define the {\em normalized} IET associated with $T$ by $\bar{T}:[0,1) \to [0,1)$ by $\bar{T}(x) = \frac{1}{r-l} (T((r-l)x+l)-l)$.
Observe that if $\tau$ is the permutation of $\{1,\dots,d\}$ defined by $\tau(i) = d+1-i$, then replacing $\pi_0$ and $\pi_1$ by $\pi_0 \circ \tau$ and $\pi_1 \circ \tau$ respectively defines another interval exchange transformation which is also conjugate to $T_{\lambda,\pi_0,\pi_1}$; we denote it by $\tilde{T}_{\lambda,\pi_0,\pi_1}$.

We let $\binom{\pi_0(1)\pi_0(2)\cdots\pi_0(d)}{\pi_1(1)\pi_1(2)\cdots\pi_1(d)}$ denote the pair of permutations $(\pi_0,\pi_1)$.

\begin{example}
\label{ex:3-IET abc}
A $3$-interval exchange transformation on $[0,1)$ with pair of permutations $\binom{123}{231}$.
It is the rotation $R:[0,1)\to [0,1),\  x \mapsto x+\lambda_2+\lambda_3 \bmod 1$.
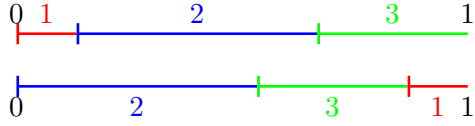
\begin{center}
\begin{tikzpicture}
\draw[|-,line width=1pt,color=red] 
	(0,0) 
	node[above,color=black] {$0$} 
	-- 
	node[above] {$1$} 
	(.8,0);
\draw[|-,line width=1pt,color=blue] 
	(.8,0) 
	-- 
	node[above] {$2$} 
	(4,0)
	;
\draw[|-,line width=1pt,color=green] 
	(4,0) 
	-- 
	node[above] {$3$} 
	(6,0)
	node[above,color=black] {$1$} 
	;
\draw[|-,line width=1pt,color=blue] 
	(0,-.7) 
	node[below,color=black] {$0$} 
	-- 
	node[below] {$2$} 
	(3.2,-.7);
\draw[|-,line width=1pt,color=green] 
	(3.2,-.7) 
	-- 
	node[below] {$3$} 
	(5.2,-.7)
	;
\draw[|-,line width=1pt,color=red] 
	(5.2,-.7) 
	-- 
	node[below] {$1$} 
	(6,-.7)
	node[below,color=black] {$1$} 
	;
\end{tikzpicture}
\end{center}
\end{example}

%%%%%%%%%%%%%%%%%%%%%%%%%%%%%%%%%
\subsubsection{Regular interval exchange transformations}

Let $T$ be an interval exchange transformation on $I$.
The \emph{orbit} of a point $x\in I$ is the set 
$\{T^n(x)\mid n\in\ZZ \}$. 
The transformation $T$ is said to be \emph{minimal} if, for any  $x\in I$, the  orbit of $x$ is dense in $I$.

A $d$-interval exchange transformation is said to be \emph{regular} if the orbits of the points $\mu_i$, $1 \leq i < d$, are infinite and disjoint.
A regular interval exchange transformation is also said to be \emph{without connections} or to satisfy the~\emph{idoc} condition (where idoc stands for {\em infinite disjoint orbit condition}).
As an example, any non-trivial $2$-interval exchange transformation is a rotation and, if normalized, it is regular if and only if the angle of the rotation is irrational.
The following result is due to Keane.

\begin{theorem}[Keane~\cite{Keane:1975}]\label{theoremKeane}
A regular interval exchange transformation is minimal.
\end{theorem}

The converse is not true. Indeed, consider the rotation of angle $\alpha$ with $0<\alpha<1/2$ irrational, as a $3$-interval exchange transformation with $\lambda=(1-2\alpha,\alpha,\alpha)$ and permutations $\binom{123}{312}$.
The transformation is minimal as any rotation of irrational angle but it is not regular since $\mu_1=1-2\alpha$, $\mu_2=1-\alpha$ and thus $\mu_2=T(\mu_1)$.

The following necessary condition for minimality of an interval exchange transformation is useful.
If an interval exchange transformation $T=T_{\lambda,\pi_0,\pi_1}$ is minimal, then the pair of permutations $(\pi_0,\pi_1)$ is {\em indecomposable}, that is, $\pi_0(\{1,\dots,k\}) \neq \pi_1(\{1,\dots,k\})$ for every $k<d$.
When $d=3$, any interval exchange transformation with indecomposable pair of permutation is conjugate to a 3-IET with pair of permutations $\binom{123}{231}$ or $\binom{132}{231}$.

\subsubsection{Codings of IET}

Let $T = T_{\lambda,\pi_0,\pi_1}$ be a $d$-interval exchange transformation and let $\cA = \{1,\dots,d\}$.
We say that a word $w=b_0b_1\cdots b_{m-1} \in \cA^*$ is {\em admissible for $T$} if the set 
\[
I_w=I_{b_0}\cap T^{-1}(I_{b_1})\cap\ldots\cap T^{-m+1}(I_{b_{m-1}})
\]
is non-empty.
The {\em language} of $T$ is the set $\cL_T$ of admissible words for $T$.
It uniquely defines the shift space 
\[
	X_T =\{x \in \cA^\ZZ \mid \cL(x) \subset \cL_T\}
\] 
that we call as the {\em natural coding} of $T$.
If $T$ is regular, then $X_T$ is minimal, hence $X_T =\{x \in \cA^\ZZ \mid \cL(x) = \cL_T\}$.
Of course, $X_T$ is invariant under normalization and reflection, i.e., we have $X_T = X_{\tilde{T}} = X_{\bar{T}}$.

\subsubsection{Derivation and induction}

The $\Sigma_3\cSD \Sigma_3$-adic representations of minimal dendric subshifts that we consider are based on derivation by return words.
In the context of codings of IET, this derivation can be understood through induced transformations and in particular Rauzy inductions~\cite{Rauzy1979}.

Let $T = T_{\lambda,\pi_0,\pi_1}$ be a $d$-IET on $I$. A semi-interval $I' \subset I$ is said to be {\em recurrent} for $T$ if for all $x \in I'$, there exists $n >0$ such that $T^n(x) \in I'$.
In that case, the {\em transformation induced by $T$ on $I'$} is the map $T_{I'}:I' \to I', \ x \mapsto T^{r_{I'}(x)}(x)$, where $r_{I'}(x) = \min\{n>0 \mid T^n(x) \in I'\}$.
The following result is classical (for a proof, see for instance~\cite{Dolce_Perrin:2017}).

\begin{proposition}\label{prop:induction and derivation}
If $T$ is a regular $d$-interval exchange transformation.
Then for every non-empty word $w \in \cL_T$, the transformation induced by $T$ on $I_w$ is a regular $d$-interval exchange transformation $T'$ and one has $X_{T'} = \cD_w(X_T)$.
\end{proposition}

Induced transformations on $I_w$ can be obtained by the (left or right) {\em Rauzy induction}.
Set $l' = \min\{\mu_1,\nu_1\}$, $r' = \max\{\mu_{d-1},\nu_{d-1}\}$, $I_L = [l',r)$ and $I_R = [l, r')$.
Observe that $I_L$ (resp., $I_R$) is recurrent and, for all $x \in I_L$, $r_{I_L}(x) \leq 2$ (resp., for all $x \in I_R$, $r_{I_R}(x) \leq 2$).
The {\em left Rauzy induction} and {\em right Rauzy induction} of $T$ are then the transformations induced by $T$ on $I_L$ and $I_R$ respectively. We denote them $L(T)$ and $R(T)$.

A $d$-IET $T = T_{\lambda,\pi_0,\pi_1}$ is said to be {\em $L$-inducible} (resp., {\em $R$-inducible}) if $\lambda_{\pi_0(1)} \neq \lambda_{\pi_1(1)}$ (resp., $\lambda_{\pi_0(d)} \neq \lambda_{\pi_1(d)}$). 
Observe that if $T$ is $L$-inducible (resp., $R$-inducible), then $L(T)$ (resp., $R(T)$) is a $d$-IET.
A word $F_1\cdots F_n$ over $\{L,R\}$ is said to be {\em valid} for $T$ if it defines a sequence $(T_0,T_1,\dots,T_n)$ of $d$-IET by $T_0 = T$ and for all $k <n$, $T_k$ is $F_{n-k}$-inducible and $T_{k+1} = F_{n-k}(T_k)$. 

\begin{lemma}\label{lemma:composition of Rauzy inductions}
Let $T$ be a $d$-IET and let $F_1\cdots F_n$ be a word over $\{L,R\}$ which is valid for $T$.
Then $T_n$ is a $d$-interval exchange transformation on a semi-interval $I' \subset I$ recurrent for $T$ and we have $T_n = T_{I'}$.
\end{lemma}
\begin{proof}
It suffices to prove it for $n = 2$. The general case will follow by induction. As $F_1 F_2$ is valid for $T$, $T_1$ and $T_2$ are $d$-interval exchange transformations on $I_1 \subset I$ and $I_2 \subset I_1$ respectively.
Let $x \in I_2$. By definition, for all $k \geq 0$, $T_1^k(x) = T^{n_k}(x)$
for a strictly increasing sequence $(n_k)_{k \geq 0}$ such that $n_0 = 0$ and, if $n_k < i < n_{k+1}$, then $T^i(x) \notin I_1$.
Let $r = \min\{k > 0 \mid T_1^k(x) \in I_2\}$. As $I_1 \subset I_2$, $n_r$ is the smallest exponent $k > 0$ such that $T^k(x) \in I_2$ thus $T_2(x) = T_1^r(x) = T^{n_r}(x) = T_{I_2}(x)$.
\end{proof}

The next result summarizes several classical results concerning Rauzy inductions~\cite{Rauzy1979,Dolce_Perrin:2017}.

% \begin{theorem}
% Let $T = T_{\lambda, \pi_0,\pi_1}$ be a $d$-interval exchange transformation.
% If $\lambda_{\pi_0(1)} \neq \lambda_{\pi_1(1)}$, then $L(T)$ is a $d$-interval exchange transformation.
% Furthermore, 
% \begin{itemize}
%     \item 
%     if $\lambda_{\pi_0(1)}<\lambda_{\pi_1(1)}$, then defining $\tilde{\pi}_0$ by 
%     \[
%     \tilde{\pi}_0: j \mapsto 
%     \begin{cases}
%         \pi_0(j+1),    & \text{if } j < \pi_0^{-1}(\pi_1(1))-1 ; \\
%         \pi_0(1),    & \text{if } j = \pi_0^{-1}(\pi_1(1))-1 ; \\
%         \pi_0(j),    & \text{if } j \geq \pi_0^{-1}(\pi_1(1));
%     \end{cases}
%     \]
%     $\lambda' = (\lambda'_1,\dots,\lambda'_d)$ by 
%     $\lambda'_{\pi_1(1)} = \lambda_{\pi_1(1)} - \lambda_{\pi_0(1)}$ 
%     and 
%     $\lambda'_{\pi_1(j)} = \lambda_{\pi_1(j)}$ for $j >1$,
%     $\tilde{\lambda}$ by the positive probability vector proportional to $\lambda'$
%     and finally $\sigma$ by the morphism
%     \[
%         \sigma:
%         \begin{cases}
%             \pi_0(1) \mapsto \pi_1(1)\pi_0(1)  \\
%             j \mapsto j, & j \neq \pi_0(1),
%         \end{cases}
%     \]
%     we obtain that $L(T) = T_{\tilde{\lambda},\tilde{\pi}_0, \pi_1}$ and that $X_T$ is the image of $X_{L(T)}$ under $\sigma$. 

% \end{itemize}

% If $\lambda_{\pi_0(d)} \neq \lambda_{\pi_1(d)}$, then $R(T)$ is a $d$-interval exchange transformation.
% Furthermore,
% \end{theorem}

\begin{theorem}\label{thm:effect of Rauzy induction ternary}
Let $T = T_{\lambda, \pi_0,\pi_1}$ be a $3$-interval exchange transformation, where $(\pi_0,\pi_1) \in \left\{ \binom{123}{231}, \binom{132}{231}, \binom{321}{132} \right\}$.

If $\lambda$ satisfies the condition $c$ and there is an edge from $(\pi_0,\pi_1)$ to $(\tilde{\pi}_0,\tilde{\pi}_1)$ labeled by $(c, (\sigma(1), \sigma(2), \sigma(3))$ in the following graph, then $L(T) = T_{M_\sigma^{-1}\lambda,\tilde{\pi}_0,\tilde{\pi}_1}$ and $X_T$ is the image of $X_{L(T)}$ under $\sigma$.

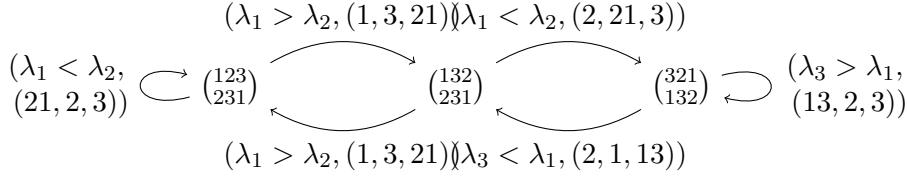
\begin{tikzpicture}
\node (123) at (0,0) {$\binom{123}{231}$};
\node (132) at (3,0) {$\binom{132}{231}$};
\node (321) at (6,0) {$\binom{321}{132}$};

\path (123) edge [loop left, ->] node [pos=0.5, left, align=center] {$(\lambda_1 < \lambda_2,$\\$\ (21,2,3))$} (123);
\path (123) edge [bend left, ->] node [pos=0.5, above] {$(\lambda_1 > \lambda_2, (1,3,21))$} (132);

\path (132) edge [bend left, ->] node [pos=0.5, above] {$(\lambda_1 < \lambda_2, (2,21,3))$} (321);
\path (132) edge [bend left, ->] node [pos=0.5, below] {$(\lambda_1 > \lambda_2, (1,3,21))$} (123);

\path (321) edge [bend left, ->] node [pos=0.5, below] {$(\lambda_3 < \lambda_1, (2,1,13))$} (132);
\path (321) edge [loop right, ->] node [pos=0.5, right, align=center] {$(\lambda_3 > \lambda_1,$\\ $\ (13,2,3))$} (321);
\end{tikzpicture}

If $\lambda$ satisfies the condition $c$ and there is an edge from $(\pi_0,\pi_1)$ to $(\tilde{\pi}_0,\tilde{\pi}_1)$ labeled by $(c, (\sigma(1), \sigma(2), \sigma(3))$ in the following graph, then $R(T) = T_{M_\sigma^{-1}\lambda,\tilde{\pi}_0,\tilde{\pi}_1}$ and $X_T$ is the image of $X_{R(T)}$ under $\sigma$.

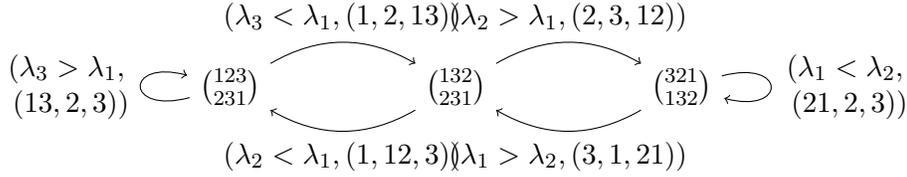
\begin{tikzpicture}
\node (123) at (0,0) {$\binom{123}{231}$};
\node (132) at (3,0) {$\binom{132}{231}$};
\node (321) at (6,0) {$\binom{321}{132}$};

\path (123) edge [loop left, ->] node [pos=0.5, left, align=center] {$(\lambda_3 > \lambda_1,$\\$\ (13,2,3))$} (123);
\path (123) edge [bend left, ->] node [pos=0.5, above] {$(\lambda_3 < \lambda_1, (1,2,13))$} (132);

\path (132) edge [bend left, ->] node [pos=0.5, above] {$(\lambda_2 > \lambda_1, (2,3,12))$} (321);
\path (132) edge [bend left, ->] node [pos=0.5, below] {$(\lambda_2 < \lambda_1, (1,12,3))$} (123);

\path (321) edge [bend left, ->] node [pos=0.5, below] {$(\lambda_1 > \lambda_2, (3,1,21))$} (132);
\path (321) edge [loop right, ->] node [pos=0.5, right, align=center] {$(\lambda_1 < \lambda_2,$\\ $\ (21,2,3))$} (321);
\end{tikzpicture}

\end{theorem}

%%%%%%%%%%%%%%
\subsubsection{$\cS$-adic representations of regular $3$-interval exchange transformations}

Let $\cX{2}$ denote the set of codings of interval exchange transformations associated with the permutations $\binom{123}{231}$, or equivalently, with the permutations $\binom{321}{132}$, and $\cX{3}$ the set of codings of interval exchanges with permutations $\binom{132}{231}$.
We study the link between the two classes using inductions.

\begin{proposition}\label{prop:inductions for [3,2]}
Let $X \in \cX{2}$ and let $T$ be the corresponding interval exchange transformation with permutations $\binom{123}{231}$ and length vector $\lambda$.
We have the following cases.
\begin{itemize}
\item
    If $\lambda_1 > \lambda_2 + \lambda_3$, then $I_1$ is recurrent and, if $T' = T_{I_1}$, $X_{T'} \in \cX{2}$ and $X$ is the image of $X_{T'}$
    under $\alpha$.
\item
    If $\lambda_2, \lambda_3 < \lambda_1 < \lambda_2 + \lambda_3$, then $I_1$ is recurrent and, if $T' = T_{I_1}$, $X_{T'} \in \cX{2}$ and $X$ is the image of $X_{T'}$
    under $\pi_{132}\eta\pi_{321}$.
\item
    If $\lambda_2 < \lambda_1 < \lambda_3$, then $I_3$ is recurrent and, if $T' = T_{I_3}$, $X_{T'} \in \cX{3}$ and $X$ is the image of $X_{T'}$
    under $\pi_{312}\beta\pi_{213}$.
\item
    If $\lambda_3 < \lambda_1 < \lambda_2$, then $I_2$ is recurrent and, if $T' = T_{I_2}$, $X_{T'} \in \cX{2}$ and $X$ is the image of $X_{T'}$ under $\pi_{213}\gamma$.
\item
    If $\lambda_1 < \lambda_2, \lambda_3$ and $k \lambda_1 < \lambda_3 < (k+1) \lambda_1$ for some integer $k\geq1$, then $I_2$ is recurrent and, if $T' = T_{I_2}$, $X_{T'} \in \cX{3}$ and $X$ is the image of $X_{T'}$ under $\pi_{213}\delta^{(k)}$.
\end{itemize}
Moreover, the vector of interval lengths in $T'$ is equal to $M_\sigma^{-1} \lambda$ where $\sigma$ is the morphism given by the cases above.
\end{proposition}
\begin{proof}
It follows from Lemma~\ref{lemma:composition of Rauzy inductions} by applying Theorem~\ref{thm:effect of Rauzy induction ternary} several times until we obtain the induction on the given interval. Let us detail the steps for the third case.

Assume that $\lambda_2 < \lambda_1 < \lambda_3$ and assume that $T$ is normalized.
We start by applying a left Rauzy induction and, as $\lambda_1 > \lambda_2$, by Theorem~\ref{thm:effect of Rauzy induction ternary}, $T_1 := L(T)$ is the interval exchange transformation on $[\lambda_2, 1[$ with permutations $\binom{132}{231}$ and length vector $(\lambda_1 - \lambda_2, \lambda_3, \lambda_2)$. Moreover, $X$ is the image of $X_{T_1}$ under the morphism $\sigma_1 : 1 \mapsto 1$, $2 \mapsto 3$, $3 \mapsto 21$.
We then apply another left Rauzy induction on $T_1$. As $\lambda_1 - \lambda_2 < \lambda_1 < \lambda_3$, this gives the transformation $T_2$ on $[\lambda_1, 1[$ with permutations $\binom{321}{132}$ and interval lengths $(\lambda_2 + \lambda_3 - \lambda_1, \lambda_1 - \lambda_2, \lambda_2)$. The shift $X_{T_1}$ is the image of $X_{T_2}$ under the morphism $\sigma_2 : 1 \mapsto 2$, $2 \mapsto 21$, $3 \mapsto 3$.
We apply a third left Rauzy induction on $T_2$ to obtain the transformation $T_3$ on $[\lambda_1 + \lambda_2, 1[ = I_3$. Since $\lambda_2 < \lambda_2 + \lambda_3 - \lambda_1$, the associated permutations are $\binom{132}{231}$. Moreover, $X_{T_2}$ is the image of $X_{T_3}$ under $\sigma_3 : 1 \mapsto 2$, $2 \mapsto 1$, $3 \mapsto 13$. 
By construction, $X_{T'} = X_{T_3}$ is in $\cX{3}$, the lengths of the intervals in $T'$ are given by $M_{\sigma_1\sigma_2\sigma_3}^{-1}\lambda$ and $X$ is the image of $X_{T'}$ under the morphism $\sigma_1 \sigma_2 \sigma_3 = \pi_{312}\beta\pi_{213}$.

For the other cases, we simply give the steps and leave the details to the reader.
\begin{itemize}
    \item
        If $\lambda_1 > \lambda_2 + \lambda_3$, then we do two consecutive right Rauzy inductions.
    \item
        If $\lambda_2, \lambda_3 < \lambda_1 < \lambda_2 + \lambda_3$, we do three right Rauzy inductions.
    \item
        If $\lambda_3 < \lambda_1 < \lambda_2$, we apply a left and a right Rauzy induction then swap the letters $1$ and $2$.
    \item
        If $\lambda_1 < \lambda_2, \lambda_3$ and $k \lambda_1 < \lambda_3 < (k+1) \lambda_1$, we apply a left Rauzy induction and $k + 1$ right Rauzy inductions then swap the letters $1$ and $2$.
\end{itemize}
\end{proof}

The proof of the next result is similar.

\begin{proposition}\label{prop:inductions for [3,3]}
Let $X \in \cX{3}$ and let $T$ be the corresponding interval exchange transformation with permutations $\binom{132}{231}$ and length vector $\lambda$.
We have the following cases.
\begin{itemize}
\item
    If $\lambda_1 > \lambda_2 + \lambda_3$, then $I_1$ is recurrent and, if $T' = T_{I_1}$, $X_{T'} \in \cX{3}$ and $X$ is the image of $X_{T'}$ under $\alpha$.
\item
    If $\lambda_2 > \lambda_1 + \lambda_3$, then $I_2$ is recurrent and, if $T' = T_{I_2}$, $X_{T'} \in \cX{3}$ and $X$ is the image of $X_{T'}$ under $\pi_{213}\alpha\pi_{213}$.
\item
    If $\lambda_2 < \lambda_1 < \lambda_2 + \lambda_3$ and $k( \lambda_1 - \lambda_2) < \lambda_3<(k+1)( \lambda_1 - \lambda_2)$ for some integer $k \geq 1$, then $I_1$ is recurrent and, if $T' = T_{I_1}$, $X_{T'} \in \cX{3}$ and $X$ is the image of $X_{T'}$ under $\zeta^{(k)}\pi_{213}$.
\item
    If $\lambda_1 < \lambda_2 < \lambda_1 + \lambda_3$ and 
    $k( \lambda_2 - \lambda_1) < \lambda_3< (k+1)( \lambda_2 - \lambda_1)$ for some integer $k \geq 1$, then $I_2$ is recurrent and, if $T' = T_{I_2}$, $X_{T'} \in \cX{3}$ and $X$ is the image of $X_{T'}$ under $\pi_{213}\zeta^{(k)}\pi_{213}$.
\end{itemize}
Moreover, the vector of interval lengths in $T'$ is equal to $M_\sigma^{-1} \lambda$ where $\sigma$ is the morphism given by the cases above.
\end{proposition}

\begin{remark}
If $X \in \cX{2} \cup \cX{3}$ is the coding of a regular interval exchange transformation, then it satisfies one of the cases of Proposition~\ref{prop:inductions for [3,2]} or Proposition~\ref{prop:inductions for [3,3]} and, by Proposition~\ref{prop:induction and derivation}, $X_{T'}$ is also the coding of a regular interval exchange.
\end{remark}

Using Propositions~\ref{prop:inductions for [3,2]} and~\ref{prop:inductions for [3,3]}, we can define the graph $\cG_\IET$ represented in Figure~\ref{fig:graph for interval exchanges}.
It has two vertices $[3,2]$ and $[3,3]$ and, for each of the cases of Proposition~\ref{prop:inductions for [3,2]} (resp., of Proposition~\ref{prop:inductions for [3,3]}), it has an edge leaving $[3,2]$ (resp., $[3,3]$) going to the vertex $\cC$ such that $X_{T'} \in \mathcal{X}_{\cC}$ and labeled by the morphism $\sigma$ such that $X$ is the image of $X_{T'}$ under $\sigma$.
Moreover, we add incoming edges labeled by each of the permutations.
Remark that it is a subgraph of the graph $\cG$ represented in Figure~\ref{fig:graph of graphs} used to characterize ternary dendric shifts.

\begin{figure}[h]
\begin{center}
\begin{tikzpicture}
\node (32) at (0,0) {$[3,2]$};
\node (33) at (4,0) {$[3,3]$};
\node (fake32) at (-0.7,1.5) {};
\node (fake33) at (4.7,1.5) {};

\path (fake32) edge [->] node [pos=0.5,left] {$\Sigma_3$} (32);
\path (fake33) edge [->] node [pos=0.5,right] {$\Sigma_3$} (33);

\path (32) edge [loop left, ->] node [pos=0.5,left,align=center] {
$\alpha$, $\pi_{132}\eta\pi_{321}$} (32);

\path (33) edge [loop right, ->] node [pos=0.5,right,align=center] {
$\alpha$, $\pi_{213}\alpha\pi_{213}$,\\
$\zeta^{(k)}\pi_{213}$, $\pi_{213}\zeta^{(k)}\pi_{213}$} (33);

\path (32) edge [->] node [pos=0.5,above,align=center] {
$\pi_{312}\beta\pi_{213}$,\\
$\pi_{213}\gamma$, $\pi_{213}\delta^{(k)}$} (33);
\end{tikzpicture}
\end{center}
\caption{A shift space on $\cA_3$ is the coding of a regular interval exchange if and only if it has a primitive $\Sigma_3\cSD\Sigma_3$-adic representation labeling an infinite path in this graph denoted $\cG_{\IET}$.}
\label{fig:graph for interval exchanges}
\end{figure}
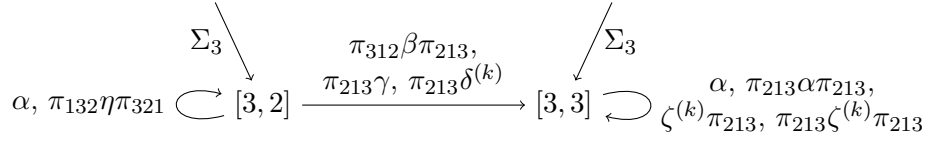

Before using this graph to give a $\cS$-adic characterization of regular interval exchanges, we recall the notion of letter frequency.

A shift space $X$ over $\cA$ is said to have {\em letter frequencies} if for all $a \in \cA$, there exists $\lambda_a \in [0,1]$ such that for all $x \in X$, $\lim_n \frac{|x_1 x_2 \cdots x_n|_a}{n} = \lambda_a$. 
A classical result of Boshernitzan~\cite{boshernitzan:1984} shows that minimal dendric shift spaces over $\cA_3$ (hence codings of regular 3-IET) are uniquely ergodic and thus have letter frequencies that are given by the measure of the letter cylinders.
In particular, when $X$ is the coding of a regular 3-IET, then the letter frequencies are given by the lengths of the intervals.
The following result is a direct consequence of~\cite[Theorem 5.7]{berthe_delecroix}.

\begin{proposition}\label{prop:frequency using s-adic representation}
Let $X$ be a uniquely ergodic shift space with primitive $\cS$-adic representation $\bsigma$. For all $n$, if $f_n$ is the vector of letter frequencies in $X_\bsigma^{(n)}$, then $f_1$ is proportional to $M_{\sigma_{[1,n)}} f_n$.
\end{proposition}

\begin{lemma}\label{lemma:unique morphism with frequencies from [3,2]}
Let $X$ be a minimal ternary dendric shift space, $\bsigma = (\sigma_n)_{n \in \N}$ a primitive $\Sigma_3\cSD\Sigma_3$-adic representation of $X$ and $\lambda_1$, $\lambda_2$, $\lambda_3$ the letter frequencies in $X$.
If $\sigma_1 \in \{\alpha, \pi_{132}\eta\pi_{321}, \pi_{312}\beta\pi_{213}, \pi_{213}\gamma\} \cup \{\pi_{213}\delta^{(k)} \mid k \geq 1\}$, then $\sigma_1$ is determined by $\lambda_1$, $\lambda_2$ and $\lambda_3$ in the following way:
\begin{itemize}
\item
    $\lambda_1 > \lambda_2 + \lambda_3$ if and only if $\sigma_1 = \alpha$;
\item
    $\lambda_2, \lambda_3 < \lambda_1 < \lambda_2 + \lambda_3$ if and only if $\sigma_1 = \pi_{132}\eta\pi_{321}$;
\item
    $\lambda_2 < \lambda_1 < \lambda_3$ if and only if $\sigma_1 = \pi_{312}\beta\pi_{213}$;
\item
    $\lambda_3 < \lambda_1 < \lambda_2$ if and only if $\sigma_1 = \pi_{213}\gamma$;
\item
    $\lambda_1 < \lambda_2, \lambda_3$ and $k \lambda_1 < \lambda_3 < (k+1) \lambda_1$ if and only if $\sigma_1 = \pi_{213}\delta^{(k)}$.
\end{itemize}
\end{lemma}
\begin{proof}
Let $\mu_1$, $\mu_2$, $\mu_3$ be the letter frequencies in $X_\bsigma^{(2)}$. The vector $(\lambda_1 \ \lambda_2 \ \lambda_3)^t$ is proportional to $M_{\sigma_1} (\mu_1 \ \mu_2 \ \mu_3)^t$ and, by minimality, the letter frequencies in $X_\bsigma^{(2)}$ are positive. Thus, given $\sigma_1$, the inequalities follow. As $\lambda_1$, $\lambda_2$, $\lambda_3$ can satisfy at most one (and exactly one) set of inequalities, we have the equivalences.
\end{proof}

\begin{lemma}\label{lemma:unique morphism with frequencies from [3,3]}
Let $X$ be a minimal ternary dendric shift space, $\bsigma = (\sigma_n)_{n \in \N}$ a primitive $\Sigma_3\cSD\Sigma_3$-adic representation of $X$ and $\lambda_1$, $\lambda_2$, $\lambda_3$ the letter frequencies in $X$.
If $\sigma_1 \in \{\alpha, \pi_{213}\alpha\pi_{213}\} \cup \{\zeta^{(k)}\pi_{213}, \pi_{213}\zeta^{(k)}\pi_{213} \mid k \geq 1\}$, then $\sigma_1$ is determined by $\lambda_1$, $\lambda_2$ and $\lambda_3$ in the following way:
\begin{itemize}
\item
    $\lambda_1 > \lambda_2 + \lambda_3$ if and only if $\sigma_1 = \alpha$;
\item
    $\lambda_2 > \lambda_1 + \lambda_3$ if and only if $\sigma_1 = \pi_{213}\alpha\pi_{213}$;
\item
    $\lambda_2 < \lambda_1 < \lambda_2 + \lambda_3$ and $k(\lambda_1 - \lambda_2) < \lambda_3 < (k+1)(\lambda_1 - \lambda_2)$ if and only if $\sigma_1 = \zeta^{(k)}\pi_{213}$;
\item
    $\lambda_1 < \lambda_2 < \lambda_1 + \lambda_3$ and $k(\lambda_2 - \lambda_1) < \lambda_3 < (k+1)(\lambda_2 - \lambda_1)$ if and only if $\sigma_1 = \pi_{213}\zeta^{(k)}\pi_{213}$.
\end{itemize}
\end{lemma}

We will also need the following result. 
The proof can, for example, be found in~\cite{Delecroix:2015}.

\begin{proposition}\label{prop:dendric interval exchanges are regular}
Let $X$ be the coding of a $d$-interval exchange transformation. If $p_X(n) = (d - 1)n + 1$ then the interval exchange transformation is regular.
\end{proposition}

\begin{theorem}\label{thm:representation interval exchange}
A shift space $X$ is the coding of a regular 3-interval exchange transformation if and only if it has a primitive $\Sigma_3 \cSD \Sigma_3$-adic representation labeling a path in the graph $\cG_\IET$ represented in Figure~\ref{fig:graph for interval exchanges}.
\end{theorem}
\begin{proof}
First assume that $X$ is the coding of a regular 3-interval exchange transformation. 
Possibly starting the representation with a permutation, we can assume that $X \in \cX{2} \cup \cX{3}$.
The sequence obtained by iterating Propositions~\ref{prop:inductions for [3,2]} and~\ref{prop:inductions for [3,3]} gives a $\Sigma_3\cSD\Sigma_3$-adic representation $\bsigma$ of $X$ and, if $X_\bsigma^{(n)} \in \cC_n$, then $\bsigma$ labels the path $(\mathcal{X}_{\cC_n})_{n \in \N}$ by construction of the graph.
Using Theorem~\ref{thm:S-adic representation of minimal} and Proposition~\ref{prop:induction and derivation}, the sequence is primitive.

It remains to prove that if $X$ has a primitive $\Sigma_3 \cSD \Sigma_3$-adic representation $\bsigma = (\sigma_n)_{n \in \N}$ labeling a path $(\cC_n)_{n \in \N}$ in the $\cG_\IET$, then $X$ is the coding of a regular 3-interval exchange transformation. We can assume that the initial permutation is the identity.
As $\cG_\IET$ is a subgraph of the graph $\cG$, $X$ is minimal dendric by Theorem~\ref{thm:main} and the letter frequencies exist by~\cite{boshernitzan:1984}. Similarly, for all $n \in \N$, the letter frequencies exist in $X_\bsigma^{(n)}$. 
Let $\lambda^{(n)}$ denote the vector of frequencies in $X_\bsigma^{(n)}$.
By Proposition~\ref{prop:frequency using s-adic representation}, $\lambda^{(n+1)}$ is proportional to $M_{\sigma_n}^{-1} \lambda^{(n)}$.

For every $n$, let $Y_n \in X_{\cC_n}$ be the coding of the interval exchange transformation whose vector of interval lengths is $\lambda^{(n)}$.
We prove that $Y_n$ is the image of $Y_{n+1}$ under $\sigma_n$.
This will prove that $\bsigma$ is a primitive $\Sigma_3 \cSD \Sigma_3$-adic representation of $Y_1$, hence that $X = Y_1$.
The fact that the underlying IET is regular then follows from Proposition~\ref{prop:dendric interval exchanges are regular}.

If $\cC_n = [3,2]$, then $\sigma_n \in \{\alpha, \pi_{132}\eta\pi_{321}, \pi_{312}\beta\pi_{213}, \pi_{213}\gamma\} \cup \{\pi_{213}\delta^{(k)} \mid k \geq 1\}$.
By Lemma~\ref{lemma:unique morphism with frequencies from [3,2]}, the vector $\lambda^{(n)}$ satisfies some inequalities that make $Y_n$ fall into one of the cases of Proposition~\ref{prop:inductions for [3,2]}.
By checking the different cases, we deduce from this result that $Y_n$ is indeed the image of $Y_{n+1}$ under $\sigma_n$.
The proof is similar if $\cC_n = [3,3]$, using Lemma~\ref{lemma:unique morphism with frequencies from [3,3]} and Proposition~\ref{prop:inductions for [3,3]} instead.
\end{proof}

%########################
\subsection{Cassaigne shifts}
%########################

A shift space $X$ over $\cA_3$ is a {\em Cassaigne shift} if it has a primitive $\mathfrak{C}$-adic representation, where $\mathfrak{C} = \{c_1,c_2\}$ and
\[
c_{1}:
\begin{cases}
1 \mapsto 1 	\\
2 \mapsto 13	\\
3 \mapsto 2
\end{cases}
\qquad\text{and}\qquad
c_{2}:
\begin{cases}
1 \mapsto 2\\
2 \mapsto 13\\
3 \mapsto 3
\end{cases}.
\]
Cassaigne shifts are minimal ternary dendric shifts and a directive sequence $(\sigma_n)_{n \geq 1} \in \mathfrak{C}^\N$ is primitive if and only if it cannot be eventually factorized over $\{c_1^2,c_2^2\}$, i.e., there is no $N \in \N$ such that for all $n \geq N$, $c_{N+2n}=c_{N+2n+1}$~\cite{Cassaigne_Labbe_Leroy_WORDS}. 

By considering products of morphisms, we obtain that a shift space is a Cassaigne shift if and only if it has a primitive $\mathfrak{C}'$-adic representation where 
$\mathfrak{C}' = \{c_{11},c_{22},c_{122},c_{211},c_{121},c_{212}\}$ and
\[
\begin{array}{lll}
c_{11} = c_1^2:
\begin{cases}
1 \mapsto 1		\\
2 \mapsto 12 	\\
3 \mapsto 13
\end{cases};
&
c_{122} = c_1 c_2^2:
\begin{cases}
1 \mapsto 12	\\
2 \mapsto 132 	\\
3 \mapsto 2
\end{cases};
&
c_{121} = c_1c_2c_1:
\begin{cases}
1 \mapsto 13		\\
2 \mapsto 132 	\\
3 \mapsto 12
\end{cases};
\\ 
\\
c_{22} = c_2^2:
\begin{cases}
1 \mapsto 13	\\
2 \mapsto 23 	\\
3 \mapsto 3
\end{cases};
&
c_{211} = c_2 c_1^2:
\begin{cases}
1 \mapsto 2	\\
2 \mapsto 213 	\\
3 \mapsto 23
\end{cases};
&
c_{212} = c_2c_1c_2:
\begin{cases}
1 \mapsto 23		\\
2 \mapsto 213 	\\
3 \mapsto 13
\end{cases}.
\end{array}
\]
Thus a shift space is a Cassaigne shift if and only if it has a primitive $\cS$-adic representation using morphisms from the set 
\[
	\{\alpha,
	\pi_{321}\overline{\alpha}\pi_{321},
	\pi_{213}\overline{\gamma}\pi_{231},
	\pi_{231}\beta\pi_{132},
	\pi_{132}\eta\pi_{231},
	\pi_{321}\overline{\eta}\pi_{231}	
	\}
\]
or, equivalently, if and only if it has a primitive $\cS_C$-adic representation where
\[
	\cS_C = \{\alpha,
	\pi_{321}\alpha\pi_{321},
	\pi_{213}\gamma\pi_{231},
	\pi_{231}\beta\pi_{132},
	\pi_{132}\eta\pi_{231},
	\pi_{321}\eta\pi_{231}	
	\}.
\]

\begin{proposition}
\label{prop:cassaigne not iet}
There is no Cassaigne shift which is an Arnoux-Rauzy or the coding of a regular interval exchange.
\end{proposition}
\begin{proof}
Let $X$ be a Cassaigne shift and $\bsigma = (\sigma_n)_{n \geq 1}$ be a primitive $\cS_C$-adic representation of $X$. 
Firstly, it is clear that $X$ is not an Arnoux-Rauzy shift. 
Using Proposition~\ref{prop:at most one letter}, it would indeed require $\bsigma$ to only use left-invariant and right-invariant morphisms from $\cSD$, hence to be in $\{\alpha,\pi_{321}\alpha\pi_{321}\}^\NN$.
It then suffices to observe that there is no primitive sequence in $\{\alpha,\pi_{321}\alpha\pi_{321}\}^\NN$.

By Proposition~\ref{prop:equivalent S-adic representations} and Theorem~\ref{thm:representation interval exchange}, if $X$ is the coding of a regular interval exchange, then $\bsigma$ is equivalent to a primitive sequence $\boldsymbol{\tau} \in (\Sigma_3\cSD\Sigma_3)^\N$ (preceded by a permutation $\psi_0$) labeling an infinite path in the graph $\cG_\IET$.
Let $\sigma_n = \pi_n \sigma'_n \pi'_n$ and $\tau_n = \psi_n \tau'_n \psi'_n$ with $\pi_n, \pi'_n, \psi_n, \psi'_n \in \Sigma_3$ and $\sigma'_n, \tau'_n \in \cSD$.

As $\bsigma$ and $\boldsymbol{\tau}$ are equivalent, $\sigma'_n = \tau'_n$ thus, $\tau_n \notin \{\pi_{213}\delta^{(k)}, \zeta^{(k)}\pi_{213}, \pi_{213}\zeta^{(k)}\pi_{213} \mid k \geq 1\}$.
Since $\boldsymbol{\tau}$ is primitive, it cannot belong to $(\Sigma_3\cSD\Sigma_3)^*\{\alpha,\pi_{213}\alpha\pi_{213}\}^\N$ thus the path labeled by $\boldsymbol{\tau}$ stays in the vertex $[3,2]$ and $\boldsymbol{\tau} \in \{\alpha, \pi_{132}\eta\pi_{321}\}^\N$. Moreover, there exists $N$ such that $\tau_N = \pi_{132}\eta\pi_{321}$.

Let us show that we obtain a contradiction. 
Let $\xi_n$ be the permutation such that $\psi_0\tau_1\dots\tau_{n-1} = \sigma_1\dots \sigma_{n-1}\xi_n$. Since $\psi_0\tau_1\dots\tau_{n-1}$ is injective, $\tau_n = \xi_n^{-1}\sigma_n\xi_{n+1}$ for all $n \geq 1$.
For $n = N$, this equality becomes
\[
    \pi_{132}\eta\pi_{321} = \xi_N^{-1}\pi_N\eta\pi_{231}\xi_{N+1}
\]
thus, $\pi_{321} = \pi_{231}\xi_{N+1}$ and $\xi_{N+1} = \pi_{213}$.
If $\tau_{N+1} = \alpha$, then we have
\[
    \alpha = \pi_{213}\pi_{N+1}\alpha\pi_{N+1}\xi_{N+2}.
\]
This is impossible as $\pi_{N+1}$ is either the identity or $\pi_{321}$. If $\tau_{N+1} = \pi_{132}\eta\pi_{321}$, then
\[
    \pi_{132}\eta\pi_{321} = \pi_{213}\pi_{N+1}\eta\pi_{231}\xi_{N+2}
\]
which is also impossible since $\pi_{N+1} \in \{\pi_{132}, \pi_{321}\}$.
\end{proof}

\section{Further work}

The problem of finding an $\cS$-adic characterization of minimal dendric shift over larger alphabets is still open.
It is likely that, for any fixed alphabet $\cA_k = \{1,2,\dots,k\}$, there exists a graph $\cG_k$ that allows to extend Theorem~\ref{thm:main}, but its definition is more tricky.
We will attack this problem in a future work.

In particular, for Lebesgue almost every 3-dimensional probability vector $\lambda$, there is a Cassaigne shift $X$ with vector of letter frequencies $\lambda$ and with the additional property of being finitely balanced~\cite{CassaigneLabbeLeroy:arxiv}, that is, there is a constant $K$ such that for every $a \in \cA_3$ and every $n \in \NN$, $\sup_{u,v \in \cL_n(X)} |u|_a-|v|_a \leq K$.
It is an open problem to generalize Cassaigne shifts over larger alphabets and it seems reasonable to look for such a generalization among minimal dendric shifts.
Better understanding the $\cS$-adic representations of minimal dendric shifts could thus be helpful.

Another interesting question would be to study the $\cS$-adic representations obtained by factorizing the morphisms of $\cSD$ into elementary automorphisms.
Such factorizations always exist (see Theorem~\ref{thm:decoding dendric}) and therefore yield to another graph where the edges are labeled by elementary automorphisms of $F_{\cA_3}$.

\section*{Acknowledgement}

We thank the anonymous referee that helped us improve the graphs $\cG$ and $\cG_{\IET}$.

\bibliographystyle{alpha}
\bibliography{ternarycase.bib}

\end{document}